\documentclass[pdflatex,sn-mathphys-num]{sn-jnl}%

\usepackage{graphicx}%
\usepackage{multirow}%
\usepackage{amsmath,amssymb,amsfonts}%
\usepackage{amsthm}%
\usepackage{mathrsfs}%
\usepackage[title]{appendix}%
\usepackage{xcolor}%
\usepackage{textcomp}%
\usepackage{manyfoot}%
\usepackage{booktabs}%
\usepackage{algorithm}%
\usepackage{algorithmicx}%
\usepackage{algpseudocode}%
\usepackage{listings}%

\theoremstyle{thmstyleone}%
\newtheorem{theorem}{Theorem}%
\newtheorem{proposition}[theorem]{Proposition}%

\theoremstyle{thmstyletwo}%
\newtheorem{remark}{Remark}%

\theoremstyle{thmstylethree}%
\newtheorem{definition}{Definition}%

\usepackage{mathtools}
\usepackage[mathscr]{euscript}
\usepackage{enumerate}
\usepackage{framed}
\usepackage{subcaption}
\usepackage{IEEEtrantools}
\usepackage{cases}
\usepackage{eqnarray}
\usepackage{empheq}
\usepackage{bm}

\usepackage{pgf,tikz}
\usetikzlibrary{calc, intersections, arrows.meta, positioning}
\usetikzlibrary{arrows}
\usetikzlibrary{pgfplots.groupplots}
\tikzset{>=stealth',inner sep=0pt,outer sep=2pt}
\usepackage{pgfplots}
\pgfplotsset{compat=1.18}

\usepackage{colortbl}

\usepackage[most]{tcolorbox}

\def\hquad{\hskip0.5em\relax}
\def\quad{\hskip1em\relax}
\def\qquad{\hskip2em\relax}
\def\qqquad{\hskip3em\relax}

\hypersetup{
  colorlinks=true,
  citecolor=blue,
  linkcolor=black,
  breaklinks=true,
  hidelinks
}

\newtheorem{corollary}[theorem]{Corollary}
\newtheorem{lemma}[theorem]{Lemma}

\newtheorem*{constraints*}{Constraints}
\newtheorem{model}{Model}
\newtheorem{modelmy}{Model}

\newenvironment{modelCustom}[2]
{
    \begingroup
    \setcounter{modelmy}{#1-1}%
    \renewcommand{\themodelmy}{#1#2}%
    \begin{modelmy}
}
{
    \end{modelmy}
    \endgroup
}

\newcounter{maincounter}
\newcounter{tech}

\newtcbtheorem[number format=\alph]{technique}{Approximation}{
  colback=white,
  colframe=blue!10,
  before skip=1em,
  coltitle=blue!20!black,
  fonttitle=\small\bfseries,
  breakable
}{tech}

\usepackage{lscape}

\makeatletter
\newcommand{\pushright}[1]{\ifmeasuring@#1\else\omit\hfill$\displaystyle#1$\fi\ignorespaces}
\newcommand{\pushleft}[1]{\ifmeasuring@#1\else\omit$\displaystyle#1$\hfill\fi\ignorespaces}
\makeatother

\usepackage{etoolbox}
\usepackage{float}

\usepackage{eqnarray,xstring}

\NewDocumentEnvironment{FeasRegion}{sO{center}}
    {\IfBooleanTF{#1}
        {\IfEqCase{#2}{
            {center}{\begin{equationarray*}{c@{\hskip2em}l}}
            {align}{\begin{equationarray*}{r@{\hskip0.5em}c@{\hskip0.5em}l@{\hskip2em}l}}
        }}
        {\IfEqCase{#2}{
            {center}{\begin{equationarray}{c@{\hskip2em}l}}
            {align}{\begin{equationarray}{r@{\hskip0.5em}c@{\hskip0.5em}l@{\hskip2em}l}}
    }}}
    {\IfBooleanTF{#1}
        {\end{equationarray*}\ignorespacesafterend}
        {\end{equationarray}\ignorespacesafterend}
    }

\NewDocumentEnvironment{LPArray}{sO{center}mm}
    {\IfBooleanTF{#1}
        {\IfEqCase{#2}{
            {center}{\begin{equationarray*}{rc@{\hskip2em}l}
                \text{#3} & \multicolumn{2}{l}{#4} \\
         	\text{s.t.}}
            {align}{\begin{equationarray*}{rr@{\hskip0.5em}c@{\hskip0.5em}l@{\hskip2em}l}
                \text{#3} & \multicolumn{4}{l}{#4} \\
         	\text{s.t.}}
        }}
        {\IfEqCase{#2}{
            {center}{\begin{equationarray}{rc@{\hskip2em}l}
                \text{#3} & \multicolumn{2}{l}{#4} \\
         	\text{s.t.}}
            {align}{\begin{equationarray}{rr@{\hskip0.5em}c@{\hskip0.5em}l@{\hskip2em}l}
                \text{#3} & \multicolumn{4}{l}{#4} \\
         	\text{s.t.}}
    }}}
    {\IfBooleanTF{#1}
        {\end{equationarray*}}
        {\end{equationarray}}
    \ignorespacesafterend\unskip}

\usepackage{environ}

\NewEnviron{SetArray}[2]{%
    \begin{empheq}[left={#1\left\lbrace#2\hskip0.5em\relax\middle|\ },right={\right\rbrace}]{align}
  \BODY
\end{empheq}
}

\NewEnviron{partialSetArray}{%
    \begin{empheq}[left={ \left.
 \relax \middle|\ },right={\right\rbrace}]{align}
  \BODY
\end{empheq}
}

\NewEnviron{SetArray.}[2]{%
    \begin{empheq}[left={#1\left\lbrace#2\hskip0.5em\relax\middle|\ },right={\right\rbrace.}]{align}
  \BODY
\end{empheq}
}

\renewcommand{\P}{\mathbb P}
\newcommand{\Reals}{\mathbb R}
\newcommand{\R}{\mathbb R}

\newcommand{\TV}{\hyperlink{def:TV}{\textnormal{TV}}}
\newcommand{\DV}{\hyperlink{def:DV}{\textnormal{DV}}}

\newcommand{\p}{\hyperref[tab:labeling]{p}}
\newcommand{\tildep}{\hyperref[tab:labeling]{\tilde p}}

\newcommand{\varColor}[1]{{\color{blue} #1 }}

\newcommand{\brObj}{{\phi}}
\newcommand{\brProbit}{{\Phi_{\beta, \beta_0}}}

\newcommand{\Z}{\mathbb{Z}}

\newcommand{\XFeasPoly}{\mathscr{S}}
\newcommand{\XYZFeasPoly}{\mathscr{U}}
\newcommand{\DomOne}{\mathscr{D}}
\newcommand{\DomTwo}{\mathscr{B}}
\newcommand{\Approx}{\mathscr{A}}
\newcommand{\Relax}{\mathscr{R}}

\newcommand{\XContPoly}{{\XFeasPoly_{\textnormal{C}}}}

\newcommand{\districts}{{M}}
\newcommand{\numDistricts}{{m}}

\newcommand{\edges}{{E}}
\newcommand{\numEdges}{{|E|}}

\newcommand{\parcels}{{N}}
\newcommand{\numParcels}{{n}}

\newcommand{\numBreakpoints}{\ell}
\newcommand{\breakpointsIter}{[\![\ell]\!]}

\newcommand{\years}{S}
\newcommand{\yearsIter}{[\![s]\!]}
\newcommand{\numYears}{s}

\newcommand{\bigM}{\mathscr{M}}
\newcommand{\vecM}{\boldsymbol{\mathscr{M}}}
\newcommand{\bigN}{\mathscr{N}}
\newcommand{\vecN}{\boldsymbol{\mathscr{N}}}

\newcommand{\LS}{\hquad<\hquad}

\newcommand{\LEQ}{\hquad\leq\hquad}
\newcommand{\GEQ}{\hquad\geq\hquad}
\newcommand{\EQ}{\hquad=\hquad}

\newcommand{\DEF}{\hquad\coloneq\hquad}

\newcommand{\SUBEQ}{\hquad\subseteq\hquad}

\newcommand{\MOD}{\text{mod\ }}

\newcommand{\x}{{\varColor{ \mathbf{x}}}}
\newcommand{\xx}{{\varColor{x}}}
\newcommand{\y}{{\varColor{\mathbf{y}}}}
\newcommand{\yy}{{\varColor{y}}}
\newcommand{\z}{{\varColor{\mathbf{z}}}}
\newcommand{\zz}{{\varColor{z}}}

\newcommand{\ff}{{\varColor{f}}}
\newcommand{\rr}{{\varColor{r}}}
\renewcommand{\ss}{{\varColor{s}}}

\newcommand{\deltaVar}{\varColor{\delta}}
\newcommand{\deltaVec}{\varColor{\boldsymbol\delta}}
\newcommand{\etaVar}{\varColor{\eta}}

\newcommand{\xiVar}{\varColor{\xi}}
\newcommand{\lambdaVar}{\varColor{\lambda}}
\newcommand{\lambdaVec}{\varColor{\boldsymbol\lambda}}

\newcommand{\VAP}{\hyperlink{def:VAP}{\textnormal{VAP}}}
\newcommand{\BVAP}{\hyperlink{def:BVAP}{\textnormal{BVAP}}}
\newcommand{\HVAP}{\hyperlink{def:HVAP}{\textnormal{HVAP}}}

\newcommand{\CPVI}{\mathrm{CPVI}}

\newcommand{\RH}[1]{}
\newcommand{\RHans}[1]{}
\newcommand{\RHilans}[1]{}
\newcommand{\RHil}[1]{}
\newcommand{\JF}[1]{}
\newcommand{\JFil}[1]{}
\newcommand{\JFilres}[1]{}
\newcommand{\JFilans}[1]{}

\newcommand{\lowhat}[1]{%
  \widehat{\smash{#1}\vphantom{x}}\vphantom{#1}%
}

\newcommand{\jfixed}{\hat\jmath}

\newcommand{\conv}{\mathrm{conv}}
\newcommand{\cone}{\mathrm{cone}}

\newcommand{\alphabm}{\boldsymbol \alpha}
\newcommand{\betabm}{\boldsymbol \beta}
\newcommand{\deltabm}{\boldsymbol \delta}
\newcommand{\omegabm}{\boldsymbol \omega}

\DeclareMathOperator*{\proj}{\mathrm{proj}\!}

\newcommand{\pder}[1]{\frac{\partial }{\partial #1}}

\newcommand{\X}{\mathbf{X}}

\newcommand{\doubleblind}[1]{#1}

\newcommand{\graycell}{\color{gray!40}}

\newcommand{\MID}{\hquad\middle|\ }

\DeclareMathOperator{\hyp}{hyp}
\DeclareMathOperator{\gra}{gra}

\DeclareMathOperator{\step}{step}
\DeclareMathOperator{\pwlg}{pwlg}
\DeclareMathOperator{\pwlh}{pwlh}
\DeclareMathOperator{\BNFull}{BNFull}
\DeclareMathOperator{\BNStep}{BNStep}
\DeclareMathOperator{\BNRot}{BNRot}
\DeclareMathOperator{\LogE}{logE}

\def\cdf(#1)(#2)(#3){0.5*(1+(erf((#1-#2)/(#3*sqrt(2)))))}

\tikzset{
    declare function={
        normcdf(\x,\m,\s)=1/(1 + exp(-0.07056*((\x-\m)/\s)^3 - 1.5976*(\x-\m)/\s));
    }
}

\newcommand{\steps}{\lowhat{\step}_{\{b_{tk}\}}^{\numYears}\!\left(\psi\right)}

\usepackage{longtable}
\floatstyle{plaintop}
\restylefloat{table}
\usepackage{setspace}
\usepackage{pifont}
\newcommand{\cm}{\ding{51}}
\newcommand{\xm}{\ding{55}}

\usepackage{mathabx}

\newcommand{\referee}[1]{}
\newcommand{\suggestion}[1]{}

\newcommand{\deleted}[1]{}

\begin{document}

\title[Optimizing Representation in Redistricting]{Optimizing Representation in Redistricting: Dual Bounds for Partitioning Problems with Non-Convex Objectives}

\author*[1]{\fnm{Jamie} \sur{Fravel}}\email{jamiefravel@me.com}
\author[1]{\fnm{Robert} \sur{Hildebrand}}\email{rhil@vt.edu}
\author[2]{\fnm{Nicholas} \sur{Goedert}}\email{ngoedert@vt.edu}
\author[3]{\fnm{Laurel} \sur{Travis}}\email{ltravis@vt.edu}
\author[4]{\fnm{Matthew} \sur{Pierson}}\email{mxzf@vt.edu}

\affil[1]{\orgdiv{Grado Department of Industrial and Systems Engineering}, \orgname{Virginia Tech}, \orgaddress{\city{Blacksburg}, \state{Virginia}, \country{USA}}}

\affil[2]{\orgdiv{Department of Political Science}, \orgname{Virginia Tech}, \orgaddress{\city{Blacksburg}, \state{Virginia}, \country{USA}}}

\affil[3]{\orgdiv{Business Information Technology}, \orgname{Virginia Tech}, \orgaddress{\city{Blacksburg}, \state{Virginia}, \country{USA}}}

\affil[4]{\orgdiv{Center for Geospatial Information Technology}, \orgname{Virginia Tech}, \orgaddress{\city{Blacksburg}, \state{Virginia}, \country{USA}}}

\abstract{We investigate optimization models for the purpose of computational redistricting. Our focus is on nonconvex objectives for estimating expected Black Representatives and Political Representation. The objectives are a composition of a ratio of variables and a normal distribution's cumulative distribution function (or ``probit curve"). We extend the work of Validi et al.~\cite{validi2022imposing}, which presented a robust implementation of contiguity constraints. By developing mixed integer linear programming models that closely approximate the parent nonlinear model, our approaches yield tight bounds on these optimization problems. We exhibit the effectiveness of these approaches on county-level data.}


\maketitle

\section{Introduction}

We study mixed integer optimization problems of the form
    \begin{LPArray}{Maximize}{\sum_{j=1}^{\numDistricts} \brObj\left(\beta_0 + \beta_j\sum_{t=1}^{\numYears} \frac{\yy_{jt}    \label{eq:general}}{\zz_{jt}}\right)}
        & (\x,\y,\z)  \in \XFeasPoly, \notag
    \end{LPArray}
where $\XFeasPoly$ is a mixed integer linear set, $m$ and $s$ are integers, and $\brObj$ is an increasing function.  We use boldface to denote vectors and blue color to denote variables in our optimization models.

We focus on applications of state redistricting in the United States, with the aim of partitioning land parcels into a set of contiguous districts that are approximately equal in population.  Variations in state policies for valid district plans exist and\textemdash although it is formally required by only 23 states\textemdash contiguity of districts is a norm typically practiced in all states~\cite{altman_mcdonald_redistricting, duchin_gerrymandering_2018}.

We focus on optimizing districts for both minority and political representation via \eqref{eq:general} where $\brObj$ is the cumulative distribution function of the standard normal distribution and $\beta_0\in\Reals$ and $\betabm \in \Reals^m$ are parameters drawn from historical data. This optimization problem maximizes the expected value of some outcome under the probit model; that is, for random variables $Y \in \{0,1\}$ and $\X \in \Reals^n$, we assume $\P(Y=1\, |\, \X) = \brObj(\beta_0 + \betabm^\top \X)$.  Specifically, $\X$ is the list of features used to estimate the outcome of $Y$, which can be inferred through curve fitting from some distributions of historical data.

The features of \eqref{eq:general}, as applied to political districting, appear as ratios (or the sums of ratios) of the number of voters of one type divided by the number of voters of a larger type \textemdash typically the total voter base.  The key obstacles presented by \eqref{eq:general} are the non-linearity of $\brObj$ itself, the rational terms $\yy_{jt}/\zz_{jt}$, and the complexity and dimension of the set $\XFeasPoly$.  Before presenting our main contributions, we provide a background on redistricting to emphasize its importance and review recent work on algorithmic redistricting.

\subsection{Background on Redistricting}
In the United States, members of both the U.S. Congress and state legislatures are primarily elected through single-member geographic districts. The design of these districts (and the institutions and norms governing their creation) are vital to legislative composition and various voter groups' political power. Districts can be drawn to suppress or enhance interests of political parties, racial or ethnic groups, and economic classes or to influence policies that mitigate or exacerbate inequality.

In most states, the redistricting of state and federal legislative chambers is performed by the state legislature following the national census every decade. Partisan gerrymandering, or the deliberate drawing of district lines to benefit one party, has been a controversial practice in the U.S. and has received increased attention in recent years due to highly effective Republican gerrymanders. These cases have resulted in competing and sometimes conflicting metrics for evaluating partisan bias in maps, with no definitive resolution on a single legal or empirical standard~\cite{best-considering-2017, gelman, mcdonald_unfair_2015, McGann2016, nagle_measures_2015, stephanopoulos_measure_2018}.

As an example case we discuss the Virginia House of Delegates in the southeastern third of the state. This region and chamber were the subject of recent federal litigation in \textit{Bethune-Hill v.\ Virginia State Board of Elections (2017)} under which the initial map was struck down as an unconstitutional, racial gerrymander. The Virginia state legislature consists of two groups: the Senate and the House of Delegates. The latter consists of 100 single-member districts representing a population of around 80,000 people. The map of the House of Delegates was redrawn after the 2010 Census by a Republican-controlled legislature. This map created twelve districts in the southeastern region of the state with a Black Voting-Age-Population (BVAP) of over 55 percent their total Voting-Age-Population (VAP). Following a federal suit from residents of these districts, the U.S. Supreme Court ruled that the map violated the Equal Protection Clause with race being shown to be the predominant factor motivating the map. These districts were an example of legislators having packed voters of one type into a few districts, which diluted their influence throughout the state. Eleven districts were overturned, but the Virginia government, with partisan control divided, was unable to agree on a revision. The District Court, acting on the recommendation of special master Bernard Grofman, implemented its own~\cite{grofman2018special}.

\subsection{Integer Programming used for Algorithmic Redistricting}
The early stages of algorithmic redistricting research were marked by works such as ~\cite{Hess1965} and ~\cite{Garfinkel1970}. For a comprehensive understanding of the history of the field, the reader may refer to ~\cite{validi2021political, validi2022imposing}. Here, we will summarize some of the latest developments in the field.

One promising optimization approach that has gained attention in recent years is the ``Fairmandering" method~\cite{wes2021}. It employs a column generation algorithm to optimize various objective functions, and is a compelling approach for optimizing a wide range of metrics with the potential for optimality over a large pool of maps. However, this method does not guarantee optimality as it does not provide dual bounds for the entire problem.

Another successful approach is the use of ensemble methods, such as Markov chain Monte Carlo (MCMC) (e.g.,~\cite{mcmc1, mcmc3, mcmc2}) that samples maps from a distribution using random steps to walk through the solution space and generate a large ensemble of maps, allowing for statistical analysis. Duchin et\,al.~\cite{duchin_gerrymandering_2018} provide the open-source python package \texttt{gerrychain} for such computations and use a \emph{recombination} approach for stepping through the solution space. This has been modified for optimization purposes in~\cite{validi2021political, bvap_southern}. A particularly noteworthy example is the ``short bursts" technique developed in~\cite{short_bursts}, which selects the best improvements from a collection of steps to optimize the solution. This paper stems from our curiosity into the performance of the ensemble methods which we have applied to Racial and Partisan Representation objectives in \cite{bvap_southern,political}. These methods do not generate dual bounds, so it can be difficult to verify the optimality of any given solution.

Recent studies, such as Becker et al.~\cite{Becker2021}, have generated random maps with a focus on ``effectiveness" of electing a representative or competitiveness while systematically including or excluding these maps within larger ensembles \textemdash particularly to meet certain VRA criteria related to Latino Representation in Texas. These works demonstrate a careful consideration of detailed, state-specific information. Recent work employing mixed integer programming approaches~\cite{Belotti2023} demonstrates that it is indeed possible to incorporate state-specific legal requirements such as population deviation limits, compactness measures, and Voting Rights Act compliance criteria into tractable optimization models. Our current focus is on developing exact optimization methodology; while we acknowledge the importance of incorporating such detailed criteria, this paper does not delve into these state-specific complexities. Instead, our focus is on developing and refining the underlying algorithmic approaches so that greater complexities can be addressed in future. 

A major advancement in the field of exact optimization was made by Validi et al. ~\cite{validi2022imposing, validi2021political} who applied lazy constraint techniques to impose contiguity within an integer programming framework. They also developed a preprocessing technique called ``variable fixing'' which uses Lagrangian Relaxation to eliminate a significant number of variables by ruling out obviously suboptimal or infeasible assignments. This technique is specific to the compactness objective they studied.

Due to the large scale of the underlying integer program, we found it important to use a powerful linear integer programming solver that can handle problems at a large scale. Additionally, being defined by an integral, our probit objective function does not have a closed form and is poorly suited to optimization without some kind of discretization or linearization methods. For these reasons, we did not use any nonlinear optimization software, such as BARON, which is more focused on solving complicated MINLPs at not quite as large scale.

\subsection{Contributions}
This paper addresses measures of representation. Because there are several heuristic techniques available \cite{wes2021} for optimizing arbitrary measures, our main purpose is to prove strong dual bounds and prove optimality (or near optimality) of certain solutions. This is fundamental to understanding the limitations of certain objectives and is rarely addressed when generating suitable solutions.  

Our main contribution is developing novel models for mixed integer nonlinear programs when rational expressions are composed with increasing, nonlinear functions. We study this in the context of high-dimensional integer programming, wherein the underlying constraints are already difficult to address. In doing so, we expand the literature on optimization for redistricting to non-convex objective functions that more realistically reflect the interests of political scientists. Prior work has focused on formulations preferred by mathematicians due to their convenient structure for optimization.  For instance, \cite{validi2022imposing} addresses a measure of compactness that is closely related to a facility location problem more often studied in the optimization community.  

We build on the excellent work in~\cite{validi2021political, validi2022imposing} that provides substantial improvements in the handling of contiguity constraints; we use their models and lazy constraint techniques as a basis for handling the domain of our problem. We develop several techniques for linearizing nonconvex objective functions.  We compare a naive model using \emph{piecewise-linear approximation} to novel mixed integer linear relaxations including optimized \emph{stepwise function relaxations}.  We then present two \emph{compact formulations}: one based on the Ben-Tal \& Nemirovski approximation of the second order cone~\cite{Ben-Tal-Nemirovski} and another based on logarithmic embeddings of 2-dimensional piecewise-linear functions. We include some improvements in these approaches which are valid for our context.

We further compute improved bounds on our domain for the variables (both in preprocessing and through gradient cuts) and show that these bounds are effective, 40--75\% across our test instances (see Figure~\ref{fig:LogE_with_bounds}), particularly when combined with the logarithmic piecewise-linear approximation.

Lastly, we show how the multi-ratio objectives can be modeled using a stepwise formulation before presenting extensive computational results to compare our formulations and techniques.

\paragraph{On Legality and Compactness of Maps.}
The district maps generated in this study are not intended to satisfy all legal or normative criteria for valid redistricting plans. This work focuses on methodological development, specifically, the application of mixed-integer nonlinear programming techniques to the redistricting problem, rather than on producing implementable or legally compliant district maps.

There are several reasons why the resulting maps may not meet legal standards. First, we do not enforce tightly constrained population deviations across districts. Relaxing this constraint makes the underlying optimization problem more tractable and allows for clearer insights into the behavior of different nonlinear formulations. Second, legal requirements for redistricting vary across jurisdictions and can include complex criteria such as contiguity, multiple notions of compactness, preservation of communities of interest, and compliance with the Voting Rights Act. For example, district shapes that appear highly irregular or non-compact may raise constitutional concerns under cases such as \textit{Shaw v. Reno} (1993), where race-based districting without sufficient justification was found to violate the Equal Protection Clause. Modeling these legal considerations in full detail poses significant challenges and lies beyond the scope of the present study.

As such, the maps presented here should be viewed as theoretical outputs of mathematical optimization models, not as proposals for legally viable redistricting plans.

\subsection{Outline}

In Section~\ref{sec:core-models}, we precisely define the optimization problems that we are aiming to solve. We start with core models and describe our approach to enforcing contiguity before explaining the nonconvex objectives of interest. 
In Section~\ref{sec:MILP-relaxations}, we present mixed-integer linear programming (MILP) relaxations and approximations based on graph and hypograph formulations.
In Section~\ref{sec:single-ratio}, we analyze some common techniques for addressing nonlinear representation objectives with a single ratio. 
In Section~\ref{sec:LogSize-single-ratio}, we expand our scope to include more abstract methods which encode similar approximations using only a logarithmic number of binary variables.
In Section~\ref{sec:extra-bounds}, we discuss methods for improving bounds as a pre-processing step and in the form of cuts.
In Section~\ref{sec:multi-ratios}, we expand the ideas from the previous sections to apply to Political Representation objectives with multiple ratios.
We provide some tables of computations in Section~\ref{sec:Computations-and-Conclusions} before delivering our conclusions.
In Appendix~\ref{sec:computations}, we compile the best bounds found for each objective on the instance classes and provide more details on our computational experiments.

\section {Redistricting Optimization Models}\label{sec:core-models}
We present the fundamental parameters of our problem in Table \ref{tab:labeling} before constructing the basis of our Mixed-Integer Linear Programming (MILP) Models. Our data takes the form of graphs, created from national census data, which are partitioned into districts. The graph of each state implies the following sets and is embedded with the following parameters:

    \begin{table}[ht]
    \begin{longtable*}{lll}
    Set & Size & Description \\
    \hline
    $\parcels$ & $\numParcels$ & Set of nodes in the graph (i.e., parcels of land)\\
    $\districts$ & $\numDistricts$ & Set of districts into which the nodes are partitioned; define $M = \{1,\ldots,m\}$ \\
    $\edges$ & $\numEdges$ & Set of edges in the graph (i.e., borders between parcels of land)
    \\
    \\
    Parameter & Space & Description \\
    \hline
    $\p_i$ & $\Z_+$ & Population of node $i$ for $i\in\parcels$\\
    $\tildep$ & $\Reals$ & Average population of districts; i.e. $\tfrac{\sum_{i\in\parcels}p_i}{m}$\\
    $\tau$ & [0,1] & Allowable fraction of deviation from average population.
    \end{longtable*}
    \caption{Sets and parameters for the labeling model.\vspace{-1em}
    }
    \label{tab:labeling}
    \end{table}

For each $i\in \parcels$ and $j \in \districts$, define the decision variable $\hypertarget{def:x}{\xx_{ij}}$ to take the value $1$ if node $i$ is assigned to district $j$ and $0$ otherwise. The following constraints form the basis of a labeling style redistricting model.
    \begin{constraints*}[Labeling Constraints]
    \begin{subequations}
    \begin{FeasRegion}
    \sum_{j\in \districts}\xx_{ij}  \EQ  1                                     & \forall\ i\in\parcels \label{labeling:basic:assignment}\\
    (1-\tau)\tildep  \LEQ  \sum_{i\in \parcels}\p_i\xx_{ij}  \LEQ  (1+\tau)\tildep    & \forall\ j \in \districts \label{labeling:basic:pop}
    \end{FeasRegion}
    \end{subequations}
    \end{constraints*}
Constraint \eqref{labeling:basic:assignment} assigns each node to a unique district while \eqref{labeling:basic:pop} ensures that the population of each district is within $\tau$ of the mean $\tildep$. 
Given in \cite{validi2022imposing}, the Hess model is also popular. It resembles a clustering problem which has more variables but is useful when considering spatial objectives like compactness or moment of inertia.

\subsection{ Contiguity}\label{sec:Contiguity}
Given a set of edges $\edges \subseteq \big\{\{u,v\}:u,v\in \parcels,\ u\neq v\big\}$, we say that a partition $\parcels_1 \sqcup \dots \sqcup \parcels_{\numDistricts}$ of the nodes $\parcels$ is contiguous if the induced subgraph $G_j = G[N_j]$ is connected for each district $j \in \districts$. For our data, $\edges$ is constructed by the physical adjacency of land parcels represented in $\parcels$; thus the resulting graph $G = (\parcels, \edges)$ is planar. 

Contiguity is often considered difficult to include in Integer Programming, but recent papers by Validi et al. \cite{validi2021political, validi2022imposing} demonstrate efficient methods for enforcing contiguity in models such as our own. We chose to adapt their Python code and Gurobi model (given in \cite[\href{https://github.com/hamidrezavalidi/Political-Districting-to-Minimize-Cut-Edges}{GitHub}]{validi2021political}) to our objectives and approximations. We use the notation $\XContPoly$ to refer to the contiguity-enforced feasible region of our models.

    \begin{constraints*}[Contiguous Labeling Constraints]
    \begin{equation*}
    \hypertarget{def:Plcut}\XContPoly = \left\lbrace \x\in\{0,1\}^{n\times m}\ :\ \eqref{labeling:basic:assignment},\ \eqref{labeling:basic:pop},\  \xx_{aj}+\xx_{bj} \leq 1+\sum_{c\,\in\,C}\xx_{cj} \qquad \forall\ \text{$a,\!b$-separators $C$}, \ \forall \  j  \right\rbrace.
    \end{equation*}
    \end{constraints*}
    An $a,b$-separator is a set of nodes $C$ such that any path between nodes $a$ and $b$ must pass through at least one element of the separator.

The additional constraint ensures that, if both nodes $a$ and $b$ are assigned to district $j$, then at least one element of each $a,b$-separator must also be assigned to the same district. Validi et al.\ propose strengthened versions of these inequalities using district population upper bounds; see~\cite{validi2021political, validi2022imposing} for details. There are exponentially many $a,b$-separators but the corresponding inequalities can be lazily added via an efficient separation algorithm.

\subsection{The Expected Black Representation Objective}
The metric that inspired this work is that of Black Representation since it is the crucial value at issue in the Bethune-Hill racial gerrymandering literature.  We operationalize Black Representation as the sum of the likelihood of each district electing a Black Representative which we take to be a function of the ratio of Black Voting-Age-Population (BVAP) to the total Voting-Age-Population (VAP) in the district.  This likelihood is modeled as a probit function with parameter values estimated from the relationship between BVAP/VAP and elected representatives in national legislative elections since 1992.  The function is the affine transformation of the cumulative standard-normal distribution $\Phi$ evaluated at $r = \frac{\BVAP}{\VAP}$, that is, the ratio of Black voters to the total population of voters in a district. 
 In particular, 
\begin{align*}
     \Phi_{\beta, \beta_0}(r) \EQ  {\frac{1}{\sqrt{2\pi}} \int_{-\infty}^{\beta \cdot r - \beta_0} e^{-\frac{1}{2}t^2} \text dt}.
\end{align*}
Being drawn from historical data, the parameters $\beta$ and $\beta_0$ vary from region to region. For the sake of simplicity we use an average over a group of southern states (MD, VA, NC, TN, AL, GA, LA, MS, SC): $\beta = 6.826$ and $\beta_0 = 2.827$. 

We now define the first optimization problem we will study. Letting \hypertarget{def:VAP}{$\VAP_i$} and \hypertarget{def:BVAP}{$\BVAP_i$} respectively represent total and Black Voting-Age-Populations of each node $i \in \parcels$, define new auxiliary variables \hypertarget{def:y}{$\yy_{j}$} and \hypertarget{def:z}{$\zz_{j}$} to be the Black and total Voting-Age-Populations of each district $j \in \districts$. Their mathematical definitions are given in the following set:
    \begin{equation}
    \XYZFeasPoly_\textrm{B} = 
        \left\lbrace
        (\x,\y,\z) \in \XContPoly\!\times\!\Reals^{\numDistricts}\!\times\!\Reals^{\numDistricts}
        \MID
            \begin{aligned}
            \yy_j  &=  \textstyle{\sum_{i\in \parcels}}\BVAP_i\xx_{ij}
                &\forall\ j \in \districts  \\
            \zz_j  &=  \textstyle{\sum_{i\in \parcels}\VAP_i\xx_{ij}}
                &\forall\ j \in \districts 
            \end{aligned}
        \right\rbrace.
    \end{equation}
Model \ref{model:BR-Redistricting} aims to maximize the expected number of Black Representatives elected in a given state.
    \begin{model}[BR-Redistricting]\label{model:BR-Redistricting}
        \begin{equation*}
        \max_{(\x,\y,\z)\,\in\,\XYZFeasPoly_\textrm{B}}\sum_{j\,\in\, \districts}\brProbit\left(\tfrac{\yy_j}{\zz_j}\right)
        \end{equation*}
    \end{model}

Notice that $\brProbit$ does not have a convenient closed form and is evaluated over a rational term. This means that the objective function is non-linear and, in fact, nonconvex. We will discuss linearization techniques in Section \ref{sec:single-ratio}.

\subsection{Including Hispanic Population}

In census data, BVAP counts individuals who identify as Black alone (non-Hispanic), while HVAP counts individuals who identify as Hispanic or Latino of any race. These categories are mutually exclusive in our analysis, so individuals are not double-counted.
The number of Hispanic voters (\hypertarget{def:HVAP}{HVAP}) in a district is also correlated with the probability of electing a Black Representative. See \cite{bvap_southern} for details. Following the approach in \cite{bvap_southern}, using ratio of \BVAP\ and \HVAP\ to \VAP\ as features, define the function
    \begin{equation}\label{obj:HVAP}
   \Phi_{1, -\beta_0}(r)   \qquad \textrm{where} \qquad 
    r  \EQ  \frac{\beta_b\cdot\textrm{BVAP}+\beta_h\cdot\textrm{HVAP}}{\textrm{VAP}}.
    \end{equation}
Again, the $\beta$-values are drawn from historical data with an added regional distinction. The data suggests a clear difference in the effect of BVAP and HVAP in two regional areas, \emph{rim south} states and \emph{deep south states}, suggesting cultural differences in these areas.  For rim south states (MD, VA, NC, TN): $\beta_0 = -4.194$, $\beta_b = 0.975$, and $\beta_h = 0.3$. On the other hand, for deep south states (AL, GA, LA, MS, SC): $\beta_0 = -4.729$, $\beta_b = 1.044$, and $\beta_h = 0.3$. These values see more detailed discussion in~\cite{bvap_southern} wherein this objective is heuristically explored in connection with compactness.

Letting $\HVAP_i$ indicate the Hispanic Voting-Age-Population of each node $i \in \parcels$, redefine the auxiliary variable $\yy_j$ to be the combined $\HVAP$ and $\BVAP$ score for each district $j \in \districts$ as in \eqref{obj:HVAP}. The adjusted definition is given in the following set:
    \begin{equation}
    \XYZFeasPoly_\textrm{H} = 
        \left\lbrace
        (\x,\y,\z) \in \XContPoly\!\times\!\Reals^{\districts}\!\times\!\Reals^{\districts}
        \MID
            \begin{aligned}
            \yy_j  &=  \textstyle{\sum_{i\in \parcels}}(\beta_b\,\BVAP_i+\beta_h\,\HVAP_i)\xx_{ij}
                &\forall\ j \in \districts  \\
            \zz_j  &=  \textstyle{\sum_{i\in \parcels}\VAP_i\xx_{ij}}
                &\forall\ j \in \districts 
            \end{aligned}
        \right\rbrace.
    \end{equation}
Model \ref{model:HVAP} expands on Model \ref{model:BR-Redistricting} by considering the Hispanic population of each district.
    \begin{modelCustom}{1}{'}[BRH-Redistricting]\label{model:HVAP}
        \begin{equation*}
        \max_{(\x,\y,\z)\,\in\,\XYZFeasPoly_\textrm{H}}\sum_{j\,\in\,\districts}\Phi_{1, \beta_0}\left(\tfrac{\yy_j}{\zz_j}\right)
        \end{equation*}
    \end{modelCustom}

Again, the objective is nonconvex but any linearization strategy for Model \ref{model:BR-Redistricting} will also apply to Model \ref{model:HVAP} as both have a single ratio in the objective.

\subsection{The Expected Partisan Representation Objective}

Cook's Partisan Voting Index (CPVI) is a widely used measure of political lean in a region. It compares the average share of votes received by a party of interest to the national average over the two most recent presidential elections (here, 2016 and 2020). Let $V_{16}$ and $V_{20}$ denote the number of votes for the party of interest in 2016 and 2020, and let $T_{16}$ and $T_{20}$ be the total number of major-party votes cast in those years, respectively. Then we define:

\begin{equation}\label{eqn:CPVI}
    \CPVI  \EQ  50\left(\frac{V_{16}}{T_{16}}+\frac{V_{20}}{T_{20}}\right)-\theta,
\end{equation}
where $\theta$ is the average national support for the party of interest across both elections (e.g., $\theta = 51.69$ for the Democratic Party).

This measure is positive if the region supported the party of interest more than the national average, and negative otherwise. Following the approach in \cite{political}, the probability of a district electing a representative from the party of interest can be predicted by CPVI:

\begin{equation}
   \mathbb{P}(\textrm{party wins})  \EQ  \brProbit\left(r_1 + r_2\right) \qquad\textrm{where}\qquad 
    r_1  \EQ  \frac{V_{16}}{T_{16}}, \qquad
    r_2 = \frac{V_{20}}{T_{20}}.\label{obj:CPVI}
\end{equation}

Historical data suggests that $\beta = \frac{50}{4.8} = 10.42$ and $\beta_0 = \frac{\theta}{4.8} = 10.77$. See \cite{political} for more details on this calibration.

Let $T_{it}$ and $V_{it}$ represent the total number of major-party votes and the number of votes for the party of interest, respectively, cast in each unit \(i\in\parcels\) and year \(t\in\years = \{2016,2020\}\). Define the auxiliary variables $\yy_{jt}$ and $\zz_{jt}$ to be the total party and total votes assigned to each district \(j \in \{1,\dots,\numDistricts\}\) in year \(t\), respectively:

\begin{equation}
\XYZFeasPoly_\textrm{P} = 
    \left\lbrace
    (\x,\y,\z) \in \XContPoly\!\times\!\Reals^{\districts\times \years}\!\times\!\Reals^{\districts \times \years}
    \MID
        \begin{aligned}
        \yy_{jt}  &=  \textstyle{\sum_{i\in \parcels}} V_{it}\,\xx_{ij}
            &\forall\ j \in \districts,\ t\in\years  \\
        \zz_{jt}  &=  \textstyle{\sum_{i\in \parcels}} T_{it}\,\xx_{ij}
            &\forall\ j \in \districts,\ t\in\years 
        \end{aligned}
    \right\rbrace.
\end{equation}

Model \ref{model:CPVI} then maximizes the expected number of districts that elect a representative from the party of interest:
\begin{model}[CPVI-Redistricting]\label{model:CPVI}
    \begin{equation*}
    \max_{(\x,\y,\z)\,\in\,\XYZFeasPoly_\textrm{P}}\sum_{j\,\in\,\districts}\brProbit \left(\frac{\yy_{j\,16}}{\zz_{j\,16}}+\frac{\yy_{j\,20}}{\zz_{j\,20}}\right)
    \end{equation*}
\end{model}

Unlike Models \ref{model:BR-Redistricting} and \ref{model:HVAP}, the objective of Model \ref{model:CPVI} involves two rational expressions. We generalize our linearization techniques to handle multiple ratios in Section \ref{sec:multi-ratios}.

\section{MILP Relaxations and Approximations via Graph or Hypograph}
\label{sec:MILP-relaxations}
We use various mixed-integer linear programming (MILP) formulations to approximate the nonlinear objectives. First we introduce some terms and demonstrate their properties in support of a general approach.
    \begin{definition}[Graph and Hypograph]
    For a function $\brObj: \DomOne \to \Reals$ with $\DomOne \subseteq \Reals^n$, the \emph{graph} $\gra(\brObj)$ and \emph{hypograph} $\hyp(\brObj)$ of $\brObj$ are given by:
        \begin{equation*}
        \gra(\brObj) \coloneq \left\lbrace(\mathbf{r}, f) \in \DomOne\!\times\!\Reals \MID f = \brObj(\mathbf{r})\right\rbrace  \qquad\textrm{and}\qquad
        \hyp(\brObj) \coloneq \left\lbrace(\mathbf{r}, f) \in \DomOne\!\times\!\Reals \MID f \leq \brObj(\mathbf{r})\right\rbrace.
        \end{equation*}
    \end{definition}
    
    \begin{definition}[Graph and Hypograph Reformulations]
    (See, e.g., \cite{rockafellar1970convex} for standard definitions of graph and hypograph.)
    Optimization problems of the form 
        \begin{equation}\label{model:general}
        \max_{\mathbf{r}\,\in\,\DomOne}\sum_{j=1}^{\numDistricts}\brObj_j(\mathbf{r})
        \end{equation}
    may be equivalently reformulated in terms of either the graph or hypograph of $\brObj$:
        \begin{equation*}
        \max_{\mathbf{r}\,\in\,\DomOne}\left\{\sum_{j=1}^{\numDistricts}f_j  \MID (\mathbf{r},f_j) \in \gra(\brObj_j),\ \forall j \in \{1,\ldots,m\}\right\}
        \ \textrm{or}\ 
        \max_{\mathbf{r}\,\in\,\DomOne}\left\{\sum_{j=1}^{\numDistricts}f_j  \MID (\mathbf{r},f_j) \in \hyp(\brObj_j),\ \forall j \in \{1,\ldots,m\}\right\}.
        \end{equation*}
    Call these the \textit{Graph Reformulation} and the \textit{Hypograph Reformulation} respectively.
    \end{definition}

Each of the techniques presented here rely on identifying an efficient MILP approximation of either the graph or the hypograph. If the set $\Approx$ is such an approximation, then 
    $$
    \max_{\mathbf{r}\,\in\,\DomOne}\left\{\sum_{j=1}^{\numDistricts}f_j  \MID (\mathbf{r},f_j) \in \Approx\right\}
    $$
represents an approximation of \eqref{model:general}. In particular:
    \begin{proposition}\label{prop:GraphRelaxation}
    Suppose $\brObj_j: \DomOne \to \Reals$ for each $j \in M$ and consider a set $\Approx_j\subset\DomOne\!\times\!\Reals^{\numDistricts}$. If either $\gra(\brObj_j) \subseteq \Approx_j$ or $\hyp(\brObj_j) \subseteq \Approx_j$, then
        \begin{equation*}
        \max_{\mathbf{r}\,\in\,\DomOne}\sum_{j=1}^m \brObj_j(\mathbf{r}) 
            \LEQ  \max_{\mathbf{r}\,\in\,\XFeasPoly} \left\{ \sum_{j=1}^m f_j  \MID (\mathbf{r}, f_j) \in \Approx_j\right\}.
        \end{equation*}
    \end{proposition}
    \begin{proof}
    This result follows trivially from either the graph or hypograph reformulation.
    \end{proof}

Establishing an error guarantee for any such approximation is critical; we begin by considering the sources of error in our techniques.
    \begin{definition}[$\varepsilon$-Relaxation and $\varepsilon$-Approximation]
    Let $\mathscr{H}\subseteq\Reals^{n+1}$. 
    We say that $\Relax\supseteq \mathscr{H}$ is an \emph{$\varepsilon$-output-tolerance relaxation} (\emph{$\varepsilon$-relaxation} for short) of $\mathscr{H}$ if, for any $(\mathbf{r},f) \in \Relax$, there exists a value $f'\geq f-\epsilon$ such that $(\mathbf{r},f') \in \mathscr{H}$.

    We say that $\Approx\subseteq \Reals^{n+1}$ is an \emph{$\varepsilon$-output-tolerance approximation} (\emph{$\varepsilon$-approximation} for short) of $\mathscr{H}$ if, for any $(\mathbf{r},f) \in \Approx$, there exists a value $f' \in [f - \varepsilon, f + \varepsilon]$ such that $(\mathbf{r},f') \in \mathscr{H}$.
    \end{definition}
    \begin{figure}[!ht]
    \centering

    \includegraphics[scale = 0.7]{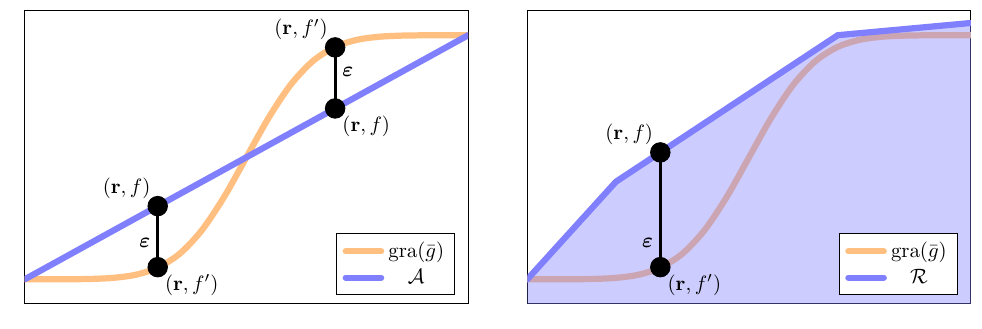}

    \caption{Left: $\Approx$ is an $\varepsilon$-approximation of $\gra(\brObj)$. Right: $\Relax$ is an $\varepsilon$-relaxation of $\gra(\brObj)$}
    \label{fig:enter-label}
    \end{figure}
Intuitively, $\Relax$ and $\Approx$ are sets with a maximum Hausdorff distance of $\epsilon$ from $\mathscr{H}$ in the last coordinate. We now establish a simple proposition related to aggregating these estimates.
    \begin{proposition}[Model Error]\label{prop:OverestimateError}
    \label{prop:UnderestimateError}
    Suppose $\brObj_j: \DomOne \to \Reals$ for each $j=\{1,\dots,m\}$ and consider a set $\Relax_j\subseteq\DomOne\!\times\!\Reals^{\numDistricts}$. If $\Relax_j$ is an $\varepsilon$-relaxation of either $\gra(\brObj_j)$ or $\hyp(\brObj_j)$, then
        \begin{equation*}
        \max_{\mathbf{r}\,\in\,\DomOne} \left\{ \sum_{j=1}^m f_j  \MID (\mathbf{r}, f_j) \in \Relax_j\right\}
        \GEQ
        \max_{\mathbf{r}\,\in\,\DomOne}\sum_{j=1}^m \brObj_j(\mathbf{r}) 
            \GEQ  \max_{\mathbf{r}\,\in\,\DomOne} \left\{ \sum_{j=1}^m f_j  \MID (\mathbf{r}, f_j) \in \Relax_j\right\} - \numDistricts\varepsilon.
        \end{equation*}
Similarly, if $\Approx_j$ is an $\varepsilon$-approximation, then 
        \begin{equation*}
        \max_{\mathbf{r}\,\in\,\DomOne} \left\{ \sum_{j=1}^m f_j  \MID (\mathbf{r}, f_j) \in \Approx_j\right\} + \numDistricts\varepsilon
        \GEQ
        \max_{\mathbf{r}\,\in\,\DomOne}\sum_{j=1}^m \brObj_j(\mathbf{r}) 
            \GEQ  \max_{\mathbf{r}\,\in\,\DomOne} \left\{ \sum_{j=1}^m f_j  \MID (\mathbf{r}, f_j) \in \Approx_j\right\} - \numDistricts\varepsilon.
        \end{equation*}
    \end{proposition}
    \begin{proof}
    Suppose first that $\Relax_j$ is an $\varepsilon$-relaxation of $\gra(\phi_j)$.  The upper bound is trivial, so we prove the lower bound. Suppose that $\mathbf{r}^*$ and $\mathbf{f}^*$ maximize $\sum_{j=1}^{\numDistricts}f_j$ while maintaining $(\mathbf{r}^*,f^*_j)\in\Relax_j$ for each $j\in\{1,\dots,\numDistricts\}$. There exists, by the definition of a $\varepsilon$-relaxation, a value $f'_j \geq f^*_j-\varepsilon$ such that $(\mathbf{r}^*,f'_j) \in \gra(\brObj_j)$; therefore,
        \begin{align*}
\max_{\mathbf{r}\,\in\,\DomOne}\sum_{j=1}^{\numDistricts}\brObj_j(\mathbf{r})
            &\EQ  \max_{\mathbf{r}\,\in\,\DomOne}\left\{\sum_{j=1}^{\numDistricts}f_j \MID(\mathbf{r},f_j) \in \gra(\brObj_j)\qquad\forall\ j\in\{1,\dots,\numDistricts\}\right\} \\
            &\GEQ  \sum_{j=1}^{\numDistricts}f'_j 
            \GEQ  \sum_{j=1}^{\numDistricts}\left(\overline{f}_j-\varepsilon\right)
            \EQ  \sum_{j=1}^{\numDistricts}\overline{f}_j-\numDistricts\varepsilon
        \end{align*}
    which establishes the lower bound. The logic is identical in the hypograph case.

    The subsequent $\varepsilon$-approximation bounds are proven similarly.
    \end{proof}

\noindent Proposition~\ref{prop:OverestimateError} allows us to characterize the approximation error in each of our techniques.

\section{Mixed Integer Approximations for Single Ratio Objectives} \label{sec:single-ratio}

This section describes several linearization approaches for objectives with a single ratio (i.e. Models \ref{model:BR-Redistricting} and \ref{model:HVAP}). We will focus on graph and hypograph relaxations of $\phi \colon \DomOne \to \R$ and $\psi \colon \DomTwo \to \R$ where $\DomOne \subseteq \R$, $\DomTwo \subseteq \R^2$, with
\begin{equation}\label{eq:AlternateExpressions}
    \DomOne \coloneq \left\{r=\frac{\yy}{\zz}\MID(\yy,\zz)\in\DomTwo\right\}
    \quad\textrm{and}\quad
    \psi(\yy,\zz) \coloneq \brObj\left(\frac{\yy}{\zz}\right).
    \end{equation}

We aim to identify strong approximations of the graph or hypograph of $\psi$:
    $$
    \gra(\psi) = \left\{(\yy,\zz,\ff)\in\DomTwo\!\times\!\Reals \MID \ff = \brObj\left(\frac{\yy}{\zz}\right)\right\}
    \quad\textrm{or}\quad
    \hyp(\psi) = \left\{(\yy,\zz,\ff)\in\DomTwo\!\times\!\Reals \MID \ff \leq \brObj\left(\frac{\yy}{\zz}\right)\right\}.
    $$
Additional properties of $\phi$ (increasing, bounded, etc) may be assumed in certain statements. We will then discuss approximation errors in these techniques when applied to our redistricting problems. 

We present approximations based on piecewise-linear and stepwise functions before presenting some radically different techniques: one based on a logarithmic-sized embedding of SOS2 constraints and one based on the SOCP approximation done by Ben-Tal and Nemirovski.

One approach that we do not describe is the \emph{Normalized Multiparametric Disaggregation Technique} (NMDT)~\cite{castro2015-nmdt}.  This approach has been useful in several classes of mixed integer bilinear quadratic programs, but was not effective for this problem.

 Their properties of each relaxation are summarized in this table, where $\ell$ denotes the number of breakpoints used in the approximation:
\begin{table}[ht]
    \centering
    \begin{tabular}{lccccc}
        \hline
        Relaxation & \# Binary Vars & \# Continuous Vars & Linear? & Relax/Approx & Error $(\varepsilon)$\\
        \hline
        PWL      & $\ell$ & $O(\ell)$ & \xm& $\varepsilon$-Approx  & $O\left(\tfrac{1}{\ell^2}\right)\footnotemark$\\
        VA-PWL   & $\ell$ & $O(\ell)$ &  \cm &$\varepsilon$-Approx       & data-dependent \\
        Step-Max & $\ell$ & $O(\ell)$   & \cm &$\varepsilon$-Relaxation      &  $O\left(\tfrac{1}{\ell}\right)$ \\
        Step-Exp.& $\ell$ & $O(\ell)$   & \cm &$\varepsilon$-Relaxation      &  $O\left(\tfrac{1}{\ell}\right)$ \\
        BN-Step   & $\ell$ & $O(\ell)$   & \cm &$\varepsilon$-Relaxation      &  $O\left(\tfrac{1}{2^\ell}\right)$ \\
        BN-Full   & $\ell$ & $O(\ell)$   & \cm &$\varepsilon$-Approx      &  $O\left(\tfrac{1}{2^\ell}\right)$ \\
        LogE & $\ell$ & $O(\log \ell)$   & \cm &$\varepsilon$-Approx      &  $O\left(\tfrac{1}{\ell}\right)$\\
        \hline
    \end{tabular}
    \caption{Each model uses $O(\ell)$ constraints and $O(\ell)$ continuous variables.}
    \label{tab:relaxations}
    \footnotetext{Assumes the approximated function is twice differentiable with a bounded second derivative.}
\end{table}

    \footnotetext{This holds provided the function has a bounded second derivative. But applying this bound assumes that the nonlinearity in the formulation is exactly handled.}
    We will later compare these approaches computationally.  Also, in Section~\ref{sec:extra-bounds}
we investigate strategies for adding additional inequalities to better describe the feasible region.

\subsection{Standard Piecewise-Linear Approximation}
Using traditional SOS2 constraints, it is simple enough to replace $\brObj$ with a piecewise-linear approximation over $\hypertarget{def:ell}{\numBreakpoints}$ breakpoints $\{b_1,\dots,b_\ell\}$. 

    \begin{definition}[Piecewise Linear Approximations]\label{def:PWLRelax}
    
    The following definitions follow standard constructions for piecewise-linear approximations; see, e.g., \cite{vielma2010} for a comprehensive survey.
    Given a finite increasing sequence of breakpoints $\{b_k\}_{k\in\breakpointsIter}$ the \textit{Piecewise-Linear Graph} and \textit{Hypograph Approximations} $\pwlg_{\{b_k\}}(\brObj)$ and $\pwlh_{\{b_k\}}(\brObj)$ of $\brObj$ are respectively given by:
        \begin{equation}\label{def:pwlgRelaxation:Set}
        \pwlg_{\{b_k\}}(\brObj)  \DEF  
            \left\{
            (r,f)\in \DomOne\!\times\!\Reals
            \MID
            \frac{f-\brObj(b_{k-1})}{\brObj(b_k)-\brObj(b_{k-1})}  =
            \frac{r-b_{k-1}}{b_k-b_{k-1}}
            \quad\textrm{if}\ 
            r\in[b_{k-1},b_k)
            \quad\forall\ k\in\breakpointsIter
            \right\},
        \end{equation}
        
        \begin{equation}\label{def:pwlhRelaxation:Set}
        \pwlh_{\{b_k\}}(\brObj)  \DEF  
            \left\{
            (r,f)\in  \DomOne\!\times\!\Reals
            \MID
            \frac{f-\brObj(b_{k-1})}{\brObj(b_k)-\brObj(b_{k-1})}  \leq
            \frac{r-b_{k-1}}{b_k-b_{k-1}}
            \quad\textrm{if}\ 
            r\in[b_{k-1},b_k)
            \quad\forall\ k\in\breakpointsIter
            \right\}.
        \end{equation}
    \end{definition}
    
    \begin{proposition}\label{lem:PiecewiseOverestimate}
    Suppose $\brObj(r) \in [b_0,b_\numBreakpoints]$ for all $r\in \DomOne$ and suppose that $\phi$ is monotonically increasing. 
    Then  $\pwlg_{\{b_k\}}(\brObj)$ is an
    $\varepsilon$-approximation) of $\gra(\brObj)$
    with 
        \begin{equation*}
        \varepsilon \leq \max_{k\,\in\,\breakpointsIter} \big\{\brObj(b_k) - \brObj(b_{k-1})\big\}.
        \end{equation*}
    The same relationships hold between $\pwlh_{\{b_k\}}(\brObj)$ and $\hyp(\brObj)$.
    \end{proposition}
This proposition follows from standard estimates on piecewise linear approximations of continuous functions.
Tighter estimates can be obtained by considering bounds on the derivatives of $\phi$. See Appendix~\ref{app:error}, Corollary~\ref{prop:taylor}.

We formulate the traditional SOS2 embedding of the piecewise linear approximation. Note that the SOS2 constraint is historically implemented using $\ell$ binary variables. 
In our computational experiments, we used Gurobi Optimizer~9.1.2 \cite{gurobi_release_9_1_2}, which implements SOS2 constraints using such a formulation (see Gurobi documentation).

    \begin{technique}{Piecewise Linear Approximation for Single Ratio Problems}{PWL}
    Identify a finite, increasing sequence of breakpoints $\{b_k\}_{k\in\breakpointsIter}$ 
 (abbreviated by $\{b_k\}$) such that $r\in[b_0,b_\numBreakpoints)$ for all $r\in\DomOne$ and define the following set:
        \begin{subequations}\label{tech:PWL:set}
        \begin{SetArray.}
            {\lowhat{\pwlg}_{\{b_k\}}\!\left(\psi\right)  \DEF }
            {(\yy,\zz,\ff,\lambdaVec) \in \DomTwo\!\times\!\Reals\!\times\![0,1]^\numBreakpoints}
                \ff &=  \textstyle{\sum_{k=1}^{\numBreakpoints}}\brObj(b_k)\lambdaVar_{k}  
                    \label{tech:PWL:set:objective}\\
                \yy &=  \zz\textstyle{\sum_{k=1}^{\numBreakpoints}}b_k\lambdaVar_{k} 
                    \label{tech:PWL:set:ratio}\\
                \lambdaVec  &\,\in\,  \text{SOS2}
                    \label{tech:PWL:set:SOS2}
        \end{SetArray.}
        \end{subequations}
   
    \end{technique}
    \begin{proposition}
         The projection to the space where $\tfrac{\yy}{\zz}\mapsto r$ satisfies the equivalence
\begin{equation}\label{tech:PWL:proj}
        \proj_{\left(r=\frac{\yy}{\zz},\,\ff\right)}\!\!\!\left(\lowhat{\pwlg}_{\{b_k\}}\!\left(\psi\right)\right)
        \EQ  \pwlg_{\{b_k\}}\!\left(\brObj\right).
        \end{equation}
    \end{proposition}
    \begin{proof} 
    Consider any particular point $(\overline{\yy},\overline{\zz})\in\DomTwo$ and identify the index $\overline{k}$ for which $\overline{r} = \tfrac{\overline{y}}{\overline{z}} \in \left[b_{\overline{k}-1},b_{\overline{k}}\right)$.
    Recall that \eqref{tech:PWL:set:SOS2} means that at most two consecutively indexed entries of $\lambdaVec$ may take non-zero values and that those values must sum to one. By \eqref{tech:PWL:set:ratio}, the two non-zero entries must be $\lambdaVar_{\overline{k}-1}$ and $\lambdaVar_{\overline{k}}$. Thus, \eqref{tech:PWL:set:ratio} and \eqref{tech:PWL:set:objective} collapse to
        \begin{align*}
        \begin{cases}
        \overline{r} \EQ  b_{\overline{k}-1}\left(1-\lambdaVar_{\overline{k}}\right)+b_{\overline{k}}\lambdaVar_{\overline{k}} \\
        \overline{f} \EQ  \phi\left(b_{\overline{k}-1}\right)\left(1-\lambdaVar_{\overline{k}}\right)+\phi\left(b_{\overline{k}}\right)\lambdaVar_{\overline{k}}
        \end{cases}
        &\Rightarrow\quad
        \begin{cases}
        \lambdaVar_{\overline{k}} = \frac{f-\brObj(b_{k-1})}{\brObj(b_k)-\brObj(b_{k-1})}  \\
        \lambdaVar_{\overline{k}} = \frac{r-b_{k-1}}{b_k-b_{k-1}}
        \end{cases} \\
        &\Rightarrow\quad
        \frac{f-\brObj(b_{k-1})}{\brObj(b_k)-\brObj(b_{k-1})}  = \frac{r-b_{k-1}}{b_k-b_{k-1}}.
        \end{align*}
    This establishes exactly the condition in \eqref{def:pwlgRelaxation:Set} so, upon projecting out the auxiliary variables $\lambdaVec$ and conjoining $\yy$ and $\zz$ into their ratio $r$, we are left with $\pwlg_{\{b_k\}}\!\left(\brObj\right)$.
    \end{proof}

    \begin{modelCustom}{1}{a}[Piecewise-Linear Approximation of BR-Redistricting]\label{labeling:PWL}
    We apply Approximation \ref{tech:PWL} to each district $j$ in Model \ref{model:BR-Redistricting} independently:
        \begin{equation*}
        \max_{(\x,\y,\z)\,\in\,\XYZFeasPoly_{\textrm{B}}}
        \left\{
        \sum_{j=1}^{\numDistricts}\ff_j
        \MID
        (\yy_j,\zz_j,\ff_j,\lambdaVec_j)\in\lowhat{\pwlg}_{\{b_k\}}(\psi)
        \quad
        \forall\ j\in\{1,\dots,\numDistricts\}
        \right\}.
        \end{equation*}
    \end{modelCustom}

Note that Model~\ref{labeling:PWL} does not necessarily produce an upper bound to Model \ref{model:BR-Redistricting} unless $\brObj$ is convex over its whole domain. Although we could perturb the PWL approximation to obtain such an upper bound; we found that Model \ref{labeling:PWL} performed poorly and did not pursue such variations.  In particular, the techniques in \cite{FRENZEN2010437,pwl_to_solve_minlp} could give alternative 1-dimensional piecewise-linear models.

Unfortunately, this common approximation technique does not address the nonconvexity of $\rr=\frac{\yy}{\zz}$. The placeholder variables $\y$ and $\z$ are, in practice, projected out in favor of their definitions from $\XYZFeasPoly_\textrm{B}$, but the nonconvexity remains. In our Gurobi implementation, we recast and discretize the first constraint of Approximation \ref{tech:PWL}. In particular, we introduce an additional quadratic constraint $\yy_j = \rr_j \zz_j$ so that the discrete variables $\rr_j$ can represent $\frac{\yy_j}{\zz_j}$ linearly. This auxiliary variable is made discrete by forcing $\rr_j = \frac{\rho}{1000}$ for $\rho \in \Z$.  Gurobi seems to employ different techniques when solving this problem, particularly if we give $\rr_j$ a high branching priority. 

\subsection{Linear approximation 
 of PWL: VA-PWL model}
 We consider a simplification of the PWL approach where we assume that the denominator $\zz$ is constant.  This increases the error in the approximation, but makes for a linear formulation by removing the need to approximate the quadratic constraints.
 
In our application, with general populations being roughly equal between districts, one might assume that the VAP of each district is roughly equal. This would allow for the restriction of $\zz_j$ to equal $\frac{1}{\numDistricts}\sum_{i\in\parcels}\VAP_i$, effectively linearizing the quadratic constraint. Such an approximation is questionable since a suitable $\varepsilon$ cannot be found, but some results for this VAP Approximated (VA-PWL) model are available in Section~\ref{sec:computations}.

\subsection{Stepwise Relaxation}
Given a monotone increasing objective function $\brObj$, the following stepwise approximation technique removes the need for quadratic constraints (unlike Model~\ref{labeling:PWL}) while simultaneously ensuring a strict upper bound. We first define some new notation and demonstrate its properties to support a general approach for relaxing the hypograph of $\brObj\left(\tfrac{\yy}{\zz}\right)$.

    \begin{definition}[Stepwise Approximation]\label{def:StepRelax}
    For a function $\brObj:\DomOne\rightarrow\Reals$ with $\DomOne\subseteq\Reals$ bounded and a finite sequence of increasing breakpoints $\{b_k\}_{k\in\breakpointsIter}$ the \textit{Stepwise Hypograph Approximation} $\step_{\{b_k\}}(\brObj)$ of $\brObj$ is given by:
        \begin{equation}\label{def:StepRelax:Set}
        \step_{\{b_k\}}(\brObj)  \DEF  
            \big\{
            (r,f)\in \DomOne\!\times\!\Reals
            \hquad\big|\ 
            f \leq \brObj(b_k) 
            \quad\textrm{if}\quad
            r\in[b_{k-1},b_k)
            \quad\forall k\in\breakpointsIter
            \big\}.
        \end{equation}
    \end{definition}
    
    \begin{proposition}\label{lem:StepOverestimate}
    If $\brObj$ is increasing over $\DomOne$ and $\DomOne\subseteq[b_0,b_\numBreakpoints]$, then $\step_{\{b_k\}}(\brObj)$ is a $\varepsilon$-relaxation of $\hyp(\brObj)$ with
        \begin{equation*}
        \hyp(\brObj) \subseteq \step_{\{b_k\}}(\brObj)
        \qquad\textrm{and}\qquad
        \varepsilon = \max_{k\,\in\,\breakpointsIter} \big\{\brObj(b_k) - \brObj(b_{k-1})\big\}.
        \end{equation*}
    \end{proposition}
    \begin{proof}
    Suppose, for the sake of contradiction, that there exists a point $(\overline{r},\overline{f})\in\hyp(\brObj)$ which does \textit{not} belong to $\step_{\{b_k\}}(\brObj)$ and identify the index $\overline{k}$ for which $\overline{r}\in[b_{\overline{k}-1},b_{\overline{k}})$. By its set membership, we have
        $$
        \brObj(b_{\overline{k}})  \LS  \overline{f}  \LEQ  \brObj(\overline{r})
        $$
    However, $\brObj$ is assumed to be increasing, so it should be that $\brObj(b_{\overline{k}}) \geq \brObj(\overline{r})$. This contradiction establishes the containment of $\hyp(\brObj)$ within $\step_{\{b_k\}}(\brObj)$.
    
    To verify the error term $\varepsilon$, we reverse our assumptions. Consider a point $(\overline{r},\overline{f})\in\step_{\{b_k\}}(\brObj)$ which does \textit{not} belong to $\hyp(\brObj)$; that is,
        $$
        \brObj(\overline{r})  \LS  \overline{f}  \LEQ  \brObj(b_{\overline{k}})
        $$
    for the index $\overline{k}$ such that $\overline{r}\in[b_{\overline{k}-1},b_{\overline{k}})$. Since $\brObj$ is increasing, it must be that 
        \begin{align*}
        \brObj(\overline{r})  \GEQ  \brObj(b_{\overline{k}-1}) 
            \EQ  \brObj(b_{\overline{k}-1}) + \brObj(b_{\overline{k}}) - \brObj(b_{\overline{k}}) 
            \EQ  \brObj(b_{\overline{k}}) - \varepsilon.
        \end{align*}
    \end{proof}

    We demonstrate an MILP embedding of the stepwise hypograph approximation.
    
    \vspace{1em}
    \begin{technique}{Stepwise Relaxation for Single Ratio Problems}{Step}
    Identify a finite, increasing sequence of breakpoints  $\{b_k\}_{k\in\breakpointsIter}$ such that $b_\numBreakpoints>r$ for all $r\in\DomOne$ and define the following set:
    \begin{subequations}\label{tech:Step:set}
    \vspace{1em}\\$\lowhat{\step}_{\{b_k\}}\!\left(\psi\right)  \DEF \vspace{-1em}$
    \begin{SetArray.}
    {\qqquad}
    {(\yy,\zz,\ff,\deltaVec) \in \DomTwo\!\times\!\Reals\!\times\!\{0,1\}^\numBreakpoints}
        \ff - \brObj\left(b_k\right)  &\leq  \bigN\left(1-\deltaVar_k\right) 
            & \forall\ k\in\breakpointsIter
            \label{tech:Step:set:objective}\\
        b_k\zz - \yy  &\leq  \bigM\deltaVar_k
            & \forall\ k\in\breakpointsIter
            \label{tech:Step:set:ratio}\\
        \deltaVar_{k-1}  &\leq  \deltaVar_{k} 
            & \forall\ k\in\breakpointsIter  
            \label{tech:Step:set:symmetry}
    \end{SetArray.}
\end{subequations}

    \end{technique}
    \begin{proposition}
        If $\bigM$ and $\bigN$ are sufficiently large,  we have     
        \begin{equation}\label{tech:Step_Projection}
        \proj_{\left(r=\frac{\yy}{\zz},\,\ff\right)}\!\!\!\left(\lowhat{\step}_{\{b_k\}}\!\left(\psi\right)\right)
        \EQ  \step_{\{b_k\}}\!\left(\brObj\right).
        \end{equation}
    \end{proposition}
    \begin{proof} 
    Consider any particular point $(\overline{\yy},\overline{\zz})\in\DomTwo$ and identify the index $\overline{k}$ for which $\overline{r} = \tfrac{\overline{y}}{\overline{z}} \in \left[b_{\overline{k}-1},b_{\overline{k}}\right)$. Notice that, by \eqref{tech:Step:set:ratio}, the indicator variable $\deltaVar_k$ must take the value one for all $k\geq \overline{k}$ while those with smaller indices are free to take the value zero. Given this, \eqref{tech:Step:set:objective} reduces to $\ff \leq \brObj(b_{\overline{k}})$ since $\brObj$ is increasing. This establishes exactly the condition in \eqref{def:StepRelax:Set} so that, upon projecting out the indicator variables $\deltaVec$ and conjoining $\yy$ and $\zz$ into their ratio $r$, we are left with $\step_{\{b_k\}}\!\left(\brObj\right)$. Notice that \eqref{tech:Step:set:symmetry} is ultimately redundant, we include it to reduce the size of the branch-and-bound tree.
    \end{proof}

    \begin{modelCustom}{1}{b}[Stepwise Approximation of BR-Redistricting]\label{labeling:Step}
    We apply Approximation \ref{tech:Step} to each district $j$ in Model \ref{model:BR-Redistricting} independently:
        \begin{equation*}
        \max_{(\x,\y,\z)\,\in\,\XYZFeasPoly_\textrm{B}}
        \left\{
        \sum_{j=1}^{\numDistricts}f_j
        \MID
        (\yy_j,\zz_j,\ff_j,\deltaVec^j)\in \lowhat{\step}_{\{b_j\}}\!\left(\brObj\right)
        \quad
        \forall\ j\in\{1,\dots,\numDistricts\}
        \right\}
    \end{equation*}
    In our implementations, the placeholder variables $\y$ and $\z$ are projected out in favor of their definitions from $\XYZFeasPoly_B$. 
    \end{modelCustom}
    \begin{corollary}
    Model~\ref{labeling:Step} is an upper bound on Model~\ref{model:BR-Redistricting} with a maximum error of: 
        $$
        \sum_{j=1}^{\numDistricts}\left(\max_{k\,\in\,\breakpointsIter} \big\{\brObj(b_{j,k}) - \brObj(b_{j,k-1})\big\}\right).
        $$
    \end{corollary}
    \begin{proof}
    Since $\brObj$ is assumed to be increasing over $\DomOne$, this follows from \eqref{tech:Step_Projection} by Proposition~\ref{prop:GraphRelaxation}, Proposition~\ref{prop:OverestimateError}, and Proposition~\ref{lem:StepOverestimate}.
    \end{proof}

Although stepwise functions are traditionally considered a weaker approximation, this technique sidesteps nonconvexity that a direct piecewise linear model inherits (see Table~\ref{tab:PWLvStepComparison}). We define Step-Max and Step-Exp. models that differ only in the choice of breakpoints, which is our next topic of discussion.

We chose to use distinct values of $\bigN$ and $\bigM$ for each district $j$ and each breakpoint index $k$. In each of our redistricting settings we have $\yy_j \geq 0$ and $\zz_j \leq (1+\tau)\tildep$; thus $\brObj\left(b_{j\numBreakpoints}\right) - \brObj\left(b_{jk}\right)$ is a valid value of $\bigN_{jk}$ while $b_{jk}(1+\tau)\tildep$ is a sufficiently large value of $\bigM_{jk}$. Knowing that $\frac{\yy_j}{\zz_j} \in [0,1]$ for every $j\in\{1,\dots,\numDistricts\}$ and $(\x,\y,\z)\in\XYZFeasPoly_\textrm{B}$, it can be convenient to set $b_{j\numBreakpoints} = 1$; for notational convenience, we also define $b_{j0} = 0$ for each $j\in\{1,\dots,\numDistricts\}$. Finding larger values for $b_{j0}$ and smaller values for $b_{j\numBreakpoints}$ which maintain $b_{j0}  \leq  \frac{\yy_j}{\zz_j}  \leq  b_{j\numBreakpoints}$ for each $j\in\{1,\dots,\numDistricts\}$ and $(\x,\y,\z)\in\XYZFeasPoly_\textrm{B}$ is a good way to increase breakpoint density and improve the resulting bound. We discuss breakpoint choice in more detail in Sections \ref{sub:optBreakpointChoices} and \ref{sub:BoundingRatios}.

\subsection{Optimizing breakpoint choices for Approximation \ref{tech:Step}}\label{sub:optBreakpointChoices}
We assume here that $\phi$ is strictly increasing.  Hence, the inverse function $\phi^{-1}$ is well defined on the range of $\phi$.

Given some $(\x,\y,\z)\in\XYZFeasPoly_\textrm{B}$, it is easy to compute the exact error of Model \ref{labeling:Step} in approximating Model \ref{model:BR-Redistricting}:
    $$
    \textrm{error}(\y,\z)  \EQ  \sum_{j=1}^{\numDistricts}\left(\ff_j-\brObj\left(\tfrac{\yy_j}{\zz_j}\right)\right)  \EQ  \sum_{j=1}^{\numDistricts}\left(\brObj(b_{j\overline{k}})-\brObj\left(\tfrac{\yy_j}{\zz_j}\right)\right)
    $$
where $b_{j\overline{k}}$ is the smallest breakpoint greater than $\tfrac{\yy_j}{\zz_j}$. We propose two methods for choosing breakpoints with respect to this error. 

\paragraph{Minimizing Maximum Error (Step-Max)}
Choosing breakpoints that induce a small maximum error is an obvious way to ensure the quality of the approximation in Technique \ref{tech:Step}. Having already identified an expression for maximum error, we proceed to demonstrate the breakpoint selection which minimizes maximum error.

\begin{proposition}
    When implementing Technique \ref{tech:Step}, distributing the breakpoints evenly across the range of $\brObj$ will give an approximation that minimizes maximum error. In particular, this minimum is achieved by the breakpoints defined
        $
        b_{jk}^* = \brObj^{-1}\left(\brObj\left(b_{j0}\right)+k\varepsilon_j\right)
        $
    where
        $
        \varepsilon_j = \frac{\brObj(b_{j\numBreakpoints})-\brObj(b_{j0})}{\numBreakpoints},
        $
    and $b_{j0}$ and $b_{j\numBreakpoints}$ are respectively the infimum and supremum of $\DomOne_j = \left\{\tfrac{\yy_j}{\zz_j}\ :\ (\x,\y,\z)\in\XYZFeasPoly\right\}$.
\end{proposition}
\begin{proof}
    Recall that $\brObj$ is assumed to be increasing under Model \ref{labeling:Step} so $\brObj\big(b_{j\numBreakpoints}\big)$ and $\brObj\big(b_{j0}\big)$ are the respective max and min of the range of $\brObj$ for each district $j$. Thus, $\varepsilon_j$ yields the equidivision of the range into $\numBreakpoints$ parts. 

    By the construction of $b_{jk}^*$, we have
        $
        \brObj(b_{jk}^*)-\brObj\left(b_{j\,(k-1)}^*\right) = \varepsilon_j^*
        $
    for every $k \in \breakpointsIter$ and $j \in \districts$. Thus, the maximum error of this construction is $\sum_{j=1}^{\numDistricts}\varepsilon_j^*$. Suppose, for the sake of contradiction, that there exists another sequence $\big\{b_{jk}^\dagger\big\}_{k\in[\![\numBreakpoints]\!]}$ of breakpoints which induces a maximum error strictly less than $\sum_{j=1}^{\numDistricts}\varepsilon_j^*$. There must exist an index $\jfixed\in\{1,\dots,\numDistricts\}$ for which 
        $
        \brObj\left(b_{\jfixed k}^\dagger\right)-\brObj\left(b_{\jfixed\,(k-1)}^\dagger\right) < \varepsilon_{\jfixed}^*.
        $
    for each $k\in\breakpointsIter$. That is
        \begin{align*}
        \brObj\big(b_{\jfixed\numBreakpoints}^\dagger\big)
            &\EQ \brObj(b_{\jfixed0}) + \sum_{k=1}^{\numBreakpoints}\left(\brObj\big(b_{\jfixed k}^\dagger\big) - \brObj_{\jfixed}\left(b_{\jfixed\,(k-1)}^\dagger\right)\right) \\
            &\LS \inf(\DomOne_j) + \numBreakpoints\varepsilon_{\jfixed}^* 
            \EQ  \brObj\big(\sup(\DomOne_j)\big)
        \end{align*}
    which contradicts our earlier assertion that $b_{j\numBreakpoints} \geq \sup(\DomOne_j)$. 
\end{proof}
We refer to the stepwise model with breakpoints spread uniformly over the range as Step-Max since it has minimal maximum error for a given number of breakpoints $\ell$. Some computational results are given in Table \ref{tab:PWLvStepComparison} on page \pageref{tab:PWLvStepComparison}.

\paragraph{Minimizing Expected Error (Step-Exp.)}
Suppose that the values of $\tfrac{\yy_j}{\zz_j}$ are drawn randomly from its domain $\DomOne_j$ according to some CDF $P_j\left(\tfrac{\yy_j}{\zz_j}\right)$. We can use the distribution to compute the expected error of an approximation. We were unable to find an analytic solution to the problem of minimizing the expected error for our objectives, but SciPy's~\cite{2020SciPy-NMeth} optimize function seems to solve this problem quickly (numerically) over a uniform distribution $\left(\text{i.e. }P_j(\tfrac{\yy_j}{\zz_j}) = \frac{\tfrac{\yy_j}{\zz_j}-b_{j0}}{b_{j\numBreakpoints}-b_{j0}}\right)$. To be specific, we feed the \verb|scipy.optimize.minimize| command the objective
    $$
    \sum_{j=1}^{\numDistricts}\sum_{k=1}^{\numBreakpoints} \brObj(b_{jk})\left(b_{jk}-b_{j\,(k-1)}\right)
    $$
with the constraints $b_{j0} = \inf(\DomOne_j)$ and $b_{j\numBreakpoints} = \sup(\DomOne_j)$ alongside the gradient
    $$
    \brObj'(b_{jk})\left(b_{jk}-b_{j\,(k-1)}\right)+\brObj(b_{jk})-\brObj\left(b_{j\,(k+1)}\right)
    $$
and an initial set of breakpoints spread evenly over the domain.

We implement a stepwise model with breakpoints chosen by SciPy to minimize expected error over a uniform distribution and call it the Step-Exp model. Some computational results are given in Table \ref{tab:PWLvStepComparison} on page \pageref{tab:PWLvStepComparison}.

\begin{remark}
For specific choices of $\brObj$, there may be elegant closed-form solutions for the breakpoints.   For instance, if $\brObj$ is linear, then the minimum expected error breakpoints are uniformly distributed. 
\end{remark}

    \begin{figure}[!ht]
    \centering
        \begin{subfigure}{0.485\textwidth}
        \centering
       
\begin{tikzpicture}
\begin{axis}[
xticklabel=\empty,
yticklabel=\empty,
tick style={draw=none},
legend entries={$\gra(\brObj)$,$\step\!\left(\brObj\right)$},
legend pos=south east,
]
\addplot [smooth, line width=3pt, orange!50, line cap=round]
table {%
0 0.00234931639457267
0.02 0.00356746558185301
0.04 0.00532527455585051
0.06 0.00781505529955336
0.08 0.01127659061793
0.1 0.0160004251340584
0.12 0.0223280224771477
0.14 0.0306475996293951
0.16 0.0413845745115257
0.18 0.0549858892889104
0.2 0.0718980120615618
0.22 0.0925391499624855
0.24 0.117267061258748
0.26 0.146344725900874
0.28 0.179906886171658
0.3 0.217930955058255
0.32 0.260215880552359
0.34 0.306372166403236
0.36 0.355825371524136
0.38 0.407834112474073
0.4 0.461522028248008
0.42 0.515921546941505
0.44 0.570025858947631
0.46 0.622844473981992
0.48 0.673457283763879
0.5 0.721062242578815
0.52 0.765012583210308
0.54 0.804840776017323
0.56 0.840268011002484
0.58 0.871199598792277
0.6 0.897708117832569
0.62 0.920007201100905
0.64 0.938419450164447
0.66 0.953342065915221
0.68 0.965213449834059
0.7 0.974483370502769
0.72 0.981588450176675
0.74 0.986933849428018
0.76 0.990881235454448
0.78 0.993742495286192
0.8 0.995778239629907
0.82 0.997199937811329
0.84 0.998174500550846
0.86 0.998830238643764
0.88 0.999263319742254
0.9 0.999544073929743
0.92 0.99972272328036
0.94 0.999834305454005
0.96 0.99990271336966
0.98 0.999943879121587
};
\addplot [smooth, line width=3pt, blue!50, line cap=round]
coordinates {%
(0, 0.0995143063504609)
(0.226, 0.0995143063504609)
};
\addplot [smooth, line width=3pt, blue!50, line cap=round]
table {%
0.226 0.200276502615947
0.291 0.200276502615947
};
\addplot [smooth, line width=3pt, blue!50, line cap=round]
table {%
0.291 0.299222499068621
0.337 0.299222499068621
};
\addplot [smooth, line width=3pt, blue!50, line cap=round]
table {%
0.337 0.399903070951487
0.377 0.399903070951487
};
\addplot [smooth, line width=3pt, blue!50, line cap=round]
table {%
0.377 0.499586695871437
0.414 0.499586695871437
};
\addplot [smooth, line width=3pt, blue!50, line cap=round]
table {%
0.414 0.599296268461439
0.451 0.599296268461439
};
\addplot [smooth, line width=3pt, blue!50, line cap=round]
table {%
0.451 0.700057536212316
0.491 0.700057536212316
};
\addplot [smooth, line width=3pt, blue!50, line cap=round]
table {%
0.491 0.79914242890351
0.537 0.79914242890351
};
\addplot [smooth, line width=3pt, blue!50, line cap=round]
table {%
0.537 0.9001228699136
0.602 0.9001228699136
};
\addplot [smooth, line width=3pt, blue!50, line cap=round]
table {%
0.602 0.999967264848155
0.999 0.999967264848155
};
\end{axis}

\end{tikzpicture}

        \caption{Minimizes Maximum Error}
        \end{subfigure}
    ~
        \begin{subfigure}{0.485\textwidth}
        \centering
        
\begin{tikzpicture}
\begin{axis}[
xticklabel=\empty,
yticklabel=\empty,
tick style={draw=none},
legend entries={$\gra(\brObj)$,$\step\!\left(\brObj\right)$},
legend pos=south east,
]
\addplot [smooth, line width=3pt, orange!50, line cap=round]
table {%
0 0.00234931639457267
0.02 0.00356746558185301
0.04 0.00532527455585051
0.06 0.00781505529955336
0.08 0.01127659061793
0.1 0.0160004251340584
0.12 0.0223280224771477
0.14 0.0306475996293951
0.16 0.0413845745115257
0.18 0.0549858892889104
0.2 0.0718980120615618
0.22 0.0925391499624855
0.24 0.117267061258748
0.26 0.146344725900874
0.28 0.179906886171658
0.3 0.217930955058255
0.32 0.260215880552359
0.34 0.306372166403236
0.36 0.355825371524136
0.38 0.407834112474073
0.4 0.461522028248008
0.42 0.515921546941505
0.44 0.570025858947631
0.46 0.622844473981992
0.48 0.673457283763879
0.5 0.721062242578815
0.52 0.765012583210308
0.54 0.804840776017323
0.56 0.840268011002484
0.58 0.871199598792277
0.6 0.897708117832569
0.62 0.920007201100905
0.64 0.938419450164447
0.66 0.953342065915221
0.68 0.965213449834059
0.7 0.974483370502769
0.72 0.981588450176675
0.74 0.986933849428018
0.76 0.990881235454448
0.78 0.993742495286192
0.8 0.995778239629907
0.82 0.997199937811329
0.84 0.998174500550846
0.86 0.998830238643764
0.88 0.999263319742254
0.9 0.999544073929743
0.92 0.99972272328036
0.94 0.999834305454005
0.96 0.99990271336966
0.98 0.999943879121587
};
\addplot [smooth, line width=3pt, blue!50, line cap=round]
table {%
0 0.0291088240460173
0.13667514421081 0.0291088240460173
};
\addplot [smooth, line width=3pt, blue!50, line cap=round]
table {%
0.13667514421081 0.0910179045753113
0.218647386569308 0.0910179045753113
};
\addplot [smooth, line width=3pt, blue!50, line cap=round]
table {%
0.218647386569308 0.18266141583056
0.281531040145548 0.18266141583056
};
\addplot [smooth, line width=3pt, blue!50, line cap=round]
table {%
0.281531040145548 0.296343351654879
0.33578278407067 0.296343351654879
};
\addplot [smooth, line width=3pt, blue!50, line cap=round]
table {%
0.33578278407067 0.424379895534015
0.386214315582847 0.424379895534015
};
\addplot [smooth, line width=3pt, blue!50, line cap=round]
table {%
0.386214315582847 0.559240563759234
0.435986529775246 0.559240563759234
};
\addplot [smooth, line width=3pt, blue!50, line cap=round]
table {%
0.435986529775246 0.693275963920814
0.488156710780103 0.693275963920814
};
\addplot [smooth, line width=3pt, blue!50, line cap=round]
table {%
0.488156710780103 0.818338691639527
0.547326709053263 0.818338691639527
};
\addplot [smooth, line width=3pt, blue!50, line cap=round]
table {%
0.547326709053263 0.924898194713199
0.624935929183523 0.924898194713199
};
\addplot [smooth, line width=3pt, blue!50, line cap=round]
table {%
0.624935929183523 0.999968194659946
1 0.999968194659946
};
\end{axis}

\end{tikzpicture}

        \caption{Minimizes Expected Error}
        \end{subfigure}
    \caption{A visual comparison of breakpoint selection methods}
    \label{fig:breakpointComparison}
    \end{figure}

\subsection{Dominating Inequalities for Technique~\ref{tech:Step}} 

We can strengthen the MIP formulation of $\lowhat{\step}_{\{b_k\}}$ from Technique~\ref{tech:Step} by replacing constraint \eqref{tech:Step:set:objective} \textemdash namely,
\[
\ff - \brObj\left(b_k\right) \leq \bigN\left(1 - \deltaVar_k\right),
\]
with a single, tighter inequality, provided that $\phi$ is an increasing function.

This strengthening is possible due to the structure of the binary vector $\deltaVec$ imposed by constraint \eqref{tech:Step:set:symmetry}, which ensures that $\deltaVar = (0, 0, \dots, 0, 1, 1, \dots, 1)$ for some transition index. Under this structure, we can express a bound on $\ff$ in terms of a telescoping sum that reflects the monotonicity of $\phi$:

\begin{equation}
\label{eq:dominatinh_ineq}
\ff \LEQ \brObj(b_{\numBreakpoints}) - \sum_{k=1}^{\numBreakpoints-1} \left(\brObj(b_{k+1}) - \brObj(b_k)\right) \deltaVar_k.
\end{equation}

This inequality dominates the family of individual constraints in \eqref{tech:Step:set:objective} and improves the tightness of the relaxation while preserving correctness.

\begin{lemma}
Inequality~\eqref{eq:dominatinh_ineq} strictly dominates \eqref{tech:Step:set:objective}. 
\end{lemma}

\begin{proof}
Given any particular breakpoint index $\overline{k} \in \breakpointsIter$, \eqref{tech:Step:set:symmetry} implies that $\delta_{k} \geq \delta_{\overline{k}}$ for $k > \overline{k}$. Thus \eqref{eq:dominatinh_ineq} reduces to:
    \begin{align*}
    \ff  &\LEQ  \brObj(b_{\numBreakpoints}) - \sum_{k=1 }^{\numBreakpoints-1}\big(\brObj(b_{k+1}) - \brObj(b_k)\big) \deltaVar_k 
        \LEQ \brObj(b_{\numBreakpoints}) - \sum_{k=\overline{k} }^{\numBreakpoints-1}\big(\brObj(b_{k+1}) - \brObj(b_k)\big) \deltaVar_{\overline{k}} \\
        &\EQ  \brObj(b_{\overline{k}}) + \brObj(b_\numBreakpoints) - \brObj(b_{\overline{k}}) - \left(\brObj(b_\numBreakpoints) - \brObj(b_{\overline{k}})\right) \deltaVar_{\overline{k}} \\
        &\LEQ  \brObj(b_{\overline{k}}) + \big(\brObj(b_{\numBreakpoints}) - \brObj(b_{\overline{k}})\big) (1-\deltaVar_{\overline{k}})  
        \LEQ  \brObj(b_{\overline{k}}) + \bigM_{k} (1-\deltaVar_{\overline{k}})
    \end{align*}
because $\brObj(b_{\numBreakpoints}) - \brObj(b_{k})$ is the minimum valid value for $\bigM_{k}$ in \eqref{tech:Step:set:objective}.

Next, we demonstrate strict domination: assume, without loss of generality, that the approximation error is better than $\tfrac{1}{2}$ and let $\overline{k}$ be the smallest index such that $\brObj(b_{\overline{k}}) + \tfrac{1}{2} \geq \brObj(b_{\numBreakpoints})$. Suppose that $\deltaVar_{k} = 0$ for all $k < \overline{k}$ and $\deltaVar_{k} = 1$ for all $k \geq \overline{k}$. Then Model~\ref{labeling:Step} is vacuous for all $k\in \breakpointsIter$. But, \eqref{eq:dominatinh_ineq} produces a nontrivial bound on $\ff_j$.
\end{proof}

By replacing \eqref{tech:Step:set:objective} and adding  inequality ~\eqref{eq:dominatinh_ineq} to the description of $\lowhat{\step}_{\{b_k\}}$, we tighten the formulation, which in theory should improve computational results.  However, in practice, we did not see significant computational gains, perhaps because this inequality is rather dense.

In the context of mixed integer formulations for piecewise linear functions, this inequality relates to the \emph{incremental model}, which has strong tightness properties. See~\cite{vielma2010} for results on its theoretical properties and discussion therein of prior references.

\section{Logarithmic Size Approximations for Single Ratio Objectives} 
\label{sec:LogSize-single-ratio}

Models \ref{labeling:PWL} and \ref{labeling:Step} both have a binary variable for every linear segment of the approximation; Model \ref{labeling:PWL}'s binary variables are embedded within the typical SOS2 implementation. However, there exist embedding techniques for which the number of binary variables is \textit{logarithmic} in the number of pieces. In this section, we describe two such techniques and apply them to create compact approximations of Model \ref{model:BR-Redistricting}. 

Instead of labeling each linear segment with a single binary variable as in the prior approaches, the forthcoming techniques will use binary \textit{vectors} to label each segment. It is typically favorable to choose this labeling according to the (reflective) Gray code. 

\subsubsection*{Gray Code}

Gray code is a binary encoding of integers with the property that any pair of successive codewords differs in exactly one bit. While this construction is well known, we provide a formal definition for completeness, as it arises naturally in Approximation~\ref{tech:BNFull}.

Following~\cite{Beach2020-compact}, we focus on reflective Gray code and introduce convenient notation for working with it. Let each integer be represented as a binary vector in $\{0,1\}^*$. The standard binary representation using $\numBreakpoints$ bits maps an integer $q \in \{0, \ldots, 2^\numBreakpoints - 1\}$ to a vector $\alphabm^q \in \{0,1\}^\numBreakpoints$ such that
\begin{equation}
q = \sum_{k = 1}^{\numBreakpoints} 2^{\numBreakpoints-k} \alpha^q_k.
\label{eq:bindef}
\end{equation}

In contrast, Gray code encoding intentionally ensures that each adjacent integer differs by exactly one bit. The $\numBreakpoints$-bit reflective Gray code $\omegabm^q \in \{0,1\}^\numBreakpoints$ corresponding to $q$ is defined recursively as
\begin{equation}\label{eq:gray_code_defn}
\omega^q_1 \DEF \alpha^q_1, \qquad
\omega^q_k \DEF (\alpha^q_k + \alpha^q_{k-1}) \bmod 2 \quad \text{for } k = 2, \dots, \numBreakpoints,
\end{equation}
where $\bmod 2$ denotes addition modulo 2 (i.e., bitwise XOR).

This recurrence leads to a simple inverse transformation:
\begin{equation}
\label{eq:alternative-definition}
\alpha^q_k \EQ \left( \sum_{j=1}^k \omega^q_j \right) \bmod 2.
\end{equation}
Hence, there is a bijection $\mathscr{Q} \colon \{0,1\}^\numBreakpoints \to \{0, 1, \dots, 2^\numBreakpoints - 1\}$ that maps a Gray code vector $\deltabm \in \{0,1\}^\numBreakpoints$ to its corresponding integer:
\begin{equation}\label{Gray2Int}
\mathscr{Q}(\deltabm) = \sum_{k=1}^{\numBreakpoints} 2^{\numBreakpoints-k} \left( \sum_{t=1}^k \delta_t \right) \bmod 2.
\end{equation}

\subsection{Logarithmic Embedding (LogE)}
Inspired by the work of Vielma et al.~\cite{Vielma2009}, our approach in this section builds on the LogE method.  This method describes a formulation for a multi-dimensional piecewise linear function.   As in Approximation \ref{tech:PWL}, we embed this in the $(\yy, \zz)$ space.  In this method, the domain is covered by a sequence of \emph{covering polyhedra} $\DomTwo_{\boldsymbol{\omega}}$, indexed by binary codes $\boldsymbol{\omega}$.  Although many choices exist for this formulation, by selecting Gray codes for the indices and sequencing these polyhedra in accordance with the corresponding integers, \cite{Vielma2009} shows that the corresponding integer programming formulation is strong. 

Recall that Approximations \ref{tech:PWL} and \ref{tech:Step} both use an increasing sequence of breakpoints $\{b_k\}$ in the ratio $\rr$ of $\yy$ to $\zz$. These breakpoints are convenient in constructing our polyhedra $\DomTwo_{\boldsymbol\omega}$.

We chose our covering polyhedra to be \emph{polygonal annular sectors}.  We identify two radii $\check{s}$ and $\hat{s}$ such that $\DomTwo\subseteq\left\{(\yy,\zz) \MID \check{s}^2\leq\yy^2+\zz^2\leq\hat{s}^2\right\}$. Each cover polyhedron $\DomTwo_{\boldsymbol\omega}$ is then the convex hull of the four points with radius $\check{s}$ or $\hat{s}$ and ratio $\tfrac{\yy}{\zz} = b_{\bar k}$ or $ \tfrac{\yy}{\zz} = b_{\bar k -1}$. Choosing $\boldsymbol{\omega}$ such that $Q(\boldsymbol{\omega}) = \bar k$, we can write $\DomTwo_{\boldsymbol{\omega}}$  in Cartesian coordinates
\begin{equation}
\label{eq:covering}
\DomTwo_{\boldsymbol\omega} \EQ \text{conv}\left\{
\sqrt{\check{s}}\,\mathbf{d}^{\bar{k}-1},\,
\sqrt{\hat{s}}\,\mathbf{d}^{\bar{k}-1},\,
\sqrt{\check{s}}\,\mathbf{d}^{\bar{k}},\,
\sqrt{\hat{s}}\,\mathbf{d}^{\bar{k}}
\right\},
\quad\text{where}\quad
\mathbf{d}^{k} \DEF \frac{1}{\sqrt{1 + b_k^2}}\,\left(  b_k, \, 1 \right).
\end{equation}

In the two-dimensional space of $\DomTwo$, each breakpoint $b_k$ appears as a ray through the origin rather than as single points in $\DomOne$; see Figure \ref{fig:Graycodes:RegionBreakpoints}.

    \begin{figure}[!ht]
    \centering
        \begin{subfigure}{.48\textwidth}
\begin{center}
\begin{tikzpicture}

\colorlet{color1}{yellow!20!white}
\colorlet{color2}{red!20!white}
\colorlet{color3}{blue!20!white}

\draw[thick,fill=color1] ($(0,0) + (5:4)$) arc (5:45:4) -- ($(0,0) + (45:1)$) arc (45:5:1) -- cycle;
\draw[fill=color2] ($(0,0) + (5:1)$) -- ($(0,0) + (45:1.5)$) --  ($(0,0) + (45:4)$) --  ($(0,0) + (5:3.3)$) -- cycle;
\draw[thick,->] (0,0) -> ($(0,0) + (45:4.5)$) node[anchor=west, rotate=45]{$\frac{\yy}{\zz} = b_{4}$};
\draw[thick,->] (0,0) -> ($(0,0) + (5:4.5)$) node[anchor=west, rotate=5]{$\frac{\yy}{\zz} = b_{0}$};

\draw[->] (0,0) -> ($(0,0) + (15:4.2)$) node[anchor=west, rotate=15]{$\frac{\yy}{\zz} = b_{1}$};
\draw[->] (0,0) -> ($(0,0) + (25:4.2)$) node[anchor=west, rotate=25]{$\frac{\yy}{\zz} = b_{2}$};
\draw[->] (0,0) -> ($(0,0) + (35:4.2)$) node[anchor=west, rotate=35]{$\frac{\yy}{\zz} = b_{3}$};

\draw[thick] ($(0,0) + (5:4)$) arc (5:48:4) node[anchor=south, rotate=20]{$\hat{s}$};
\draw[thick] ($(0,0) + (5:1)$) arc (5:55:1) node[anchor=south, rotate=20]{$\check{s}$};

\draw[->] (-.2,0) -> (5.5,0) node[below right]{$\zz$};
\draw[->] (0,-.2) -> (0,4) node[above left]{$\yy$};

\end{tikzpicture}
\end{center}

        \caption{$\DomTwo$ (represented in pink) is approximated between two scaled arcs about the origin. Each breakpoint $b_{k}$, being a particular ratio of $\yy$ to $\zz$, is represented by a ray through the origin.}
        \label{fig:Graycodes:RegionBreakpoints}
        \end{subfigure}
        \hfill
        \begin{subfigure}{.48\textwidth}

\begin{center}
\begin{tikzpicture}

\colorlet{color1}{yellow!20!white}
\colorlet{color2}{red!20!white}
\colorlet{color3}{blue!20!white}

\draw[thick,fill=color1] ($(0,0) + (5:4)$) arc (5:45:4) -- ($(0,0) + (45:1)$) arc (45:5:1) -- cycle;
\draw[fill=color2] ($(0,0) + (5:1)$) -- ($(0,0) + (45:1.5)$) --  ($(0,0) + (45:4)$) --  ($(0,0) + (5:3.3)$) -- cycle;
\draw ($(0,0) + (5:.7)$) -> ($(0,0) + (5:4.3)$);
\draw ($(0,0) + (15:.7)$) -> ($(0,0) + (15:4.3)$);
\draw ($(0,0) + (25:.7)$) -> ($(0,0) + (25:4.3)$);
\draw ($(0,0) + (35:.7)$) -> ($(0,0) + (35:4.3)$);
\draw ($(0,0) + (45:.7)$) -> ($(0,0) + (45:4.3)$);

\draw ($(0,0) + (40:4.2)$) node[anchor=west, rotate=40]{$\boldsymbol\omega_{4} = \big[\negthinspace\begin{smallmatrix} 1 \\ 0 \end{smallmatrix}\negthinspace\big]$};
\draw ($(0,0) + (30:4.2)$) node[anchor=west, rotate=30]{$\boldsymbol\omega_{3} = \big[\negthinspace\begin{smallmatrix} 1 \\ 1 \end{smallmatrix}\negthinspace\big]$};
\draw ($(0,0) + (20:4.2)$) node[anchor=west, rotate=20]{$\boldsymbol\omega_{2} = \big[\negthinspace\begin{smallmatrix} 0 \\ 1 \end{smallmatrix}\negthinspace\big]$};
\draw ($(0,0) + (10:4.2)$) node[anchor=west, rotate=10]{$\boldsymbol\omega_{1} = \big[\negthinspace\begin{smallmatrix} 0 \\ 0 \end{smallmatrix}\negthinspace\big]$};

\filldraw[fill=blue!20, draw=black] 
    ($(0,0) + (25:1)$) -- 
    ($(0,0) + (25:4)$) -- 
    ($(0,0) + (35:4)$) -- 
    ($(0,0) + (35:1)$) -- 
    cycle;
\filldraw[fill=blue!20, draw=black] 
    ($(0,0) + (25:1.13)$) -- 
    ($(0,0) + (25:3.4)$) -- 
    ($(0,0) + (35:3.62)$) -- 
    ($(0,0) + (35:1.27)$) -- 
    cycle;

\coordinate (a) at ($(0,0) + (28:2.7)$);
\fill[blue, rotate=0,scale=0.65] (a) ++(-4pt,0) -- ++(4pt,4pt) -- ++(4pt,-4pt) -- ++(-4pt,-4pt) -- cycle;
\draw ($(0,0) + (20:2.2)$) node[anchor=west, rotate=20, scale=1]{$(\overline{\zz},\overline{\yy})$};
\fill[radius=2pt] ($(0,0) + (25:1)$) circle;
\fill[radius=2pt] ($(0,0) + (25:4)$) circle;
\fill[radius=2pt] ($(0,0) + (35:1)$) circle;
\fill[radius=2pt] ($(0,0) + (35:4)$) circle;

\draw[->] (-.2,0) -> (5.5,0) node[below right]{$\zz$};
\draw[->] (0,-.2) -> (0,4) node[above left]{$\yy$};

\end{tikzpicture}
\end{center}

        \caption{Any point $(\overline{\zz},\overline{\yy})$ can be expressed as a convex combination of the four vertices of a single sector. Each sector is assigned a binary code $\boldsymbol\omega_k$ in reflective Gray code order.}
        \label{fig:Graycodes:EncodingCombibnation}
        \end{subfigure}
    \caption{An example logarithmic embedding. Four segments are encoded with only two binary variables.}
    \label{fig:Graycodes}
    \end{figure}

    \begin{technique}{Logarithmic Embedding Approximation}{LogE}    
    Let $\Omega\subseteq\{0,1\}^\numBreakpoints$ be a set of $\numBreakpoints$-digit binary codes and let each code $\boldsymbol\omega$ correspond to a single, bounded polyhedron $\DomTwo_{\boldsymbol\omega}$. Suppose that, together, these polyhedra form a pairwise interior-disjoint cover of $\DomTwo$. Let $V = \bigcup_{\boldsymbol\omega \in \Omega} \mathrm{ext}(\DomTwo_{\boldsymbol\omega})$ and define the following set: 
        \begin{subequations}\label{tech:LogE:set}
        \vspace{1em}\\$\LogE_V(\psi) \DEF$\vspace{-1em}\\
        \begin{SetArray}
            {~\hfill~}
            {(\yy,\zz,\ff,\deltaVec,\lambdaVec)\in\DomTwo\times\Reals
            \times\Omega
            \times[0,1]^{|V|}}
            \ff &= \textstyle\sum_{\mathbf v \in V}  \lambdaVar_{\mathbf v} \psi(\mathbf v)  \label{tech:LogE:set:Obj}\\
            (\yy,\zz)  &=  \textstyle\sum_{\mathbf v \in V}  \lambdaVar_{\mathbf v} \mathbf v \label{tech:LogE:set:Comb} \\
            1  &=  \textstyle\sum_{\mathbf v \in V} \lambdaVar_{\mathbf v} \label{tech:LogE:set:Conv} \\
            1-\deltaVar_k  &\geq  \textstyle\sum_{\mathbf v \in V\vert^k_0}  \lambdaVar_{\mathbf v}  
                &\forall\ k\in\breakpointsIter \label{tech:LogE:set:SOS1} \\ 
            \deltaVar_k  &\geq  \textstyle\sum_{\mathbf v \in V\vert^k_1} \lambdaVar_{\mathbf v}
                &\forall\ k\in\breakpointsIter \label{tech:LogE:set:SOS2}
        \end{SetArray}
        \end{subequations}
    where
        \begin{equation*}\label{tech:LogE:V0k}
        V\vert^k_p  =  \left\{\mathbf v \in V  \MID \omega_{k} = p \text{ if } \mathbf v \in \mathrm{ext}(\DomTwo_{\boldsymbol\omega})\right\}.
        \end{equation*}
    \end{technique}

Notice that Approximation \ref{tech:LogE} has one spare degree of freedom; that is, there are five variables $(\yy,\zz,\ff,\deltaVec,\lambdaVec)$ but only four equations \eqref{tech:LogE:set:Obj}-\eqref{tech:LogE:set:Conv}. This means that $\lambdaVec$, and therefore $\ff$, may take a range of values for any given $\yy,\ \zz,$ and $\deltaVec$. See Figure \ref{fig:LogEComparison}. Despite this, Approximation \ref{tech:LogE} remains a competitive technique for embedding rational objectives. 
    
Aside from the logarithmic relationship that Approximation \ref{tech:LogE} maintains between the number of binary variables and the number of breakpoints, the changeover to Cartesian coordinates naturally bypasses the rational term which caused problems in the piecewise linear Approximation \ref{tech:PWL} while performing at least as well as the stepwise Approximation \ref{tech:Step}. This is formalized in Proposition \ref{tech:LogE:Proposition}.

    \begin{proposition}\label{tech:LogE:Proposition}
    Given a finite, increasing sequence of breakpoints $\{b_k\}_{k\in\left[\!\left[|\Omega|\right]\!\right]}$ such that $\DomOne\subseteq\left[b_0,b_{|\Omega|}\right)$, if the cover polyhedra $\DomTwo_{\boldsymbol\omega}$ are constructed (as portrayed in Figure \ref{fig:Graycodes}) such that
        \begin{equation}\label{tech:LogE:Condition}
        \DomTwo \subseteq \bigcup_{\boldsymbol\omega\in\Omega}\DomTwo_{\boldsymbol\omega}
        \qquad\textrm{and}\qquad
        \left\{
        \frac{y}{z} \MID (y,z)\in\mathrm{ext}(\DomTwo_{\boldsymbol\omega})
        \right\} =
        \left\lbrace b_{\mathcal{Q}\left(\boldsymbol\omega\right)-1},  b_{\mathcal{Q}\left(\boldsymbol\omega\right)} \right\rbrace
        \end{equation}
     holds for each binary code $\boldsymbol\omega\in\Omega$, then we have
        \begin{equation}\label{tech:LogE:Containment}
        \proj_{\left(r=\frac{\yy}{\zz},\,\ff\right)}\left(\LogE_V(\psi)\right) \SUBEQ
        \step_{\{b_k\}}(\brObj).
        \end{equation}
    \end{proposition}

    \begin{proof}
    Constraints \eqref{tech:LogE:set:Obj}, \eqref{tech:LogE:set:Comb}, and \eqref{tech:LogE:set:Conv} translate the point $(\yy,\zz)$ into $\ff$ by way of a convex combination of the vertices $V$. Rather than using traditional SOS2 constraints, \eqref{tech:LogE:set:SOS1} and \eqref{tech:LogE:set:SOS2} ensure that all of the active vertices belong to a single cover polyhedron $\DomTwo_{\boldsymbol\omega}$. We begin by demonstrating this effect.
    
    Let $\mathbf{v}'$ and $\mathbf{v}^\dagger$ be extreme points in $V$ such that $\overline\lambdaVar_{\mathbf{v}'}$ and $\overline\lambdaVar_{\mathbf{v}^\dagger}$ both take positive value under some solution $\left(\overline{\yy},\overline{\zz},\overline{\ff},\overline{\deltaVec},\overline{\lambdaVec}\right)\in\LogE_V(\psi)$. Suppose, for the sake of contradiction, that there exists an index $\overline{k}\in\breakpointsIter$ for which, without loss of generality,    $\mathbf{v}'\in V\vert^{\overline{k}}_0$ while $\mathbf{v}^\dagger\notin V\vert^{\overline{k}}_0$. Since $\omega_{\overline k}$ is binary, it must be that $V\vert^{\overline{k}}_0\cup V\vert^{\overline{k}}_1 = V$ and therefore $\mathbf{v}^\dagger\in V\vert^{\overline{k}}_1$. Now, consider the $\overline{k}$ instance of constraints \eqref{tech:LogE:set:SOS1} and \eqref{tech:LogE:set:SOS2}:
        \begin{align*}
        1-\overline\deltaVar_{\overline k}  
            &\GEQ  \textstyle\sum_{\mathbf v \in V\vert^{\overline{k}}_0}  \overline\lambdaVar_{\mathbf v} \GEQ \overline\lambdaVar_{\mathbf v'} \ > \ 0 \\
        \overline\deltaVar_{\overline k}  &\GEQ  \textstyle\sum_{\mathbf v \in V\vert^{\overline{k}}_1} \overline\lambdaVar_{\mathbf v} \GEQ \overline\lambdaVar_{\mathbf v^\dagger} \ > \ 0.
        \end{align*}
    However, $\overline\deltaVar_{\overline k} \in \{0,1\}$, so these inequalities are contradictory. This means that, if $\overline\lambdaVar_{\mathbf{v}'}$ and $\overline\lambdaVar_{\mathbf{v}^\dagger}$ are both positive, then $\mathbf{v}'$ and $\mathbf{v}^\dagger$ must both belong to the same set $\big(V\vert^k_0\,\textrm{ or}\ V\vert^k_1\big)$ for each index $k\in\breakpointsIter$ and therefore to the same cover polyhedron.

    We now turn our attention to the result \eqref{tech:LogE:Containment}. Recall that the function $\mathcal{Q}$ converts a binary Gray code into a unique integer; thus $\mathcal{Q}(\boldsymbol\omega)$ is an integer. The condition expressed by \eqref{tech:LogE:Condition} requires that each cover polyhedron $\DomTwo_{\boldsymbol\omega}$ corresponds to a unique segment of the piecewise-linear or stepwise approximations: the specific segment between the breakpoints $b_{\mathcal{Q}(\boldsymbol\omega)-1}$ and $b_{\mathcal{Q}(\boldsymbol\omega)}$. 

    Given a point $\left(\overline{\yy},\overline{\zz},\overline{\ff},\overline{\deltaVec},\overline{\lambdaVec}\right)\in\LogE_V(\psi)$ and condition \eqref{tech:LogE:Condition}, constraints \eqref{tech:LogE:set:Comb}-\eqref{tech:LogE:set:SOS2} maintain
        $$
        \frac{\overline\yy}{\overline\zz}\in\left[ b_{\mathcal{Q}\left(\overline\deltaVec\right)-1},b_{\mathcal{Q}\left(\overline\deltaVec\right)} \right]
        $$
    while constraint \eqref{tech:LogE:set:Obj} tells us that
        $$
        \overline\ff \EQ 
            \textstyle\sum_{\mathbf v \in V}  \overline\lambdaVar_{\mathbf v} \psi(\mathbf v) \LEQ
            \textstyle\sum_{\mathbf v \in V} \overline\lambdaVar_{\mathbf v} \brObj\!\left(b_{\mathcal{Q}\left(\overline{\deltaVec}\right)}\right) \EQ \brObj\!\left(b_{\mathcal{Q}\left(\overline{\deltaVec}\right)}\right)
        $$
    so $\left(\frac{\overline\yy}{\overline\zz},\overline\ff\right)\in\step_{\{b_k\}}(\brObj)$.
    \end{proof}

    \begin{corollary}
        The LogE formulation is an $\varepsilon$-approximation of $\gra(\psi)$ with
        $$
        \varepsilon = \max_{k=1, \dots, |\Omega|} \phi(b_k) - \phi(b_{k-1}).
        $$
        As $|\Omega| = 2^\ell$, is follows that $\varepsilon = O(\tfrac{1}{2^\ell})$.
    \end{corollary}
    \begin{proof}
        [Proof sketch]
        Proposition~\ref{tech:LogE:Proposition} shows that the stepwise formulation is a relaxation of the LogE formulation.  A similar calculation shows that if you take a stepwise formulation that takes the minimum value on each interval, then the LogE is a relaxation of this minimum value stepwise formulation.  Thus, on each segment, the LogE error is at most that of a stepwise formulation. 
    \end{proof}
    
In the above corollary, we give a rough estimate of the error of the LogE approximation, but in fact we expect it to be much stronger. One connection we draw is to the piecewise linear approximation.
  We show that the LogE approximation overestimates the piecewise linear approximation when projected into the ratio space.  This may seem disadvantageous, but recall that the piecewise linear approximation also must handle the nonlinear equation of $r = \tfrac{\yy}{\zz}$, which the LogE does not.

    \begin{figure}[ht]
        \centering

\begin{tikzpicture}

\draw[->] (0.66,0.66) -> (6.3,0.66) node[below right]{$r=\frac{\yy}{\zz}$};
\draw[->] (0.98,0.45) -> (0.98,6) node[above left]{$\ff$};

\begin{axis}[
width=0.8\linewidth,
height=0.55\linewidth,
xticklabel=\empty,
yticklabel=\empty,
tick style={draw=none},
legend entries={$\gra(\brObj)$,$\pwlg(\brObj)$,$\step(\brObj)$,$\min\LogE(\brObj)$,$\max\LogE(\brObj)$},
legend pos=south east,
]
\addplot [smooth, line width=3pt, orange!50, line cap=round]
table {%
0.	0.00583449
0.003	0.00625026
0.006	0.00669205
0.009	0.00716122
0.012	0.00765917
0.015	0.00818735
0.018	0.00874727
0.021	0.00934048
0.024	0.00996859
0.027	0.0106333
0.03	0.0113362
0.033	0.0120792
0.036	0.012864
0.039	0.0136926
0.042	0.0145668
0.045	0.0154886
0.048	0.0164601
0.051	0.0174832
0.054	0.0185602
0.057	0.0196931
0.06	0.0208842
0.063	0.0221358
0.066	0.02345
0.069	0.0248293
0.072	0.026276
0.075	0.0277925
0.078	0.0293813
0.081	0.0310447
0.084	0.0327853
0.087	0.0346056
0.09	0.0365081
0.093	0.0384953
0.096	0.0405698
0.099	0.0427341
0.102	0.0449909
0.105	0.0473426
0.108	0.0497918
0.111	0.052341
0.114	0.0549928
0.117	0.0577496
0.12	0.060614
0.123	0.0635883
0.126	0.0666749
0.129	0.0698763
0.132	0.0731947
0.135	0.0766323
0.138	0.0801913
0.141	0.0838739
0.144	0.087682
0.147	0.0916177
0.15	0.0956828
0.153	0.0998791
0.156	0.104208
0.159	0.108672
0.162	0.113271
0.165	0.118008
0.168	0.122883
0.171	0.127897
0.174	0.133052
0.177	0.138349
0.18	0.143787
0.183	0.149367
0.186	0.155091
0.189	0.160957
0.192	0.166966
0.195	0.173118
0.198	0.179413
0.201	0.185849
0.204	0.192427
0.207	0.199145
0.21	0.206003
0.213	0.212998
0.216	0.220131
0.219	0.227398
0.222	0.234799
0.225	0.242331
0.228	0.249992
0.231	0.257779
0.234	0.265691
0.237	0.273724
0.24	0.281875
0.243	0.290142
0.246	0.29852
0.249	0.307007
0.252	0.315598
0.255	0.32429
0.258	0.333079
0.261	0.341961
0.264	0.350931
0.267	0.359985
0.27	0.369118
0.273	0.378326
0.276	0.387603
0.279	0.396945
0.282	0.406346
0.285	0.415802
0.288	0.425307
0.291	0.434855
0.294	0.444441
0.297	0.45406
0.3	0.463707
0.303	0.473374
0.306	0.483057
0.309	0.49275
0.312	0.502448
0.315	0.512144
0.318	0.521833
0.321	0.531509
0.324	0.541166
0.327	0.550799
0.33	0.560403
0.333	0.56997
0.336	0.579497
0.339	0.588978
0.342	0.598407
0.345	0.607779
0.348	0.617089
0.351	0.626332
0.354	0.635502
0.357	0.644596
0.36	0.653608
0.363	0.662534
0.366	0.67137
0.369	0.68011
0.372	0.688752
0.375	0.697291
0.378	0.705723
0.381	0.714046
0.384	0.722254
0.387	0.730346
0.39	0.738318
0.393	0.746168
0.396	0.753892
0.399	0.761488
0.402	0.768954
0.405	0.776288
0.408	0.783487
0.411	0.790551
0.414	0.797477
0.417	0.804264
0.42	0.810911
0.423	0.817418
0.426	0.823783
0.429	0.830006
0.432	0.836085
0.435	0.842023
0.438	0.847817
0.441	0.853468
0.444	0.858976
0.447	0.864343
0.45	0.869568
0.453	0.874652
0.456	0.879596
0.459	0.884401
0.462	0.889068
0.465	0.893599
0.468	0.897994
0.471	0.902256
0.474	0.906386
0.477	0.910385
0.48	0.914256
0.483	0.918001
0.486	0.921621
0.489	0.925118
0.492	0.928495
0.495	0.931754
0.498	0.934898
0.501	0.937927
0.504	0.940846
0.507	0.943655
0.51	0.946359
0.513	0.948959
0.516	0.951457
0.519	0.953857
0.522	0.95616
0.525	0.95837
0.528	0.960489
0.531	0.962519
0.534	0.964463
0.537	0.966323
0.54	0.968103
0.543	0.969805
0.546	0.97143
0.549	0.972982
};
\addplot [line width=1.5pt, green!50!black, line cap=round]
table {%
0.	0
0.11	0.05148
0.22	0.22985
0.33	0.560403
0.44	0.8516
0.55	0.973483
};
\addplot [line width=1.5pt, blue!50, line cap=round]
table {%
0. 0.05148
0.11 0.05148
};
\addplot [line width=1.5pt, blue!50, line cap=round, forget plot]
table {%
0.11 0.22985
0.22 0.22985
};
\addplot [line width=1.5pt, blue!50, line cap=round, forget plot]
table {%
0.22 0.560403
0.33 0.560403
};
\addplot [line width=1.5pt, blue!50, line cap=round, forget plot]
table {%
0.33 0.8516
0.44 0.8516
};
\addplot [line width=1.5pt, blue!50, line cap=round, forget plot]
table {%
0.44 0.973483
0.55 0.973483
};
\addplot [line width=3pt, red!40, line cap=round]
table {%
0.	0.00583449
0.003	0.0062102
0.006	0.0065859
0.009	0.00696159
0.012	0.00733724
0.015	0.00771285
0.018	0.00808842
0.021	0.00846392
0.024	0.00883935
0.027	0.0092147
0.03	0.00958995
0.033	0.00996511
0.036	0.0103402
0.039	0.0107151
0.042	0.0110899
0.045	0.0114645
0.048	0.011839
0.051	0.0122134
0.054	0.0125875
0.057	0.0129615
0.06	0.0133353
0.063	0.0137089
0.066	0.0140823
0.069	0.0144555
0.072	0.0148284
0.075	0.0152011
0.078	0.0155735
0.081	0.0159457
0.084	0.0192268
0.087	0.0229556
0.09	0.0266827
0.093	0.030408
0.096	0.0341315
0.099	0.037853
0.102	0.0415724
0.105	0.0452896
0.108	0.0490046
0.111	0.051975
0.114	0.0534596
0.117	0.054943
0.12	0.0564255
0.123	0.0579068
0.126	0.059387
0.129	0.0608661
0.132	0.0623439
0.135	0.0638205
0.138	0.0652958
0.141	0.0667698
0.144	0.0682424
0.147	0.0697137
0.15	0.0711835
0.153	0.0726518
0.156	0.0741187
0.159	0.075584
0.162	0.0770477
0.165	0.0785099
0.168	0.0799704
0.171	0.0814292
0.174	0.0828864
0.177	0.0843417
0.18	0.0857954
0.183	0.0872472
0.186	0.0886971
0.189	0.0901452
0.192	0.0952767
0.195	0.109762
0.198	0.124233
0.201	0.138688
0.204	0.153127
0.207	0.167549
0.21	0.181955
0.213	0.196344
0.216	0.210716
0.219	0.22507
0.222	0.231704
0.225	0.234481
0.228	0.237255
0.231	0.240025
0.234	0.242791
0.237	0.245554
0.24	0.248313
0.243	0.251067
0.246	0.253818
0.249	0.256565
0.252	0.259308
0.255	0.262046
0.258	0.264781
0.261	0.267511
0.264	0.270237
0.267	0.272958
0.27	0.275676
0.273	0.278388
0.276	0.281096
0.279	0.2838
0.282	0.286499
0.285	0.289193
0.288	0.291883
0.291	0.294568
0.294	0.297248
0.297	0.299923
0.3	0.302593
0.303	0.321884
0.306	0.348563
0.309	0.375199
0.312	0.401792
0.315	0.42834
0.318	0.454844
0.321	0.481302
0.324	0.507715
0.327	0.534082
0.33	0.560403
0.333	0.562872
0.336	0.565337
0.339	0.567798
0.342	0.570254
0.345	0.572706
0.348	0.575153
0.351	0.577595
0.354	0.580033
0.357	0.582465
0.36	0.584893
0.363	0.587317
0.366	0.589735
0.369	0.592149
0.372	0.594557
0.375	0.596961
0.378	0.59936
0.381	0.601753
0.384	0.604142
0.387	0.606525
0.39	0.608904
0.393	0.611277
0.396	0.613645
0.399	0.616008
0.402	0.618365
0.405	0.620718
0.408	0.623065
0.411	0.62728
0.414	0.650698
0.417	0.674068
0.42	0.697389
0.423	0.720661
0.426	0.743885
0.429	0.767058
0.432	0.790182
0.435	0.813256
0.438	0.836279
0.441	0.851947
0.444	0.852987
0.447	0.854024
0.45	0.855059
0.453	0.856092
0.456	0.857123
0.459	0.858151
0.462	0.859177
0.465	0.8602
0.468	0.861221
0.471	0.862239
0.474	0.863255
0.477	0.864269
0.48	0.86528
0.483	0.866289
0.486	0.867295
0.489	0.868299
0.492	0.869301
0.495	0.8703
0.498	0.871296
0.501	0.87229
0.504	0.873282
0.507	0.87427
0.51	0.875257
0.513	0.876241
0.516	0.877222
0.519	0.878201
0.522	0.88334
0.525	0.893097
0.528	0.902831
0.531	0.912541
0.534	0.922227
0.537	0.93189
0.54	0.941528
0.543	0.951143
0.546	0.960733
0.549	0.9703
};
\addplot [line width=3pt, red!70, line cap=round]
table {%
0.	0.00583449
0.003	0.00959163
0.006	0.0133487
0.009	0.0171055
0.012	0.020862
0.015	0.0246182
0.018	0.0283738
0.021	0.0321288
0.024	0.0358831
0.027	0.0396366
0.03	0.0415254
0.033	0.0418996
0.036	0.0422739
0.039	0.0426481
0.042	0.0430223
0.045	0.0433965
0.048	0.0437706
0.051	0.0441447
0.054	0.0445188
0.057	0.0448928
0.06	0.0452668
0.063	0.0456406
0.066	0.0460144
0.069	0.0463881
0.072	0.0467617
0.075	0.0471351
0.078	0.0475084
0.081	0.0478816
0.084	0.0482547
0.087	0.0486275
0.09	0.0490003
0.093	0.0493728
0.096	0.0497451
0.099	0.0501173
0.102	0.0504892
0.105	0.0508609
0.108	0.0512324
0.111	0.0564306
0.114	0.0712757
0.117	0.0861106
0.12	0.100935
0.123	0.115748
0.126	0.13055
0.129	0.145341
0.132	0.160119
0.135	0.174885
0.138	0.189638
0.141	0.191566
0.144	0.193036
0.147	0.194504
0.15	0.195971
0.153	0.197438
0.156	0.198903
0.159	0.200367
0.162	0.20183
0.165	0.203292
0.168	0.204753
0.171	0.206213
0.174	0.207671
0.177	0.209128
0.18	0.210584
0.183	0.212038
0.186	0.213491
0.189	0.214943
0.192	0.216393
0.195	0.217842
0.198	0.219289
0.201	0.220734
0.204	0.222178
0.207	0.22362
0.21	0.225061
0.213	0.2265
0.216	0.227937
0.219	0.229372
0.222	0.248386
0.225	0.276159
0.228	0.303896
0.231	0.331597
0.234	0.35926
0.237	0.386886
0.24	0.414473
0.243	0.442021
0.246	0.469529
0.249	0.48785
0.252	0.490586
0.255	0.493319
0.258	0.496049
0.261	0.498776
0.264	0.501498
0.267	0.504218
0.27	0.506933
0.273	0.509645
0.276	0.512354
0.279	0.515058
0.282	0.517759
0.285	0.520455
0.288	0.523148
0.291	0.525837
0.294	0.528522
0.297	0.531202
0.3	0.533879
0.303	0.536551
0.306	0.539219
0.309	0.541882
0.312	0.544541
0.315	0.547196
0.318	0.549847
0.321	0.552493
0.324	0.555134
0.327	0.557771
0.33	0.560403
0.333	0.585099
0.336	0.609751
0.339	0.634357
0.342	0.658919
0.345	0.683434
0.348	0.707904
0.351	0.732326
0.354	0.756702
0.357	0.781031
0.36	0.788651
0.363	0.791069
0.366	0.793484
0.369	0.795894
0.372	0.7983
0.375	0.800702
0.378	0.803099
0.381	0.805491
0.384	0.80788
0.387	0.810264
0.39	0.812643
0.393	0.815018
0.396	0.817388
0.399	0.819753
0.402	0.822114
0.405	0.82447
0.408	0.826821
0.411	0.829168
0.414	0.83151
0.417	0.833847
0.42	0.836179
0.423	0.838506
0.426	0.840828
0.429	0.843146
0.432	0.845458
0.435	0.847766
0.438	0.850068
0.441	0.855071
0.444	0.865469
0.447	0.875844
0.45	0.886195
0.453	0.896524
0.456	0.906828
0.459	0.917109
0.462	0.927366
0.465	0.937599
0.468	0.946512
0.471	0.947528
0.474	0.948543
0.477	0.949555
0.48	0.950565
0.483	0.951573
0.486	0.952579
0.489	0.953582
0.492	0.954584
0.495	0.955582
0.498	0.956579
0.501	0.957573
0.504	0.958565
0.507	0.959555
0.51	0.960543
0.513	0.961528
0.516	0.96251
0.519	0.963491
0.522	0.964469
0.525	0.965445
0.528	0.966418
0.531	0.967389
0.534	0.968358
0.537	0.969324
0.54	0.970288
0.543	0.971249
0.546	0.972208
0.549	0.973165
};

\end{axis}

\end{tikzpicture}

        \caption{Computed in Mathematica,  this plot shows the result of minimizing or maximizing $\ff$ over $\LogE(\psi)$ as projected onto $(r=\frac{\yy}{\zz},\ff)$. For a particular $(\yy,\zz)$, approximation \ref{tech:LogE} may give any value of $f$ between the red and pink lines over the ratio $\frac{\yy}{\zz}$. Importantly, this example has $\check{s} = 100$, $\hat{s}=1000$, and $\{b_k\}$ evenly distributed between $0$ and $0.55$. Each of the points evaluated has an identical distance from the origin:
        $\yy^2+\zz^2 = 300$.}
        \label{fig:LogEComparison}
    \end{figure}

 \begin{theorem}\label{tech:LogE:Conjecture}
 Given a finite, increasing sequence of breakpoints $\{b_k\}$, if the covering polyhedra are constructed as in \eqref{eq:covering} and such that  \eqref{tech:LogE:Condition} holds, then
        \begin{equation}\label{tech:LogE:Conjecture:EQ}
        \pwlg_{\{b_k\}}(\brObj) \SUBEQ
        \proj_{\left(r=\frac{\yy}{\zz},\,\ff\right)}\left(\LogE_V(\psi)\right).
        \end{equation}
    \end{theorem}
    Indeed, this containment is evident in Figure~\ref{fig:LogEComparison}. 
     The proof requires some lengthy calculations about the behavior of the LogE formulation in this context. 
 We prove this result in Appendix~\ref{appendix:proof-LogE}.  Note that this result might not hold true if we trim the covering polyhedra by intersecting with further domain restrictions.

    \begin{remark}
We also see from Figure~\ref{fig:LogEComparison} that the LogE can be an upper bound (and hence a relaxation of the graph) for $\psi$.  This depends on the derivative and second derivative of the function $\phi$ and the domain in which it is being studied.  Particularly, for our choice of $\phi$, if the inputs were very large ratios, then this bound may not be valid.  However, in our scenarios, we can actually guarantee that the LogE formulation provides a valid upper bound.
\end{remark}

    \begin{modelCustom}{1}{c}[LogE Approximation of BR-Redistricting]\label{labeling:LogE}
    We apply Approximation \ref{tech:LogE} to each district $j$ in Model \ref{model:BR-Redistricting} independently:
        \begin{equation*}
        \max_{(\x,\y,\z)\,\in\,\XYZFeasPoly_\textrm{B}}
        \left\{
        \sum_{j=1}^{\numDistricts}f_j
        \MID
        (\yy_j,\zz_j,\ff_j,\deltaVec^j,\lambdaVec^j)\in \LogE_{V_j}(\psi)
        \quad
        \forall\ j\in\{1,\dots,\numDistricts\}
        \right\}.
    \end{equation*}
    \end{modelCustom}

Given a set of binary codes $\Omega=\{0,1\}^\numBreakpoints$, we use scaled arcs to generate our cover polyhedra in a computationally efficient manner. It is possible to identify a finite, increasing sequence of breakpoints $\{b_k\}_{k\in[\![2^\numBreakpoints]\!]}$ along with inner and outer radii $\check{s}$ and $\hat{s}$ such that 
    $$
    \DomTwo \SUBEQ \left\{(z,y)\in\Reals^2_+ \MID \check{s} < z^2+y^2 < \hat{s}  \hquad\textrm{and}\hquad  b_0 \leq \frac{z}{y} \leq b_{2^\numBreakpoints}\right\}.
    $$
While it is possible to implement this technique with an under-full set of binary codes $(\Omega\subset\{0,1\}^\numBreakpoints)$, adding more codes (and more cover polyhedra) up to $\Omega=\{0,1\}^\numBreakpoints$ improves the approximation without requiring additional binary variables. Thus, we use a full set of binary codes for a given $\numBreakpoints$.

\subsection{Ben-Tal \& Nemirovski Stepwise Embedding}
Algorithm \ref{alg:BNFull} details a process for approximating the value of an increasing, single-ratio function (like our BR-Redistricting objective) at a given point in the domain. This scheme is heavily inspired by the Ben-Tal \& Nemirovski LP approximation of the second-order cone~\cite{Ben-Tal-Nemirovski}, an impressive result that enables the approximation of the second-order cone with exponentially many linear inequalities (upon projection to the space of the original variables). Further improvements to the approximation, using rational data of small sizes, can be found in~\cite{burak2021}. Our approach is related to that presented in~\cite{Dong-Luo-2018}, wherein binary variables are incorporated to gain additional insights. In particular, we use $\deltabm\in\{0,1\}^\numBreakpoints$ to keep track of the mirroring steps, which allows us to estimate the position of the initial point.

Define the rotation function $\mathrm{Rot}(\theta):\Reals\rightarrow[-1,1]^{2\times2}$:
    \begin{equation}\label{def:RotMat}
    \mathrm{Rot}(\theta)  \DEF  
    \begin{bmatrix}
    \cos(\theta) & \sin(\theta)\\
    -\sin(\theta) & \cos(\theta)
    \end{bmatrix}.
    \end{equation}

    \begin{algorithm}[ht!]
    \caption{Stepwise approximation of rational function value via rotation and mirroring}
    \label{alg:BNFull}
    \begin{algorithmic}[1]
    \State \textbf{Input:} an increasing and bounded function $\brObj: \Reals\rightarrow\Reals$, a point  $(\overline\zz,\overline\yy)\in\Reals^2_+$, an angle $\theta_0\in[0,\frac{\pi}{2})$ such that $\frac{\overline\yy}{\overline\zz}\leq\tan(\theta_0)$, and a number of iterations $\numBreakpoints\in\Z_+$
    \State\textbf{Output:} an approximation $f$ of $\brObj\big(\frac{\yy}{\zz}\big)$
    \item[]
    \State define $\xiVar = \overline\zz$ and $\etaVar = \overline\yy$
    \For{$k\in\breakpointsIter$}
        \State define $\theta_k = \frac{\theta_0}{2^k}$
        \State\hspace{-0.7em}\textbf{Rotation Step:} 
        \State 
            $
            \begin{bmatrix} \xiVar \\ \tilde \etaVar \end{bmatrix}  \gets  \mathrm{Rot}(\theta_k) \begin{bmatrix} \xiVar \\ \etaVar \end{bmatrix}
            $
            
        \State\hspace{-0.7em}\textbf{Mirroring Step:}
        \If{$\tilde\etaVar < 0$}
            \State $\etaVar \gets -\tilde\etaVar$ and define $\delta_k = 1$
        \Else
            \State $\etaVar \gets \tilde\etaVar$ and define $\delta_k = 0$
        \EndIf
    \EndFor
    \State set $q = \sum_{k=1}^\numBreakpoints2^{\numBreakpoints-k}\left(\sum_{t=1}^k\delta_t\right)\MOD2$
    \item[]
    \State\textbf{Return:} $f = \brObj\!\left(\tan\!\left(\frac{\theta_0}{2}-q\,\theta_\numBreakpoints\right)\!\right)$
    \end{algorithmic}
    \end{algorithm}

\begin{proposition}\label{prop:BN-graycode}
After Algorithm \ref{alg:BNFull} terminates, we have that $\tfrac{\etaVar}{\xiVar} \in \left[0, \tan(\theta_\numBreakpoints)\right]$. Furthermore, $q$ is the unique integer such that $\frac{\overline\yy}{\overline\zz} \in \left(\tan\!\left(\tfrac{\theta_0}{2} - (q+1)\,\theta_\numBreakpoints\right), \tan\!\left(\tfrac{\theta_0}{2} - q\,\theta_\numBreakpoints\right)\right]$.
\end{proposition}
\noindent A detailed proof of Proposition \ref{prop:BN-graycode} is given in Appendix \ref{appendix:gray_code}. See Figure \ref{fig:BNFull:Explain} for a visual explanation of this process.

    \begin{figure}[!hb]
    \centering
    
\begin{tikzpicture}
\usetikzlibrary{bending}
\usetikzlibrary{shapes.geometric}

\newcommand{\pgfgetlastpolar}[2]{
    \pgfgetlastxy{\tempx}{\tempy}
    \pgfmathsetmacro{#1}{atan2(\tempy,\tempx)}
    \pgfmathsetmacro{#2}{veclen(\tempx,\tempy)}
}

\colorlet{color1}{yellow!20!white}
\colorlet{color2}{red!20!white}
\colorlet{color3}{blue!20!white}

\newcommand*{\thetaOne}{22.5}
\newcommand*{\thetaTwo}{11.25}
\newcommand*{\thetaThree}{5.625}
\newcommand*{\thetaFour}{2.8125}

\path[fill=color2] (0:1) -- (45:1.5) -- (45:4.6) -- ($(45:4.6)+(3.4,0)$) -- (0:7.8) -- cycle; \draw (0:1) -- (45:1.5); \draw ($(45:4.6)+(3.4,0)$) -- (0:7.8);
\draw[thick,->] (0,0) -> (45:5.1) node[anchor=west, rotate=45]{$\frac{\yy}{\zz} = 1$};

\draw[->, dashed] (0,0) -> (\thetaOne:7.8) node[anchor=west, rotate=\thetaOne]{$\frac{\yy}{\zz} = \tan\!\left(\frac{\pi}{8}\right)$};
\draw[->, dashed] (0,0) -> (\thetaTwo:7.9) node[anchor=west, rotate=\thetaTwo]{$\frac{\yy}{\zz} = \tan\!\left(\frac{\pi}{16}\right)$};
\draw[->, dashed] (\thetaThree:6.6) -> (\thetaThree:8) node[anchor=west, rotate=\thetaThree] {$\frac{\yy}{\zz} = \tan\!\left(\frac{\pi}{32}\right)$}; \draw[dashed] (0,0) -- (\thetaThree:5.4);

\draw[dotted] (\thetaTwo+\thetaThree:1.5) -- (\thetaTwo+\thetaThree:7.15);
\draw[dotted] (\thetaOne+\thetaThree:1.55) -- (\thetaOne+\thetaThree:6.5);
\draw[dotted] (\thetaOne+\thetaTwo:1.6) -- (\thetaOne+\thetaTwo:5.5);
\draw[dotted] (\thetaOne+\thetaTwo+\thetaThree:1.65) -- (\thetaOne+\thetaTwo+\thetaThree:4.8);

\newcommand*{\etaZero}{26.5}
\node[circle, fill=black, minimum size=5pt] (a) at (\etaZero:6) {};

\node[circle, fill=black, minimum size=5pt] (b) at (\etaZero-\thetaOne:6) {};
\pgfgetlastpolar{\etaOne}{\myRho}
\draw[arrows = {->[bend]}] (a) to[out=-25,in=55] (b);

\node[circle, draw, minimum size=5pt] (c) at (\etaOne-\thetaTwo:6) {};
\pgfgetlastpolar{\etaTwo}{\myRho}
\pgfgetlastxy{\xTwo}{\yTwo}
\draw[arrows = {->[bend]}] (b) to[out=-45,in=45] (c);

\node[circle, fill=black, minimum size=5pt] (d) at  (\xTwo,-\yTwo) {};
\pgfgetlastpolar{\etaThree}{\myRho}
\pgfgetlastxy{\xThree}{\yThree}
\draw[arrows = {|->[bend]}] (c) to[out=150,in=-150] (d);

\node[star, fill, minimum size=6pt] (e) at (\etaThree-\thetaThree:6) {};
\draw[arrows = {->[bend]}] (d) to[out=-130,in=140] (e);

\node[label={[rotate=\etaZero-2] west:$(\overline{\zz},\overline{\yy})$}] at ($(a)+(0.05,-0.15)$) {};
\node[label={[rotate=\etaZero-\thetaOne+1] east:$1$}] at (b) {};
\node[label={[rotate=\etaOne-\thetaTwo] east:$\tilde2$}] at ($(c)+(0,-0.05)$) {};
\node[label={[rotate=-\etaOne+\thetaTwo] west:$2$}] at  ($(d)+(-0.03,0.15)$) {};
\node[label={[rotate=-\etaTwo-\thetaThree] south west:$3$}] at  ($(e)+(0.05,0.1)$) {};

\draw (-.2,0) -> (5.6,0);\draw[->] (6,0) -> (9.8,0) node[below right]{$\zz$};
\draw[->] (0,-1) -> (0,4.2) node[above left]{$\yy$};

\end{tikzpicture}
    
    \caption{A demonstration of Algorithm \ref{alg:BNFull}. The initial point $(\overline\zz,\overline\yy)$ goes through $3$ rotations with a mirroring at $k=2$ because $\tilde\etaVar_2 < 0$. This gives $\deltabm = (0,1,0)$ and $q = 3$. Therefore, by Proposition \ref{prop:BN-graycode}, we know that $\frac{\overline\yy}{\overline\zz}$ lies within $\left(\tan\!\left(\tfrac{\pi}{8}\right), \tan\!\left(\tfrac{\pi}{8}+\frac{\pi}{32}\right)\right]$.}
    \label{fig:BNFull:Explain}
    \end{figure}

If $f_q = \brObj\!\left(\tan\!\left(\frac{\pi}{4}-q\,\theta_\numBreakpoints\right)\right)$ is precomputed for $q\in[\![2^\numBreakpoints]\!]$, then Algorithm \ref{alg:BNFull} can be performed within a Linear Program to bound Model \ref{model:BR-Redistricting}. This approach becomes particularly elegant upon noticing that line thirteen of Algorithm \ref{alg:BNFull} is identical to \eqref{Gray2Int}; $q$ is being set to the integer representation of the Gray code $\boldsymbol\delta$.

\begin{technique}{Ben-Tal \& Nemirovski Stepwise Relaxation}{BNStep}
Identify an initial angle $\theta_0\in[0,\frac{\pi}{2}]$ such that $\tan(\theta_0) \geq \sup(\DomOne)$, let $\theta_k = \frac{\theta_0}{2^k}$ for each rotation step $k\in[\![\numBreakpoints]\!]$, and define the following set to perform the rotation and mirroring steps:
    \begin{subequations}\label{tech:BNRot:set}
    \vspace{1em}\\$\BNRot_{\{\theta_k\}}(\psi) \DEF$ \vspace{-1em}\\
    \begin{SetArray}
    {~\hfill~}
    {(\yy,\zz,\deltaVec)\in\Reals^2_+\times\Reals\times\{0,1\}^\numBreakpoints}
    \begin{bmatrix} \xiVar_0 \\ \etaVar_0 \end{bmatrix} &= \begin{bmatrix} \zz \\ \yy \end{bmatrix} \\
    \begin{bmatrix} \xiVar_k \\ \tilde{\etaVar}_k \end{bmatrix} 
        &= \mathrm{Rot}(\theta_k) \begin{bmatrix} \xiVar_{k-1} \\ \etaVar_{k-1} \end{bmatrix}
        &\forall\ k \in \breakpointsIter \label{BNRot:Rotate}\\
    \tilde{\etaVar}_k &\leq \etaVar_k \leq \tilde{\etaVar}_k + \bigM_k \deltaVar_k
        &\forall\ k \in \breakpointsIter \label{BNRot:ABS1} \\
    - \tilde{\etaVar}_k &\leq \etaVar_k \leq -\tilde{\etaVar}_k + \bigM_k(1 - \deltaVar_k)
        &\forall\ k \in \breakpointsIter \label{BNRot:ABS2}
    \end{SetArray}
    \end{subequations}

    Define another set to return the appropriate function value.
    \vspace{1em}\\$\BNStep_{\{\theta_k\}}(\psi) \DEF$ \vspace{-1em}\\
    \begin{SetArray}
    {~\hfill~}
    {(\yy,\zz,\deltaVec,\ff)\in\BNRot_{\{\theta_k\}}(\psi)\times\Reals}
    \ff &\leq \overline{f}_{\boldsymbol\omega} + \bigN_{\boldsymbol\omega} \left\Vert \deltaVec - \boldsymbol\omega \right\Vert_1 
        &\forall\ \boldsymbol\omega \in \{0,1\}^\numBreakpoints \label{BNStep:Obj}
    \end{SetArray}
where
    $
    \overline{f}_{\boldsymbol\omega} 
        \EQ \brObj\left(\tan\left(
                \theta_0-\mathscr{Q}(\boldsymbol\omega)\theta_\numBreakpoints
            \right)\right)
    $
for each binary code $\boldsymbol\omega\in\{0,1\}^\numBreakpoints$.
\end{technique}

\begin{proposition}
If $\vecM$ and $\vecN$ are sufficiently large, we have
    \begin{equation}\label{tech:BNStep_Projection}
    \proj_{\left(r=\frac{\yy}{\zz},\,\ff\right)}\left(\BNStep_{\{\theta_k\}}(\psi)\right)  \EQ  \step_{\{b_k\}}(\brObj)
    \end{equation}
where $b_{\mathscr{Q}(\boldsymbol\omega)} = \tan\left(\theta_0-\mathscr{Q}(\boldsymbol\omega)\theta_\numBreakpoints\right)$ for each $\boldsymbol\omega\in\{0,1\}^\numBreakpoints$.
\end{proposition}
\begin{proof}
$\BNStep_{\{\theta_k\}}(\psi)$ is a closed form of Algorithm \ref{alg:BNFull}. Constraint \eqref{BNRot:Rotate} encodes the rotation step (line 7) while constraints \eqref{BNRot:ABS1} and \eqref{BNRot:ABS2} encode the Mirroring Step (lines 9-12). Then, \eqref{BNStep:Obj} handles the return value: the functional upper bound of whichever sector $(\yy,\zz)$ belongs to.

Notice that $\brObj\left(\sup(\DomOne)\right)-f_{j\boldsymbol\omega}$ is a sufficiently large value for $\bigN_{\boldsymbol\omega}$ since $\left\Vert\boldsymbol\omega-\deltaVec\right\Vert_1$ equals zero if $\boldsymbol\omega = \deltaVec$ but takes a value of at least one otherwise. Similarly, if $\hat z$ and $\hat y$ are represent elementwise upper bounds on $\DomTwo$, then $\left(\hat y^2+\hat z^2\right)^\frac{1}{2}$ is a sufficiently large value for each $\bigM_k$. 
\end{proof}

The key difference between Approximation \ref{tech:BNStep} and the Ben-Tal \& Nemirovski LP approximation of SOCP~\cite{Ben-Tal-Nemirovski} appears in constraints \eqref{BNRot:ABS1} and \eqref{BNRot:ABS2}; rather than $\etaVar_k \geq \vert\tilde\etaVar_k\vert$ which is used in~\cite{Ben-Tal-Nemirovski}, our goals require the binary variables $\deltaVec$ involved in linearizing $\etaVar_k = \vert\tilde\etaVar_k\vert$. 

\begin{remark}
    It is possible to implement Approximation \ref{tech:BNStep} with an arbitrary sequence of rotation angles $\{\theta_k\}$ so long as:
        $\tan(\theta_0) \geq \sup(\DomOne)$, $\theta_k$ decreases in $k$ such that $\theta_\numBreakpoints > 0$, and $\theta_0 > \sum_{k\in[\![\numBreakpoints]\!]}\theta_k$. We chose to use $\theta_k = \theta_0 / 2^k$ for its elegant integration with reflective Gray code and for compatibility with the upcoming Approximation~\ref{tech:BNFull}.
\end{remark}

    \begin{modelCustom}{1}{d}[BN-Stepwise Relaxation of BR-Redistricting]\label{model:BNStep}
    We apply Approximation \ref{tech:BNStep} to each district $j$ in Model \ref{model:BR-Redistricting} independently:
    \begin{equation*}
    \max_{(\x,\y,\z)\,\in\,\XYZFeasPoly_\textrm{B}}
    \left\{
    \sum_{j=1}^{\numDistricts}f_j
    \MID
    (\yy_j,\zz_j,\ff_j,\deltaVec^j)\in \BNStep(\psi)
    \quad
    \forall\ j\in\{1,\dots,\numDistricts\}
    \right\}
    \end{equation*}
    \end{modelCustom}

    \begin{corollary}
    Model~\ref{model:BNStep} is an upper bound on Model~\ref{model:BR-Redistricting} with a maximum error of: 
        $$
        \sum_{j=1}^{\numDistricts}\left(\max_{k\,\in\,\breakpointsIter} \big\{\brObj(b_{j,k}) - \brObj(b_{j,k-1})\big\}\right).
        $$
    \end{corollary}
    \begin{proof}
    Since $\brObj$ is assumed to be increasing over $\DomOne$, this follows from \eqref{tech:BNStep_Projection} by Proposition~\ref{prop:GraphRelaxation}, Proposition~\ref{prop:OverestimateError}, and Proposition~\ref{lem:StepOverestimate}.
    \end{proof}

\subsection{Ben-Tal \& Nemirovski Full Embedding}

Perhaps surprisingly, this Approximation~\ref{tech:BNStep} can be strengthened to match Approximation \ref{tech:LogE} (LogE) with the addition of only 4 continuous variables.

    \vspace{1em}

\begin{technique}{Ben-Tal \& Nemirovski Full Approximation}{BNFull}
Identify an initial angle $\theta_0\in[0,\frac{\pi}{2}]$ such that $\tan(\theta_0) \geq \sup(\DomOne)$, let $\theta_k = \frac{\theta_0}{2^k}$ for each rotation step $k\in[\![\numBreakpoints]\!]$, and recall the set $\BNRot_{\{\theta_k\}}(\psi)$ from Approximation \ref{tech:BNStep}. Additionally, define a set of vertices 
    $$
    U = \left\{
    \begin{bmatrix}\check{s}\\0\end{bmatrix},\ 
    \begin{bmatrix}\hat{s}\\0\end{bmatrix},\ 
    \begin{bmatrix}\check{s}\cos(\theta_\numBreakpoints)\\\check{s}\sin(\theta_\numBreakpoints)\end{bmatrix},\ 
    \begin{bmatrix}\hat{s}\cos(\theta_\numBreakpoints)\\\hat{s}\sin(\theta_\numBreakpoints)\end{bmatrix}
    \right\}
    $$
where $\check{s}$ and $\hat{s}$ are respectively a lower and upper bound on $\sqrt{y^2+z^2}$ over $(\yy,\zz)\in\DomTwo$.
Define another set to return the appropriate function value.
    \begin{subequations}
    \vspace{1em}\\$\BNFull_{\mathbf{s}}(\psi) \DEF$ \vspace{-1em}\\
    \begin{SetArray}
    {~\hfill~}
    {(\yy,\zz,\deltaVec,\ff,\lambdaVec)\in\BNRot_{\{\theta_k\}}(\psi)\times\Reals\times[0,1]^4}
    \begin{bmatrix}\xiVar_\numBreakpoints\\ \etaVar_\numBreakpoints\end{bmatrix}  &=  \sum_{\mathbf{u}\in U} \lambdaVar_\mathbf{u} \mathbf{u} \\
        1  &=  \sum_{\mathbf{u}\in U}\lambda_\mathbf{u} \\
    \ff   &\leq  \sum_{\mathbf{u}\in U} \lambdaVar_\mathbf{u} \overline{f}_{\boldsymbol\omega\mathbf{u}} + \bigN_{\boldsymbol\omega}\left\Vert\boldsymbol\omega-\deltaVec\right\Vert_1 
            \\&~\hspace{7em}\forall\ \boldsymbol\omega \in \{0,1\}^\numBreakpoints \label{BNFull:Obj}\notag
    \end{SetArray}
    \end{subequations}
where $\overline{f}_{\boldsymbol\omega\mathbf{u}}$ is the objective function value of vertex $\mathbf{u}$ rotated to the sector which corresponds to $\boldsymbol\omega$:
    $$
    \overline{f}_{\boldsymbol\omega\mathbf{u}}  \EQ  \brObj\big(\mathrm{Rot}(\mu_{\boldsymbol\omega})\mathbf{u}\big) \qquad\text{where}\qquad 
    \mu_{\boldsymbol\omega}  \EQ  \theta_0 - \big[\mathscr{Q}(\boldsymbol\omega) - 1\big]\theta_\numBreakpoints.
    $$
for each binary code $\boldsymbol\omega\in\{0,1\}^\numBreakpoints$.
\end{technique}

\begin{proposition}
If the cover polyhedra $\mathscr{B}_{\boldsymbol\omega}$ are constructed such that 
    \begin{equation}\label{prop:BNFull:Condition}
    \mathrm{ext}\left(\mathscr{B}_{\boldsymbol\omega}\right) = \left\{
            \mathrm{Rot}(\theta_{\mathscr{Q}(\boldsymbol\omega)})\mathbf{u} \MID \mathbf{u}\in U
        \right\},
    \end{equation}
we have
    \begin{equation}\label{prop:BNFull:Result}
    \proj_{\left(r=\frac{\yy}{\zz},\,\ff\right)}\left(\LogE_{V}(\psi)\right) \EQ \proj_{\left(r=\frac{\yy}{\zz},\,\ff\right)}\left(\BNFull_{\mathbf{s}}(\psi)\right).
    \end{equation}
\end{proposition}
\begin{proof}
If the vertices $V$ are constructed to match those implied by $\BNFull_\mathbf{s}(\psi)$ as required by \eqref{prop:BNFull:Condition}, then both $\LogE_V(\psi)$ and $\BNFull_\mathbf{s}(\psi)$ use the same convex combination of vertices to translate $(\yy,\zz)$ into $\ff$.
\end{proof}

\begin{modelCustom}{1}{e}[BN-Full Approximation of BR-Redistricting]\label{labeling:BNFull}
We apply Approximation \ref{tech:BNFull} to each district $j$ in Model \ref{model:BR-Redistricting} independently:
\begin{equation*}
\max_{(\x,\y,\z)\,\in\,\XYZFeasPoly_\textrm{B}}
\left\{
\sum_{j=1}^{\numDistricts}f_j
\MID
(\yy_j,\zz_j,\ff_j,\deltaVec^j,\lambdaVec^j)\in \BNFull_{\mathbf{s}^j}(\psi)
\quad
\forall\ j\in\{1,\dots,\numDistricts\}
\right\}
\end{equation*}
\end{modelCustom}

\section{Computational Improvements for Redistricting Problems}
\label{sec:extra-bounds}
\subsection{Bounding Ratios, VAP, and BVAP for Approximation \ref{tech:Step}}\label{sub:BoundingRatios}
By tightening the domain over which we make the Technique~\ref{tech:Step} approximations, we can tighten the approximation itself. Let $\DomOne_{j}$ represent the domain of the ratio $\tfrac{y_j}{z_j}$ that is feasible under Model \ref{model:BR-Redistricting}
    $$
     \DomOne_{j}  
\EQ \left\lbrace \frac{\yy_{j}}{\zz_{j}}\ :\ (\x,\y,\z) \in \XYZFeasPoly_\textrm{B}\right\rbrace.
    $$
Notice that, due to symmetry in $\XContPoly$, $\DomOne_j$ is invariant under the choice of $j$. Breaking this symmetry by ordering the districts according to their \BVAP\ will result in tightened domains and breakpoints specific to each district. Define a modified domain $\DomOne^\circ_j$ specific to each district $j$ wherein the $\yy$ variables are ordered
    $$
    \DomOne^\circ_{j}  
    \EQ \left\lbrace \frac{\yy_{j}}{\zz_{j}}\ :\ (\x,\y,\z) \in \XYZFeasPoly_\textrm{B},  \quad 
        \yy_1 \leq \yy_2 \leq \ldots \leq \yy_{\numDistricts}\right\rbrace.
    $$

Naturally, each $\DomOne^\circ_j$ will be a proper subset of $\DomOne_j$ so that breakpoints $\mathbf{b}_j$ can be generated specific to $\DomOne^\circ_j$ for each district $j$. In particular, we can find $b_{j0}$ and $ b_{j\numBreakpoints}$ such that $b_{j0} \leq \tfrac{\yy_j}{\zz_j} \leq b_{j\numBreakpoints}$ for any element of $\DomOne^\circ_{j}$.

Since $\brObj$ is increasing, it is also possible to break symmetry and achieve a similar effect by ordering on the $\zz$ variables, the ratios $\yy_j/\zz_j$, or the function output variables $\ff_j$. However these orderings are inconsistent so only one order may be applied in symmetry-breaking steps. Through some experiments, we found ordering the $\yy_j$ variables to be most effective.  Perhaps the ratio is more complicated without being better correlated with the objective than the $\y$ variables.
    \begin{figure}[H]
        \centering
            \centering
            \includegraphics[scale=0.45]{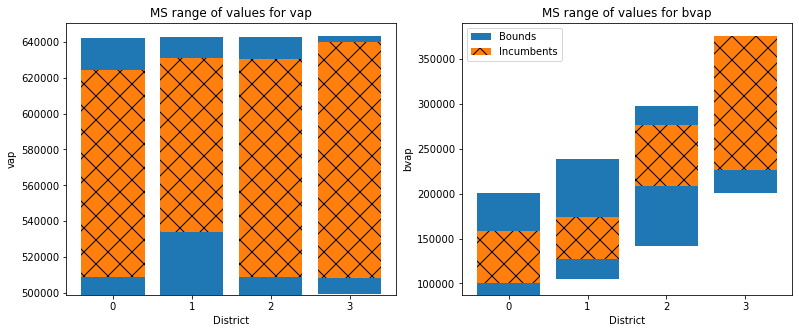}
            \label{fig:MS_bounds}
            \caption{Bounds created for Mississippi with 10 minutes of computing time for each max and min computation. Although the computation does not reach optimality, it provides significantly better bounds than simply using the upper bound for each \( \varColor{ \mathbf{y}} \) and \( \varColor{ \mathbf{z}} \) variable. Furthermore, these bounds can be computed in parallel. See Appendix~\ref{sec:variable_ranges} for additional information on other states.}
    \end{figure}
Good bounds on VAP and BVAP are doubly important for the performance of the LogE approximation as we can reduce the length of the rays and arcs needed to generate extreme points.  We do this by just computing the intersection with the domain created by the bounds; see Figure \ref{fig:BVAPCompare}.

\begin{figure}[H]
    \centering
    \begin{minipage}{.5\textwidth}
        \centering
        \includegraphics[trim= 1cm 0.1cm 3cm 0.7cm, clip, scale=0.35]{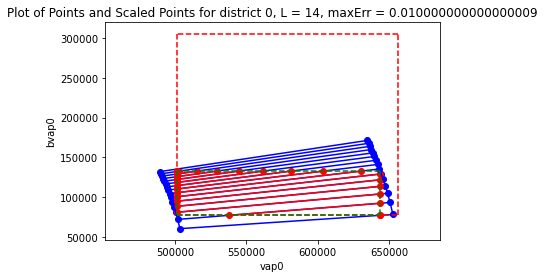}
        \includegraphics[trim= 1cm 0.1cm 3cm 0.7cm, clip, scale=0.35]{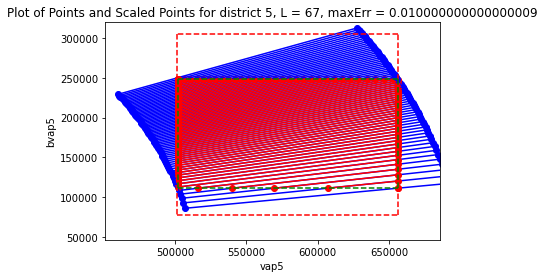}
    \end{minipage}%
    \begin{minipage}{.5\textwidth}
    \captionsetup{width=\textwidth} 
        \caption{Comparison of the BVAP (Black Voting Age Population) and VAP (Voting Age Population) ranges for the 1st and 6th districts of South Carolina when districting at the county level is shown. The dashed red bounding boxes represent general bounds for each district. The blue rays depict breakpoint rays while the red lines are their intersections with the tightest bounds known for each district. Clearly, covering the dashed, general bounds to this density would require significantly more breakpoints.}
        \label{fig:BVAPCompare}
    \end{minipage}
\end{figure}

\subsection{Gradient Cuts}\label{sec:gradient_cuts}

In prior sections, we focused on separable relaxations of the form \((\y_j, \z_j, \ff_j) \in R_j\) to relax the graph or hypograph individual terms in the objective. In this section, we shift our attention to tightening the formulation \(\mathcal{U}\) that describes the variable space for \((\x, \y, \z)\), by adding additional inequalities in terms of \((\y, \z)\) based on the objective. Our goal is to introduce linear inequalities for these variables that appropriately bound the objective function. Specifically, we consider the projection onto the space defined by \((\y, \z)\) since the objective is fully described by these variables. By obtaining a tight description of this projection, we can derive dual bounds, thereby enhancing the optimization process. This approach can serve both as a preprocessing step and an iterative optimization technique.

\begin{lemma}
\label{lem:general_gradients}
    Let $\brObj \colon \Reals_+ \to \Reals$ be differentiable.  Let $\psi \colon \R^{m\cdot s} \times \R^{m \cdot s} \to \Reals$ defined by $\psi(\y,\z) = \sum_{j=1}^m \brObj\left(\sum_{t = 1}^s \frac{\yy_{jt}}{\zz_{jt}}\right)$. Then for any $\y,\z > 0$ we have
    $$
    \nabla \psi(\y,\z) \cdot (\y,\z)  \EQ  0.
    $$
\end{lemma}
\begin{proof}
Due to the separability of $\psi$, it suffices to prove this for $s=1$ so we forgo the index $j$. Similarly, due to symmetry in the variables, we can just show this w.l.o.g. for $j=1$. Notice that
    \begin{equation*}
    \pder{y_1} \brObj\left(\tfrac{y_1}{z_1} + \sum_{t=2}^\numYears\tfrac{y_t}{z_t} \right)  \EQ  \brObj'\left(\tfrac{y_1}{z_1} + \sum_{t=2}^\numYears\tfrac{y_t}{z_t}\right) \tfrac{1}{z_1} \qquad\text{and}
    \end{equation*}
    \begin{equation*}
    \pder{z_1} \brObj\left(\tfrac{y_1}{z_1} + \sum_{t=2}^\numYears\tfrac{y_t}{z_t}\right)  \EQ  \brObj'\left(\tfrac{y_1}{z_1} + \sum_{t=2}^\numYears\tfrac{y_t}{z_t}\right) \frac{-y_1}{(z_1)^2}.
    \end{equation*}
Then 
$
\nabla \psi(y,z) \cdot (y,z) = \brObj'(\tfrac{y}{z}) \frac{1}{z} y - \brObj'(\tfrac{y}{z}) \frac{y}{z^2} z = 0.
$
\end{proof}

We use this result in the cut based separation Algorithm \ref{alg:GradCuts}. 

\setlength{\textfloatsep}{5pt}
\begin{algorithm}[ht!]
\renewcommand{\baselinestretch}{1.0}\normalsize
\caption{Gradient cut separation at the root node}
\label{alg:GradCuts}
\begin{algorithmic}[1]
\State \textbf{Input:} $\XYZFeasPoly$, $\brObj$, $\overline{\psi}$, \textit{tol}
\State \textbf{Output:} $Q$ 
\State Find upper and lower bounds on $\yy_t$ and $\zz_t$.  That is, for $\alpha = \yy_t$ and $\zz_t$ for $t \in \years$:
$$
\ell^\alpha_t / u^\alpha_t  \EQ  \min / \max \{ \alpha : (\x, \y, \z) \in \XYZFeasPoly\}.
$$
\State Find upper and lower bounds on $\zz_t - \yy_t$ for each $t \in \years$:  
$$
\ell^{\z-\y}_t / u^{\z-\y}_t  \EQ  \min / \max \{ \zz_t - \yy_t : (\x, \y, \z) \in \XYZFeasPoly\}.
$$
\State Set 
$Q = \{ (\y, \z) : \mathbf\ell^{\y} \leq \y \leq \mathbf u^{\y}, \ \ \mathbf\ell^{\z} \leq \z \leq \mathbf u^{\z}, \ \ \mathbf\ell^{\z- \y} \leq \z - \y \leq \mathbf u^{\z - \y}\}.$
\While{$\max\{\psi(\y,\z) : (\y,\z) \in Q\} > \overline{\psi} + tol$}
\State Let $(\bf y^*, \bf z^*)$ be an optimal solution to continuous relaxation (optimization over $Q$).
\State Find upper and lower bounds on $\nabla F(\bf y^*, \bf z^*) \cdot (\y,\z)$:
$$
\ell^\nabla / u^\nabla  \EQ  \min / \max \{\nabla F(\bf y^*, \bf z^*) \cdot (\y,\z) :  (\x, \y, \z) \in \XYZFeasPoly\}
$$
\If{$\ell^\nabla > 0$ or $u^\nabla < 0$}
\State Add to $Q$ the bounds 
$
\ell^{\nabla} \leq \nabla F(\bf y^*, \bf z^*) \cdot (\y,\z) \leq u^{\nabla}.
$
\ElsIf{$\ell^\nabla < 0$ and $u^\nabla > 0$} {Break while loop}
\EndIf
\EndWhile
\State Add all bounds from $Q$ to IP and solve IP. 
\end{algorithmic}
\end{algorithm}
Throughout Algorithm~\ref{alg:GradCuts}, we maintain that $Q \supseteq \proj_{\y,\z}(\XYZFeasPoly_\textrm{B})$.  Indeed, we would like to have  $Q = \proj_{\y,\z}(\XYZFeasPoly_\textrm{B})$, but achieving such a representation is intractable and not strictly necessary for the optimization.

Let $\mathcal S$ be the set of points $(\bf y^*, \bf z^*)$ computed in steps 1-11 and 
let $\ell^{\nabla}_{\bf y^*,  \bf z^*}, u^{\nabla}_{ \bf y^*,  \bf z^*}$
be the corresponding valid bounds computed.  Then in step 12, we solve the following augmented stepwise approximation model:

\begin{modelCustom}{1}{f}[Stepwise Approximation with Preprocessing via Gradient Cuts of BR-Redistricting]\label{labeling:Step+Grad}
    We apply Approximation \ref{tech:Step} with gradient cuts to each district \(j\) in Model \ref{model:BR-Redistricting} independently:
    \begin{equation*}
    \max_{(\x,\y,\z)\,\in\,\XYZFeasPoly_\textrm{B}^\nabla}
    \left\{
    \sum_{j=1}^{\numDistricts}f_j
    \MID
    (\yy_j,\zz_j,\ff_j,\deltaVec^j)\in \lowhat{\step}_{\{b_j\}}\!\left(\brObj\right)
    \quad
    \forall\ j\in\{1,\dots,\numDistricts\}
    \right\}
    \end{equation*}
    where 
$$
\XYZFeasPoly_\textrm{B}^\nabla := \left\{ (\x,\y,\z)\,\in\,\XYZFeasPoly_\textrm{B} \quad    \mid  \quad\ell^{\nabla}_{\overline{\bf y}, \overline{\bf z}} \leq \nabla F(\overline{\bf y}, \overline{\bf z}) \cdot (\y,\z) \, \leq\, u^{\nabla}_{\overline{\bf y}, \overline{\bf z}}  \quad \forall\,  (\overline{\bf y}, \overline{\bf z}) \in \mathcal{S}\right\}.
$$
   \end{modelCustom}

\section{Multi-Ratio Representation Objectives} \label{sec:multi-ratios}
In this section, we generalize our approaches to work with multiple ratios to accomodate Model~\ref{model:CPVI}. In particular, we will work with objective functions from \eqref{eq:general}; that is, functions of the form
    $$
    \sum_{j=1}^{\numDistricts} \brObj\left(\beta_0 + \beta_j \sum_{t=1}^{\numYears} \frac{\yy_{jt}}{\zz_{jt}}\right),
    $$
where $s \geq 2$. Without
 loss of generality, we will assume for our models that $\beta_0 = 0$, and $\beta_j = 1$ for all $j =1, \dots, \numDistricts$.  This can be accomplished by absorbing the $\beta_0$ into the function $\phi$ and absorbing the $\beta_j$ into the ratio variables.

Thus,  we will focus on graph and hypograph relaxations of $\phi \colon \DomOne \to \R$, $\hat \phi\colon \DomOne_1 \times \dots \times \DomOne_\numYears \to \Reals$
and $\psi \colon \DomTwo_1 \times \dots \DomTwo_\numYears \to \R$ where $\DomOne_t \subseteq \R$, $\DomTwo_t \subseteq \R^2$, with $\DomOne_t :=\left \{ \tfrac{\yy_t}{\zz_t} | (\yy_t, \zz_t) \in \DomTwo_t \right\}$, $\DomTwo = \DomTwo_1 \times \dots \times \DomTwo_\numYears$, 
\begin{equation}\label{eq:AlternateExpressions:MultiRatio}
    \DomOne \coloneq \left\{r=\sum_{t =1}^{\numYears} \frac{\yy_t}{\zz_t}\MID(\yy_t,\zz_t)\in\DomTwo_t, t \in \yearsIter\right\}
    \quad\textrm{and}\quad
    \psi(\y,\z) \coloneq \hat \phi\left(\tfrac{\yy_1}{\zz_1}, \dots, \tfrac{\yy_\numYears}{\zz_\numYears} \right) \coloneq \brObj\left(\sum_{t=1}^\numYears \frac{\yy_t}{\zz_t}\right).
    \end{equation}

We begin by altering Technique \ref{tech:Step} to fit this objective and create a stepwise overestimation (depicted in Figure~\ref{fig:2d-stepwise}) of $\hat \phi$ where each step has a rectangular domain. In doing so, we create an $\varepsilon$-Relaxation.  Just as in the single-ratio stepwise relaxation, this technique allows us to sidestep dealing with the nonlinear complications of modeling $r = \tfrac{\yy}{\zz}$, which is why we work with the stepwise relaxation instead of a piecewise linear approximation.
\begin{figure}[h!]
    \centering    \includegraphics[width=0.4\linewidth]{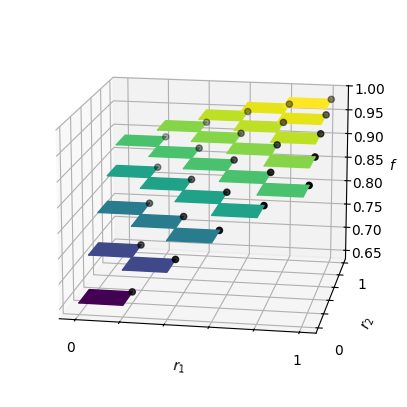}
    \vspace{-0.5cm}
    \caption{2D stepwise function.  On each rectangle $[b_{1i}, b_{1(i+1)}] \times [b_{2j}, b_{2(j+1)}]$ the stepwise plot takes the value $\phi(b_{1(i+1)} +b_{2(j+1)})$.  Thus, assuming this $\phi$ is an increasing function, this stepwise function will overestimate $\hat \phi(r_1,r_2)$.} 
    \label{fig:2d-stepwise}
\end{figure}

    \begin{technique}{Stepwise Approximation for Multi-Ratio Problems}{Step-multi-ratio}
    \label{tech:stepwise-multi}
Let $\brObj\colon \Reals \to \Reals$ be an increasing function, bounded above. Given a finite sequence of breakpoints $\{b_{tk}\}_{k\in\breakpointsIter}$ for each $ t\in \numYears$ such that $b_{t\ell} \geq \frac{\yy_t}{\zz_t}$ for all $t \in \yearsIter$ and $(\y,\z) \in \XFeasPoly$, define the following set:

\begin{subequations}\label{tech:Step:set-multi-ratio}
\vspace{1em}
$\steps  \DEF
\Bigg\{((\y,\z),\ff,\deltaVec) \in \XFeasPoly \times \Reals \times \{0,1\}^{\numYears \times \numBreakpoints}$
\begin{partialSetArray}
    \ff - \brObj\left(\sum_{t=1}^{\numYears} b_{t, k_t}\right)
        &\leq \bigN_{\mathbf{k}} \sum_{t=1}^{\numYears}(1 - \deltaVar_{tk_t})
        &\forall\ \mathbf{k} \in \breakpointsIter^\numYears
        \label{tech:Step:set:objective-multi-ratio}\\
    b_k \zz_t - \yy_t
        &\leq \bigM_{tk} \deltaVar_{tk}
        &\forall\ t \in \yearsIter,\ k \in \breakpointsIter
        \label{tech:Step:set:ratio-multi}\\
    \deltaVar_{t(k-1)} 
        &\leq \deltaVar_{tk}
        &\forall\ t \in \yearsIter,\ k \in \breakpointsIter
        \label{tech:Step:set:symmetry-multi}
\end{partialSetArray}
\end{subequations}

\end{technique}

\begin{proposition}\label{lem:StepOverestimate:MultiRatio}
    If $\brObj$ is increasing over $\DomOne$, then $\proj_{(\y,\z,\ff)}\left(\steps\right)$ is an $\varepsilon$-relaxation of $\hyp(\psi)$ with
        \begin{equation*}
        \hyp(\psi) \subseteq \proj_{(\y,\z,\ff)}\left(\steps\right)
        \qquad\textrm{and}\qquad
        \varepsilon = \max_{t \in \yearsIter, k_t\,\in\,\breakpointsIter} \left\{\brObj\left(\sum_{t = 1}^\numYears b_{t,(k_t + 1)}\right) - \brObj\left(\sum_{t = 1}^\numYears b_{t,k_t}\right)\right\}.
        \end{equation*}
    \end{proposition}
    \begin{proof}
Fix $(\y,\z) \in \XFeasPoly $
and let $ [b_{1,(k_1-1)}, b_{1,k_1)}) \times \dots  \times [b_{\numYears, (k_{\numYears-1)}}, b_{\numYears,k_\numYears)})$ 
be the half-open box containing $\left(\tfrac{\yy_1}{\zz_1}, \dots, \tfrac{\yy_{\numYears}}{  \zz_{\numYears}}\right)$.
Then $\yy_t < b_{t k_t} \zz_t$, and thus $0 <  b_{t k_t} \zz_t - \yy_t$.

Thus  any  $(\y, \z, \ff, \deltaVec) \in \steps$  
requires that 
 $\deltaVar_{t, k'_t} = 1$ for all $k'_t \geq k_t$ by \eqref{tech:Step:set:ratio-multi}.

But then \eqref{tech:Step:set:objective-multi-ratio} requires that 
$
\ff - \phi\left(\sum_{t =1}^\numYears b_{t,k_t}\right) \leq 0,
$, i.e.,  $
\ff  \leq  \phi\left(\sum_{t =1}^\numYears b_{t,k_t}\right).
$

Further, since $\phi$ is increasing, we have
  $$
    \brObj\left(\sum_{t=1}^{\numYears}b_{i(t_i^*-1)}\right)  
        \LEQ   \brObj\left(\sum_{t=1}^{\numYears}\frac{\yy_t}{\zz_t}\right)  
        \EQ  \psi(\y,\z)  \LEQ  \brObj\left(\sum_{t=1}^{\numYears}b_{tk_t^*}\right).
    $$

Clearly, setting $\deltaVar_{t,k'_t} = 0$ for all $k'_t < k_t$ establishes that if $(\y,\z,\ff) \in \hyp(\psi)$, then $(\y,\z,\ff,\deltaVec) \in \steps$.
Further, if $(\y,\z,\ff) \in \proj_{\y,\z,\ff}(\steps)$ then either $(\y,\z,\ff) \in \hyp(\psi)$ or 
$
\brObj\left(\sum_{t=1}^{\numYears}\frac{\yy_t}{\zz_t}\right) < \ff.
$
In the latter case, this implies that 
$
\ff - \brObj\left(\sum_{t=1}^{\numYears}\frac{\yy_t}{\zz_t}\right) < \varepsilon,
$
establishing the $\varepsilon$-relaxation.

    \end{proof}

\paragraph{Partisan Election Stepwise Model}
We now turn our attention to the partisan election objective given in \eqref{obj:CPVI}. For simplicity, define 
    $$
    \hat \phi(r_1, r_2)  \EQ  \brObj\left(\frac{50 (r_1 + r_2)-51.69}{4.8}\right)
    $$
so that $\hat \phi\left(\frac{\DV_{16}}{\TV_{16}}, \frac{\DV_{20}}{\TV_{20}}\right) = P(\textrm{a Democrat is elected})$ and we are left to manipulate the interior. 
    
For each $i \in \parcels$ and $s \in \years = \{2016,2020\}$, let $\hypertarget{def:TV}{\TV_{it}}$ and $\hypertarget{def:DV}{\DV_{it}}$ respectively represent the total number of votes cast and the number of votes cast in favor of the Democrat party within parcel $i$ during year $t$. Then define $\ff_j$ to be the approximate objective function value for each district $j \in\{1,\dots,\numDistricts\}$. Now we can adapt Technique \ref{tech:stepwise-multi} to this context by letting $\yy_{jt} = \sum_{i \in \parcels}\TV_{it}\xx_{ij}$ and $\zz_{jt} = \sum_{i \in \parcels}\DV_{it}\xx_{ij}$.

\begin{modelCustom}{2}{f}[Stepwise Approximation Multi-Ratio for CPVI-redistricting]\label{labeling:StepMultiRatio}
We apply Relaxation \eqref{tech:Step:set-multi-ratio} to each district $j$ in Model \ref{model:CPVI} independently:
\begin{equation*}
\max_{(\x,\y,\z)\,\in\,\XYZFeasPoly_\textrm{P}}
\left\{
\sum_{j=1}^{\numDistricts}f_j
\MID
(\y_j,\z_j,\ff_j,\deltaVec^j)\in \lowhat{\step}_{\{b_{tk}\}}^{\numYears}\!\left(\psi\right) 
\quad
\forall\ j\in\{1,\dots,\numDistricts\}
\right\}.
\end{equation*}
\end{modelCustom}

In our CPVI-redistricting setting, we have $\yy_{jt} \geq 0$ and $\zz_{jt} \leq (1+\tau)\tildep$; thus $\brObj\left(\sum_{t=1}^{\numYears} b_{jt\numBreakpoints_j}\right) - \brObj\left(\sum_{t=1}^{\numYears} b_{jtk_{jt}}\right)$ is the smallest valid value of $\bigN_{j\mathbf{k}}$ while $b_{jtk_{jt}}(1+\tau)\tildep$ is sufficiently large for $\bigM_{jtk_{jt}}$ for all $j\in\{1,\dots,\numDistricts\}$, $k\in\breakpointsIter$, and $t\in\yearsIter$. Good bounds on the features $\y$ and $\z$ are again important for identifying good breakpoints in much the same way that they were for Technique \ref{tech:Step}.

The multiple ratios being approximated make finding any good closed form  breakpoint scheme difficult. We used SciPy to find breakpoints with low maximum error in addition to expected error.

Gradient Cuts, like those in Section \ref{sec:gradient_cuts}, can indeed be computed for multi-ratio objectives. However, in our attempt, the cuts were very slow to compute and did not reach sufficient depth to improve the bounds of the continuous relaxation.  More investigation into this direction is needed.

\section{Computations}
\label{sec:Computations-and-Conclusions}
Optimization models were implemented in Python 3.7 using the Python API for Gurobi version 9.1.2~\cite{gurobi}. Computations were run on a desktop running Windows 10 Pro with 64GB of RAM, using an Intel Core i9-10980XE CPU (3.00GHz, 18 Cores, 36 Threads) and an AMD Radeon Pro WX 2100.  We conduct experiments at the county level, which serves as a tractable testbed for developing and validating our optimization techniques for nonconvex representation objectives. At this granularity, the coarseness of building blocks limits achievable population balance; we employ a tolerance of $\tau = 20\%$ (i.e., $\pm 10\%$), which reflects the constraints imposed by keeping counties whole. For reference, Alabama has 67 counties, 1,437 tracts, 1,837 precincts, and 185,976 blocks~\cite{buchanan2024widespread}. While computational redistricting papers working at the tract or precinct level typically use deviations of $\pm 0.5\%$ to $\pm 2\%$~\cite{validi2022imposing, shahmizad2024political, deford2021recombination, autry2021metropolized}, and enacted congressional plans achieve near-perfect population equality by working at the block level, county-level instances are fundamentally more constrained. Extending our techniques to finer granularities is a direction for future work.
Code for our computations is provided at \url{https://github.com/RobertHildebrand/Redistricting-Representation}.

We present the computational results discussed in this paper, with additional tables detailing our studies available in Appendix~\ref{appendix:computations}. Our computational tables include several key metrics used to evaluate MIP relaxations:

\begin{itemize}
    \item \textbf{Dual Bound}: This refers to the best dual bound obtained from the MIP approximation of the nonlinear optimization problem. Note: for the VA-PWL approach, the dual bounds are not necessarily valid. We have lightened the numbers there to indicate they should not be used.  Bounds from the stepwise approaches are indeed proper bounds.  Bounds from our other approaches are only accurate to a certain decimal.  See Appendix~\ref{app:error} for a discussion on the accuracy of these bounds.
    \item \textbf{MIP Obj.}: This metric represents the highest objective function value reported by the MIP solver when solving the approximate model.
    \item \textbf{Time/MIP Gap}: Here, we display the final primal-dual gap of the MIP, unless the solver achieved optimality, in which case the runtime is provided.
    \item \textbf{Primal Bound}: This value is obtained by evaluating the best solution found by the MIP under the original nonlinear objective function.
    \item \textbf{Gap}: This is calculated by comparing the Dual Bound to the Primal Bound: $\text{Gap} = \frac{\text{Dual Bound} - \text{Primal Bound}}{\text{Primal Bound}} \times 100\%$, providing insight into the accuracy of the approximation.
\end{itemize}

These metrics allow us to analyze the following:

\begin{enumerate}
    \item The quality of the dual bound produced by each model.
    \item The best possible dual bound a model can achieve as the difference between the Dual Bound and the MIP Obj. approaches zero (i.e., as the MIP Gap approaches zero).
    
    \item The effectiveness of the primal solutions found by each model (i.e., does the solver's primal search perform better with a specific model?).
    \item The overall Gap, derived from the dual and primal bounds. Note that since the models are approximations, this gap inherently includes both the MIP Gap and the error introduced by the approximation.
\end{enumerate}

\subsection{Single Ratio Results}
Initially, we compare stepwise and piecewise-linear models.  Note that piecewise-linear models need to handle two layers of approximation: one from approximating the function $\brObj$ and another from approximating the ratio. The choice of the number of breakpoints for each model can affect the outcome and solve times.  We display a comparison with the number of breakpoints in Figure~\ref{fig:pwl-step-breakpoints}.
    \begin{figure}[!ht]
        \centering
        \begin{tikzpicture}
        \node at (0,0) {\includegraphics[width=0.7\textwidth]{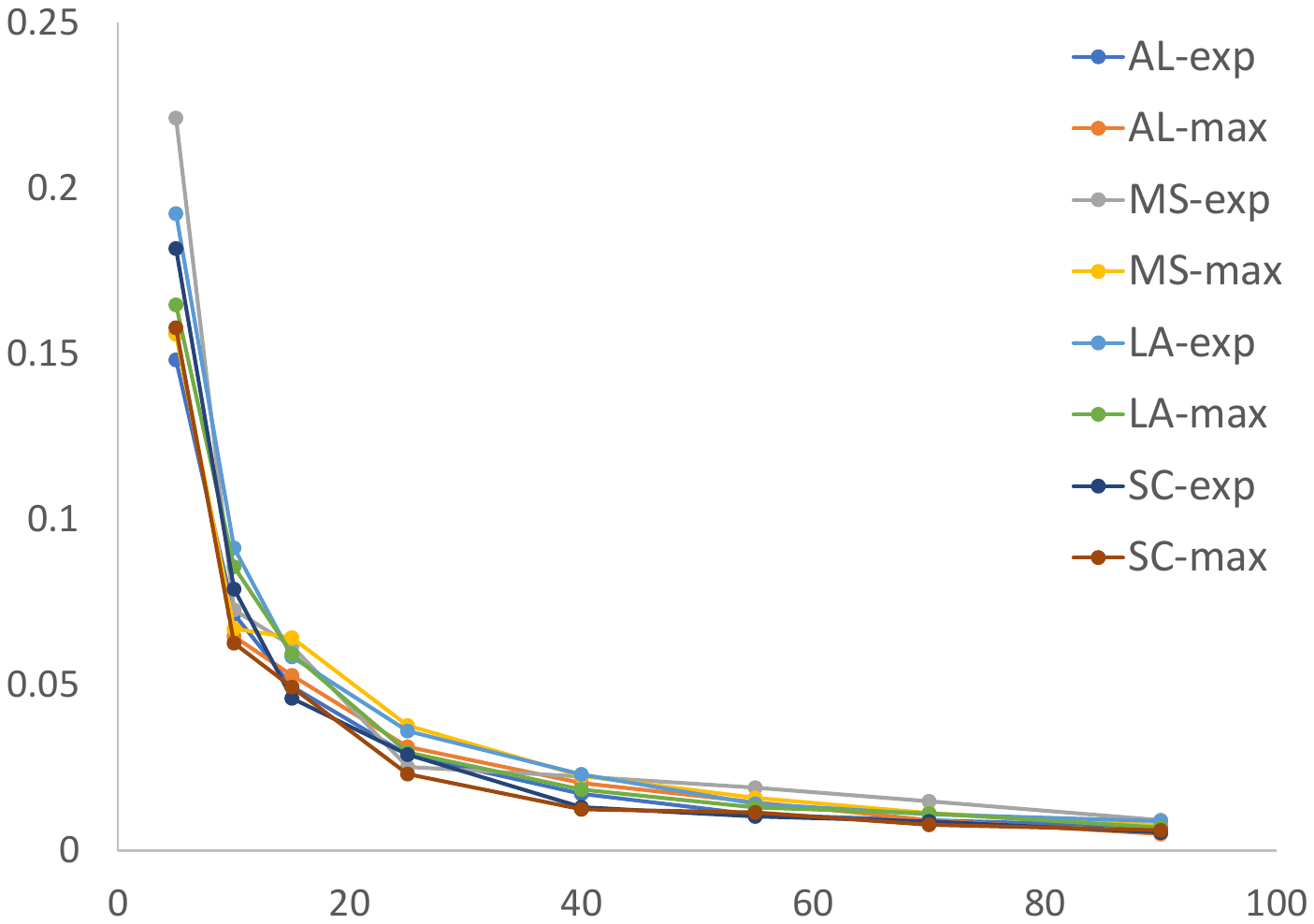}};
        \node at (0,-4) {$\numBreakpoints$};
        \node[rotate = 90] at (-5.4,0) {Error $\left(\frac{\textrm{MIP Obj.}-\textrm{Primal bound}}{k}\right)$};
        \end{tikzpicture}
    \caption{Per-district approximation error in the best solutions by state and number of breakpoints.}
    \label{fig:pwl-step-breakpoints}
    \end{figure}
    
In Table~\ref{tab:PWLvStepComparison} we report the best $\ell$ for each approach when provided 2 hour time limit. Although stepwise functions are traditionally considered a weaker approximation, as mentioned earlier, the stepwise approach sidesteps nonconvexity and performs favorably in this context. Perhaps surprisingly, the stepwise approaches tend to perform better than the piecewise-linear approaches. The two stepwise models have the smallest optimality gaps while providing improved bounds over the standard PWL model. The standard piecewise-linear model performs the worst outside of a few good objective values.

    \begin{table}[ht!]
    \caption{The best two-hour runs of the PWL, VA-PWL, and Stepwise approaches on key states. An * indicates that a MIP-optimal solution was found in that instance, and thus that model could not produce a better bound.  We shade dual bounds that are not necessarily valid due to inaccuracy in the model. The best known bounds are in bold.} \label{tab:PWLvStepComparison}
    \centering
    \begin{tabular}{|ccccc|ccc|}\hline
        \textbf{State} & $m$ & $n$ & \textbf{Model} & $\numBreakpoints$ & \textbf{Dual Bnd} & \textbf{Primal Bnd} & \textbf{Gap} \\
        \hline
        \multirow{4}{*}{AL} & \multirow{4}{*}{7} & \multirow{4}{*}{67}
              & PWL          & 60     &  2.259 & 1.571 & 30\% \\
            &&& VA-PWL       & 60     & \graycell 1.714 & 1.389 & 19\% \\
            &&& Step-Max     & 70     & 1.775 & 1.570 & 12\% \\
            &&& Step-Exp.    & 70     & \textbf{1.767} & \textbf{1.574} & 11\% \\
        \hline
        \multirow{4}{*}{MS} & \multirow{4}{*}{4} & \multirow{4}{*}{82}
              & PWL          & 60     &  1.878 & 1.655 & 12\% \\
                      &&& VA-PWL*      & 60     & \graycell 1.643 & 1.467 & 11\% \\
                      &&& Step-Max     & 70     & \textbf{1.735} & 1.647 & 5\%  \\
                      &&& Step-Exp.    & 40     & 1.781 & \textbf{1.664} & 7\%  \\
        \hline
        \multirow{4}{*}{LA} & \multirow{4}{*}{6} & \multirow{4}{*}{64}
              & PWL          & 60     & \graycell 2.193 & \textbf{1.575} & 28\% \\
                      &&& VA-PWL       & 60     & \graycell 1.772 & 1.433 & 19\% \\
                      &&& Step-Max     & 40     & 1.768 & 1.562 & 12\%  \\
                      &&& Step-Exp.    & 55     & \textbf{1.743} & 1.562 & 10\%  \\
        \hline
        \multirow{4}{*}{SC} & \multirow{4}{*}{7} & \multirow{4}{*}{46}
              & PWL          & 60     & \graycell 1.724 & 1.313 & 24\% \\
                      &&& VA-PWL*      & 60     & \graycell 1.322 & 1.142 & 14\% \\
                      &&& Step-Max*    & 70     & 1.402 & 1.327 & 5\%  \\
                      &&& Step-Exp.*   & 70     & \textbf{1.377} & \textbf{1.329} & 4\%  \\
        \hline
    \end{tabular}
    \end{table}

Next, we compare our logarithmic piecewise models with the best stepwise model when varying $\ell$. We exhibit runs with the best solutions measured with the MIP objective.  We show this in Table~\ref{tab:BestInstances}.  Impressively, the logarithmic models run for just 2 hours were competitive with the stepwise models run for 10 hours.  This shows that the logarithmic models produced strong dual bounds significantly faster than the stepwise models.  Maps of the prevailing primal solutions are given in Figure \ref{fig:BestInstances} under Appendix \ref{app:Maps}.
    \begin{table}[ht!]
      \centering
      \caption{Problem instances with largest primal bound  values and smallest gaps.}
      \label{tab:BestInstances}
      \begin{tabular}{|cccccc|ccc|}
        \hline
        \textbf{State} & $m$ & $n$ & \textbf{Model} & $\numBreakpoints$ & \textbf{Time} & \textbf{Dual Bnd} & \textbf{Primal Bnd} & \textbf{Gap} \\
        \hline
        \multirow{3}{*}{AL} & \multirow{3}{*}{7} & \multirow{3}{*}{67} & BN-Full & 7 & 2 hrs. & 1.818 & 1.582 & 13.0\% \\
        &&& LogE & 90 & 2 hrs. & 1.797 & 1.580 & 12.1\% \\
        &&& Step-Max & 90 & 10 hrs. & \textbf{1.783} & \textbf{1.586} & \textbf{11.0\%} \\
        \hline
        \multirow{3}{*}{MS} & \multirow{3}{*}{4} & \multirow{3}{*}{82} & Step-Exp & 40 & 2 hrs. & 1.781 & \textbf{1.664} & 6.5\% \\
        &&& LogE & 90 & 2 hrs. & \textbf{1.726} & 1.642 & \textbf{4.9\%} \\
        &&& Step-Max & 90 & 10 hrs. & 1.728 & 1.639 & 5.2\% \\
        \hline
        \multirow{3}{*}{LA} & \multirow{3}{*}{6} & \multirow{3}{*}{64} & Step-Exp & 25 & 2 hrs. & 1.829 & 1.579 & 13.7\% \\
        &&& BN-Full & 6 & 2 hrs. & 1.785 & 1.588 & 11.0\% \\
        &&& Step-Max & 90 & 10 hrs. & \textbf{1.722} & \textbf{1.589} & \textbf{7.8\%} \\
        \hline
        \multirow{3}{*}{SC} & \multirow{3}{*}{7} & \multirow{3}{*}{46} & Step-Exp & 90 & 2 hrs. & 1.394 & 1.334 & 4.3\% \\
        &&& BN-Full & 7 & 2 hrs. & \textbf{1.364} & \textbf{1.334} & \textbf{2.2\%} \\
        &&& Step-Exp & 90 & 10 hrs. & 1.380 & 1.334 & 3.3\% \\
        \hline
      \end{tabular}
    \end{table}

The simplicity of the stepwise models is convenient for implementation, but it is clear that the dominant models are 2d-piecewise-linear models with a logarithmic encoding scheme.  It is not clear that the $\BNFull$ model with its compact formulation is an improvement over directly modeling each polyhedral domain in the LogE method.  In fact, in subsequent computations, we improve the LogE model by incorporating improved bounds.  This step would be complicated to implement in the BN-Full model.

\newcommand{\greencell}{\cellcolor[rgb]{ .886,  .937,  .855}}
\newcommand{\yellowcell}{\cellcolor[rgb]{ 1,  .949,  .8}}

    \begin{table}[ht!]\centering
    \begin{tabular}{|ccc|cccc|}\hline
    \textbf{Model} & $\numBreakpoints$ & \textbf{Time}    & \textbf{AL (7,67)} & \textbf{MS (4,82)} & \textbf{LA (6,64)} & \textbf{SC(7,46)} \\
    \hline
    BN-Full     & 4   & 2 hrs.      & 2.064      & 1.962      & 2.08       & 1.692      \\
    BN-Full     & 6   & 2 hrs.      & 1.812      & 1.752      & 1.785      & 1.405      \\
    BN-Full     & 7   & 2 hrs.      & 1.818      & 1.753      & 1.834      & 1.364      \\
    LogE   & 25  & 2 hrs.      & 1.872      & 1.835      & 1.800      & 1.454      \\
    LogE   & 90  & 2 hrs.      & \textbf{1.797}      & 1.726      &\textbf{ 1.736}      & \textbf{1.366 }     \\
    LogE   & 200 & 2 hrs.      & 1.830      &\textbf{ 1.708 }     & 1.761      & 1.479      \\
    \hline
    PWL       & 10  & 2 hrs.      &  \graycell 2.317      & \graycell 1.723      & \graycell 1.958      & \graycell 1.503      \\
    PWL       & 60  & 2 hrs.      &  2.259      &  1.878      &  2.193      & 1.724      \\
    VA-PWL    & 10  & 2 hrs.      & \graycell 1.787      & \graycell 1.690      & \graycell 1.778      & \graycell 1.411      \\
    VA-PWL    & 60  & 2 hrs.      & \graycell 1.714      & \graycell 1.651      & \graycell 1.772      & \graycell 1.322      \\
    \hline
    Step-Exp. & 10  & 2 hrs.      & 2.171      & 2.049      & 2.177      & 1.692      \\
    Step-Exp. & 25  & 2 hrs.      & 2.131      & 1.812      & 1.829      & 1.484      \\
    Step-Exp. & 55  & 2 hrs.      & 1.794      & 1.899      & 1.749      & 1.403      \\
    Step-Exp. & 70  & 2 hrs.      & 1.767      & 1.898      & 1.743      & 1.377      \\
    Step-Exp. & 90  & 2 hrs.      & 1.787      & 1.864      & 1.879      & 1.394      \\
    Step-Exp. & 90  & 10 hrs.     & 1.764      & 1.862      & 1.716      & 1.380      \\
    \hline
    Step-Max  & 10  & 2 hrs.      & 2.202      & 2.005      & 2.128      & 1.690      \\
    Step-Max  & 25  & 2 hrs.      & 1.982      & 1.806      & 1.874      & 1.479      \\
    Step-Max  & 55  & 2 hrs.      & 1.782      & 1.752      & 1.836      & 1.419      \\
    Step-Max  & 70  & 2 hrs.      & 1.775      & 1.735      & 1.850      & 1.402      \\
    Step-Max  & 90  & 2 hrs.      & 1.825      & 1.728      & 1.794      & 1.433      \\
    Step-Max  & 90  & 10 hrs.     & 1.783      & 1.728      & 1.722      & 1.370      \\
    \hline\hline 
    \multicolumn{3}{|c|}{Best Known Primal Bound}  & 1.586      & 1.664      & 1.589      & 1.334      \\
    \hline
    \end{tabular}
    \caption{Summarized dual bounds from approximation techniques on four states for the Black Representation objective.}
    \end{table}

We proceed by computing improved bounds and implementing the LogE model.  In Figure~\ref{fig:LogE_with_bounds} we show that within 2 hours of computation time, we can certify the optimality of our best solutions (either exactly, or within a few percent).

\begin{figure}[!ht]
\centering
        \includegraphics[scale=0.55, trim=0 360pt 0 0, clip]{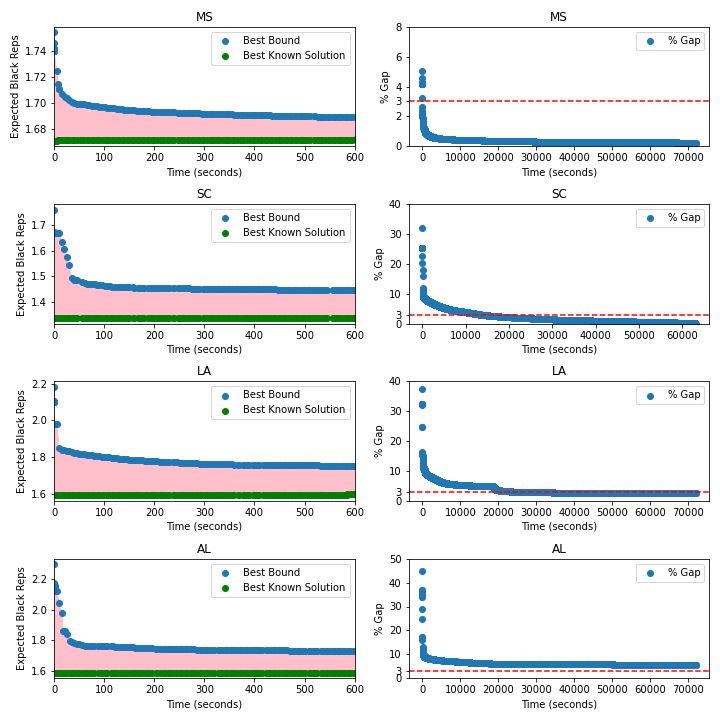}
        
        \includegraphics[scale=0.55, trim=0 0 0  360pt, clip]{BVAP_LogE_with_bounds.png}
    \caption{Progress after implementing LogE with bounds. We add, arbitrarily, a line at a 3\% gap, as a reasonable bound to attain.}
    \label{fig:LogE_with_bounds}
\end{figure}
In Appendix~\ref{appendix:computations} we present more comprehensive results on PWL objective (Table~\ref{tab:ResultsPWL2hr}), VAP Approximated PWL objective (Table~\ref{tab:ResultsPWLApprox}), Stepwise max and Stepwise expected (Tables~\ref{tab:ResultsStepMax},~\ref{tab:ResultsStepExp},~\ref{tab:ResultsLongRun}).  Then we have BN-Full computations for different $\nu$ values (Table~\ref{tab:BNFull}) and LogE for different number of polyhedral domains (Table~\ref{tab:LogE}).

\subsubsection{Gradient Cut Results}
Lastly, we demonstrate that near optimality can be obtained by convexification of the projection into the space of bvap and vap variables, $\y$ and $\z$, via improved bounds and gradient cuts.  We see this progress in Figure~\ref{fig:gradient-cut-progress}.  This achievement is surprising as the problem is nonconvex, but even so, it seems that these cuts are very valuable to proving optimality. 

We do not compare these values to the LogE with bound computations in terms of runtime since the convexification via gradient cuts is given ample time for its procedure.  The resulting cuts can be viewed as preprocessing to optimizing the LogE or any other objective.

\begin{figure}[!ht]
    \centering
        \includegraphics[scale=0.53, trim=0 310pt 0 0, clip]{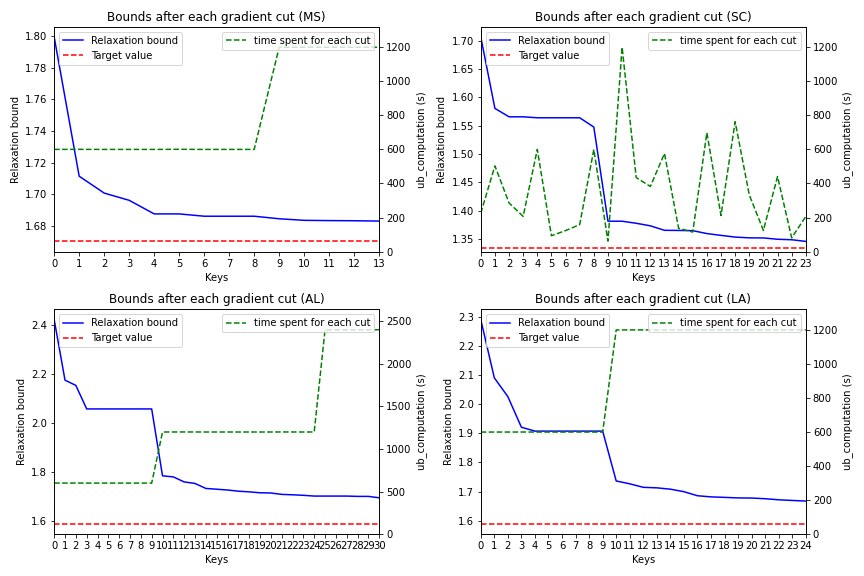}\\
         \vspace{5mm}
    \includegraphics[scale=0.53, trim=0 23pt 0 290pt, clip]{gradient_cuts_with_time.png}

    \caption{Progress of gradient cuts for each state. Allowed time for computation of each gradient cut was increased once insufficient progress was made. Thus, sometimes there are large jumps in improvement due to an increased time allowed to compute a deeper cut. These results do not include solving the IP to optimality after these bounds are added in.}
    \label{fig:gradient-cut-progress}
\end{figure}

\subsubsection{HVAP + BVAP Results}
We conducted preliminary computations for HVAP+BVAP objective using the Step-Exp model.  See Table~\ref{tab:ResultsHVAP}.  Fortunately, the added (minor) complexity to the problem does not seem to significantly affect the difficulty of solving the problem.  Further computations with our advanced models could be run to seek optimality.

\subsection{Multi-Ratio Results}

For multi-ratio results, we limited ourselves to stepwise models for ease of implementation.  We show that 10 hours of computation time can secure near optimality on several easier instances.  We provide these results in Table~\ref{tab:CPVIResultsLongRuns}. Maps of the solutions for South Carolina and Iowa are given in Figure \ref{fig:BestCPVI} under Appendix \ref{app:Maps}.

For improved performance, it seems likely that a LogE variant would be ideal.  This requires a more complicated formulation due to the increased number of variables in each part of the objective function.

    \begin{table}[ht!]
      \centering
      \caption{Results for Step-Exp. model of the Democrat gerrymandering objective (10 hours each)}
        \begin{tabular}{|cccc|ccc|cc|}\hline
        \textbf{State} & $\numDistricts$ & $\numParcels$ & $\numBreakpoints$ & \textbf{Dual Bnd} & \textbf{MIP Obj.} & \textbf{Time/MIP Gap} & \textbf{Primal Bnd} & \textbf{Gap} \\
        \hline
        AL    & 7     & 67    & 90    & 1.919 & 1.665 & 13.2\%  & 1.605 & 16\% \\
        MS    & 4     & 82    & 90    & 1.728 & 1.280 & 25.9\%  & 1.216 & 30\% \\
        LA    & 6     & 64    & 90    & 1.752 & 1.677 & 4.2\%   & 1.625 & 7\%  \\
        SC    & 7     & 46    & 90    & 2.340 & 2.140 & 8.6\%   & 2.055 & 12\% \\
              &       &       &       &       &       &         &       &      \\
        OK    & 5     & 77    & 90    & 0.171 & 0.171 & 1.1 hrs & 0.163 & 5\%  \\
        NM    & 3     & 33    & 90    & 2.218 & 2.218 & 45 min  & 2.136 & 4\%  \\
        KS    & 4     & 105   & 90    & 1.174 & 1.146 & 2.4\%   & 1.095 & 7\%  \\
        IA    & 4     & 99    & 90    & 1.678 & 1.678 & 2.8 hrs & 1.624 & 3\%  \\
        AR    & 4     & 75    & 90    & 0.889 & 0.885 & 0.4\%   & 0.870 & 2\%  \\\hline
        \end{tabular}
      \label{tab:CPVIResultsLongRuns}
    \end{table}

\section{Conclusions}

This study marks an initial exploration into optimizing nonlinear functions over the complex feasible region that characterizes the redistricting problem.

Key takeaways include:
\begin{enumerate}
    \item \textbf{Stepwise Models}: The stepwise models demonstrated their effectiveness at this scale, offering simplicity in coding and interpretation. Simplicity is crucial, especially in political contexts where the underlying mathematics must be communicated to non-mathematicians. While subtle differences exist between the Step-Exp and Step-Max approaches, such as the impact of breakpoint selection on model strength, neither emerged as a clear winner. An adaptive approach to breakpoint selection may provide superior performance over both. Stepwise models were especially useful in multi-ratio problems, where gradient cut techniques proved unexpectedly ineffective. Although further investigation is needed, stepwise approximations still provided reasonably informative bounds in these settings.

    \item \textbf{Logarithmic Size Models}: The logarithmic size models consistently generated near-optimal dual bounds significantly faster than the stepwise models. The LogE and BN-Full models exhibited similar performance, but the LogE model holds a distinct advantage due to its capacity to incorporate tighter bounds on the feasible region, thus refining the choice of vertices in the piecewise-linear approximation. However, this model demands greater detail in its programming. For single-ratio objectives, a combination of variable bounds, gradient cuts, and the LogE formulation proved to be an effective strategy \textemdash though some parameter tuning is necessary to optimize performance for specific instances.

    \item \textbf{Preprocessing Bounds}: Figure~\ref{fig:bvap_ranges} shows that even modest work (that can be done in parallel) can greatly improve known bounds on the core variables. As Figure~\ref{fig:LogE_with_bounds} shows, these bounds are very effective at reducing initial gaps substantially. Preprocessing should be considered a foundational step in any attempt to improve dual bounds for redistricting formulations.

    \item \textbf{Gradient Cuts}: The Gradient-Cuts method emerged as a particularly promising approach. Despite being the most complex method, it delivered bounds comparable to the LogE approach within two hours and approached near-optimal bounds with extended computation time (up to ten hours). This is particularly notable as Gradient-Cuts rely solely on optimizing linear functions over the feasible region. Fine-tuning the runtime per cut and the sequence in which cuts are added may further enhance performance. In multi-ratio problems, however, gradient cuts were surprisingly ineffective and were ultimately not used. The potential for parallelizing gradient cuts \textemdash especially across multiple objective functions \textemdash may provide future gains in efficiency for problems involving trade-offs or Pareto frontiers.
\end{enumerate}

If the goal is to explore a variety of viable districting plans, ensemble methods remain well-suited. The exact optimization techniques described here are intended to help understand theoretical bounds on representational constraints and guide further computational investigations.

\subsection{Future Work}

The study highlighted the considerable challenge of optimizing nonconvex objectives over contiguous maps, compared to linear objectives. Notable improvements were achieved through tighter bounds on objective variables, with the combination of bounds and additional gradient cuts proving highly efficient in optimizing the BVAP objective with a single ratio. However, several implementation questions remain, such as determining the optimal number of cuts, whether to execute them sequentially or in parallel, and how much time to allocate to optimizing bounds for each cut.

For the multi-ratio partisan objective, the performance of these cuts had limited impact on reducing dual bounds. This indicates a need for further research to enhance the effectiveness of these methods in this context.

Future research should also explore alternative approaches to these challenges, such as a spatial branch-and-bound approach (see recent work~\cite{hubner2023spatial}), different piecewise-linear formulations (see recent work~\cite{lyu2023building}), and variations of recent convexification techniques for fractional objective functions~\cite{fractional-programs}. These avenues could offer new insights and potentially more effective strategies for addressing the complexities of the redistricting problem.

\subsection*{Data, Funding, and Disclosures}
\textbf{Data:} The datasets generated and analyzed during the current study were generated from data that is publicly available at \url{https://redistrictingdatahub.org/}.  Formats of this data that were set up for our optimization models are available from the corresponding author upon request.

\doubleblind{\noindent \textbf{Funding:} J. Fravel and R. Hildebrand were partially funded by ONR Grant N00014-20-1-2156.  R. Hildebrand was also partially supported by AFOSR grant FA9550-21-1-0107. Any opinions, findings, and conclusions or recommendations expressed in this material are those of the authors and do not necessarily reflect the views of the Office of Naval Research or the Air Force Office of Scientific Research.  R. Hildebrand and N. Goedert were also partially funded by the 2020-2021 Policy+ grant from Virginia Tech.
}

\noindent \textbf{Conflicts of Interest:} The authors declare that they have no known competing financial interests or personal relationships that could have appeared to influence the work reported in this paper. This statement is made in the interest of full disclosure and to ensure the integrity of the research

\appendix
\section{Gray code and Binarization and a Proof of Proposition \ref{prop:BN-graycode}}
\label{appendix:gray_code}
The angle of an input point  determines its Gray code representation in the model.  
For $(\overline\yy,\overline\zz) \in \Reals^2_+$ define its measured angle $\measuredangle{(\overline\zz,\overline\yy)} \coloneq \tan^{-1}(\overline\yy/\overline\zz)$. When $\overline\zz = 0$, we instead define the angle to be $\sin^{-1}(\overline\yy/ |\overline\yy|)$. 
\begin{proof}[Proof of Proposition \ref{prop:BN-graycode}]
First, consider $\numBreakpoints=1$ and assume that $\measuredangle{(\overline\zz,\overline\yy)} \in \left[0,\tfrac{\pi}{4}\right]$. If $\measuredangle{(\overline\zz,\overline\yy)} \in \left(\tfrac{\pi}{8},\frac{\pi}{4}\right]$ then $i = 0$ and rotating $(\overline\zz,\overline\yy)$ clockwise $\frac{\pi}{8}$ radians results in $\measuredangle{(\xiVar_1,\tilde\etaVar_1)} \in \left(0,\tfrac{\pi}{8}\right]$ with $\tilde\etaVar_1 \geq 0$. No mirroring step is required, so $\etaVar_1 = \tilde\etaVar_1$ and $\delta_1 = 0 = \alphabm^0$. If, on the other hand, $\measuredangle{(\overline\zz,\overline\yy)} \in \left[0,\tfrac{\pi}{8}\right]$ then $i = 1$ and rotating $(\overline\zz,\overline\yy)$ clockwise $\frac{\pi}{8}$ radians results in $\measuredangle{(\xiVar_1,\tilde\etaVar_1)} \in \left[-\tfrac{\pi}{8},0\right]$ with $\tilde\etaVar_1 < 0$. In this case we set $\etaVar_1 = -\tilde\etaVar_1$ and $\delta_1 = 1 = \alphabm^1$ so that $\measuredangle{(\xiVar_1,\etaVar_1)} \in \left[0,\tfrac{\pi}{8}\right]$. Thus, the proposition holds for $\numBreakpoints = 1$. 

We take a break here to consider how our binary representations behave when a digit is appended. Suppose we have a $(\numBreakpoints-1)$-digit code $\deltabm^{\numBreakpoints-1}$ to which we append a $\numBreakpoints^\text{th}$ digit $\delta_\numBreakpoints$. There exists some nonnegative integer $q$ for which $\deltabm^{\numBreakpoints-1}  =  \alphabm^q$; let $\betabm^q$ be the traditional binary representation of $q$. By composing equations \eqref{eq:bindef} and \eqref{eq:alternative-definition} we get an equation for the integer $q'$ represented by the new vector $\deltabm^{\numBreakpoints} = (\deltabm^{\numBreakpoints-1},\delta_\numBreakpoints)$:
    \begin{align}\begin{split}\label{eq:GrayCodeNewi}
    q'  &\EQ  \sum_{k=1}^\numBreakpoints2^{\numBreakpoints-k}\left(\sum_{j=1}^k\delta^\numBreakpoints_j\right)\MOD2 \\
           &\EQ  \sum_{k=1}^{\numBreakpoints-1}2^{\numBreakpoints-k}\beta^q + \left(\sum_{k=1}^\numBreakpoints\delta^\numBreakpoints_k\right)\MOD2 \\
           &\EQ  2q_{\numBreakpoints-1} + \Vert\deltabm^\numBreakpoints\Vert_0\ \MOD2
    \end{split}\end{align}

Now assume, for the sake of induction, that $\numBreakpoints > 1$ and the proposition holds up to $\numBreakpoints-1$. If $\measuredangle{(\xiVar_{\numBreakpoints-1}, \etaVar_{\numBreakpoints-1})} \in \left(\tfrac{1}{2^{\numBreakpoints}}\frac{\pi}{4},\tfrac{1}{2^{\numBreakpoints-1}}\frac{\pi}{4}\right]$; then, much like in the first base case, $\etaVar_\numBreakpoints = \tilde\etaVar_\numBreakpoints$ and $\delta_\numBreakpoints = 0$.  If, on the other hand, $\measuredangle{(\xiVar_{\numBreakpoints-1}, \etaVar_{\numBreakpoints-1})} \in \left[0,\tfrac{1}{2^{\numBreakpoints}}\frac{\pi}{4}\right]$; then, much like in the second base case, $\etaVar_\numBreakpoints = -\tilde\etaVar_\numBreakpoints$ and $\delta_\numBreakpoints = 1$.

In either case, we know that $\measuredangle{(\overline\zz, \overline\yy)} \in \big(\tfrac{\pi}{4} - \tfrac{q+1}{2^{\numBreakpoints-1}} \tfrac{\pi}{4}, \tfrac{\pi}{4} -\tfrac{q}{2^{\numBreakpoints-1}} \tfrac{\pi}{4}\big]$ for some nonnegative integer $q$. By the induction hypothesis, we have $\deltabm^{\numBreakpoints-1} = \alphabm^q$; however, by equation \eqref{eq:GrayCodeNewi}, the new digit $\delta_\numBreakpoints$ changes the integer represented by the Gray code to be $q' = 2q + \Vert\deltabm^\numBreakpoints\Vert_0\ \MOD2$. It remains only to demonstrate that $\measuredangle{(\overline\zz, \overline\yy)} \in \big(\tfrac{\pi}{4} - \tfrac{q'+1}{2^{\numBreakpoints}} \tfrac{\pi}{4}, \tfrac{\pi}{4} -\tfrac{q'}{2^{\numBreakpoints}} \tfrac{\pi}{4}\big]$ which is necessarily a proper subset of $\big(\tfrac{\pi}{4} - \tfrac{q+1}{2^{\numBreakpoints-1}} \tfrac{\pi}{4}, \tfrac{\pi}{4} -\tfrac{q}{2^{\numBreakpoints-1}} \tfrac{\pi}{4}\big]$ since $q'\in \{2q,2q+1\}$.

Each rotation step $k < \numBreakpoints$ that requires a mirroring step $\etaVar_k = \tilde\etaVar_k$ will swap the position of the point $(\xiVar_\numBreakpoints,\etaVar_\numBreakpoints)$ is swapped with respect to the line $\measuredangle(\zz,\yy) = \tfrac{1}{2^{\numBreakpoints}}\frac{\pi}{4}$ which maps, by the repeated rotation steps, to the line $\measuredangle(\zz,\yy) = \tfrac{\pi}{4}-\tfrac{2q+1}{2^{\numBreakpoints}}\frac{\pi}{4}$. In particular, if $\Vert\deltabm^\numBreakpoints\Vert_0$ is even, then $\measuredangle(\overline\zz,\overline\yy)  >  \tfrac{\pi}{4}-\tfrac{2q+1}{2^{\numBreakpoints}}\frac{\pi}{4}$ while $\measuredangle(\overline\zz,\overline\yy)  \leq  \tfrac{\pi}{4}-\tfrac{2q+1}{2^{\numBreakpoints}}\frac{\pi}{4}$ if $\Vert\deltabm^\numBreakpoints\Vert_0$ is odd.

Thus the proposition holds, by induction, for any $\numBreakpoints \geq 1$. 
\end{proof}

\section{Computations}
\label{sec:computations}
\label{appendix:computations}

This appendix presents the best known objective values and dual bounds for our test instances across several model types and redistricting objectives. These tables support our computational study in Section 5 by showing how different formulations perform in terms of approximation accuracy, solution quality, and tractability.

In general, models based on logarithmic and stepwise approximations perform better than naive piecewise-linear methods both in terms of runtime and the tightness of dual bounds. As expected, the more breakpoints used, the better the approximation, though this comes at a computational cost. Our experiments confirm that tighter formulations (especially those incorporating domain-specific preprocessing) can improve both accuracy and convergence.

Below, we briefly summarize key observations from each table group.

\subsection*{PWL and VA-PWL}
Standard PWL models, even with high breakpoint counts, tend to produce relatively loose upper bounds and slower convergence. VA-PWL models improve tractability and sometimes yield tighter primal solutions due to the simplification in structure (by assuming equal population). However, they may deviate from realistic feasibility and should be used with care. Gaps remain significant (10–30\%) for large instances.

    \begin{table}[ht!]
      \centering
      \caption{Selected Results for the PWL approximation (2 hrs each)}
        \begin{tabular}{|cccc|ccc|cc|}\hline
        \textbf{State} & $\numDistricts$ & $\numParcels$ & $\numBreakpoints$ &   \textbf{Dual Bnd} & \textbf{MIP Obj.} & \textbf{MIP Gap} & \textbf{Primal Bnd} & \textbf{Gap} \\
        \hline
            AL    & 7     & 67    & 10    & 2.317 & 1.678 & 28\%  & 1.568 & 32\% \\
                &      &     & 60    & 2.259 & 1.572 & 30\%  & 1.571 & 30\% \\
                \hline
            MS    & 4     & 82    & 10    & 1.723 & 1.689 & 2\%   & 1.670 & 3\% \\
                &      &     & 60    & 1.878 & 1.652 & 12\%  & 1.655 & 12\% \\
                \hline
            LA    & 6     & 64    & 10    & 1.958 & 1.611 & 18\%  & 1.587 & 19\% \\
                &      &     & 60    & 2.193 & 1.568 & 28\%  & 1.575 & 28\% \\
                \hline
            SC    & 7     & 46    & 10    & 1.503 & 1.424 & 5\%   & 1.334 & 11\% \\
                &      &     & 60    & 1.724 & 1.309 & 24\%  & 1.313 & 24\% \\\hline
        \end{tabular}%
      \label{tab:ResultsPWL2hr}%
    \end{table}%

    \begin{table}[ht!]
      \centering
      \caption{Selected results for the VA-PWL approximation (2 hours each).}
        \begin{tabular}{|cccc|>{\columncolor{gray!40}}ccc|cc|}\hline
        \textbf{State} & $\numDistricts$ & $\numParcels$ & $\numBreakpoints$ &   \textbf{Dual Bnd} & \textbf{MIP Obj.} & \textbf{MIP Gap} & \textbf{Primal Bnd} & \textbf{Gap} \\
        \hline
        AL    & 7     & 67    & 10    &  1.787 & 1.772 & 1\%   & 1.387 & 22\% \\
                 &      &     & 60    &  1.714 & 1.674 & 2\%   & 1.389 & 19\% \\
                 \hline
        MS    & 4     & 82    & 10    &  1.690 & 1.682 & 0\%   & 1.466 & 13\% \\
                 &      &     & 60    &  1.651 & 1.643 & 0\%   & 1.467 & 11\% \\
                 \hline
        LA    & 6     & 64    & 10    &  1.778 & 1.764 & 1\%   & 1.436 & 19\% \\
                &      &     & 60    &  1.772 & 1.695 & 4\%   & 1.433 & 19\% \\
                \hline
        SC    & 7     & 46    & 10    &  1.411 & 1.411 & 0\% (12 sec) & 1.142 & 19\% \\
                 &      &     & 60    &  1.322 & 1.322 & 0\%(1.4 min) & 1.142 & 14\% \\\hline
        \end{tabular}%
      \label{tab:ResultsPWLApprox}%
    \end{table}%

\subsection*{Stepwise Models}
Step-Max and Step-Exp approximations show substantial improvements in both upper bounds and solution times. As expected, increasing the number of breakpoints improves accuracy and reduces the primal-dual gap. Notably, the Exp variant often performs better in high-breakpoint, long-run settings. For example, in SC, Step-Exp achieves a 3\% gap in under 4 hours.

    \begin{table}[ht!]
      \centering
      \caption{Selected results from the Step-Max approximation (2 hours each.)}
        \begin{tabular}{|cccc|ccc|cc|}\hline
        \textbf{State} & $\numDistricts$ & $\numParcels$ & $\numBreakpoints$ & \textbf{Dual Bnd} & \textbf{MIP Obj.} & \textbf{MIP Gap} & \textbf{Primal Bnd} & \textbf{Gap} \\
        \hline
        AL    & 7     & 67    & 10    & 2.202  & 2.003  & 9.0\%  & 1.550  & 30\% \\
              &       &       & 70    & 1.775  & 1.635  & 7.9\%  & 1.570  & 12\% \\
        \hline
        MS    & 4     & 82    & 10  & 2.005 & 1.906 & 5.0\% & 1.638 & 18\% \\
              &       &       & 90  & 1.728 & 1.674 & 3.2\% & 1.643 & 5\% \\\hline
        LA    & 6     & 64    & 10  & 2.128 & 2.032 & 4.5\% & 1.518 & 29\% \\
              &       &       & 40  & 1.768 & 1.672 & 5.4\% & 1.562 & 12\% \\\hline
        SC    & 7     & 46    & 10  & 1.690 & 1.690 & 0\% (2.1 min) & 1.218 & 28\% \\
              &       &       & 70  & 1.402 & 1.402 & 0\%(51.3 min) & 1.327 & 5\% \\\hline
        \end{tabular}%
      \label{tab:ResultsStepMax}%
    \end{table}%

    \begin{table}[ht!]
      \centering
      \caption{Selected results for Step-Exp. approximation (2 hours each.)}
        \begin{tabular}{|cccc|ccc|cc|}\hline
        \textbf{State} & $\numDistricts$ & $\numParcels$ & $\numBreakpoints$ & \textbf{Dual Bnd} & \textbf{MIP Obj.} & \textbf{Time/MIP Gap} & \textbf{Primal Bnd} & \textbf{Gap} \\
        \hline
        AL    & 7     & 67    & 10    & 2.171 & 2.059 & 5.2\% & 1.559 & 28\% \\
              &       &       & 70    & 1.767 & 1.639 & 7.3\% & 1.574 & 11\% \\
            \hline
        MS    & 4     & 82    & 10    & 2.049 & 1.932 & 5.7\% & 1.641 & 20\% \\
              &       &       & 40    & 1.781 & 1.754 & 1.5\% & 1.664 & 7\% \\
            \hline
        LA    & 6     & 64    & 10    & 2.177 & 2.074 & 4.7\% & 1.518 & 30\% \\
              &       &       & 70    & 1.743 & 1.628 & 6.6\% & 1.562 & 10\% \\
            \hline
        SC    & 7     & 46    & 10    & 1.692 & 1.692 & 56.2 sec & 1.300 & 23\% \\
              &       &       & 70    & 1.377 & 1.377 & 42.2 min & 1.329 & 4\% \\\hline
        \end{tabular}%
      \label{tab:ResultsStepExp}%
    \end{table}%

\subsection*{Stepwise with longer runtime comparisons}
With a 10-hour runtime and 90 breakpoints, both Step-Max and Step-Exp approximations achieve near-optimal solutions in several cases. This shows that with sufficient runtime, these approximations offer highly competitive bounds and robust convergence across states.

    \begin{table}[ht!]
      \centering
      \caption{Results for the Step-Max and Step-Exp. approximations (10 hrs, 90 Breakpoints).}
      \begin{tabular}{|cccc|ccc|cc|}\hline
        \textbf{State} & $\numDistricts$ & $\numParcels$ & $\numBreakpoints$ & \textbf{Dual Bnd} & \textbf{MIP Obj.} & \textbf{Time/MIP Gap} & \textbf{Primal Bnd} & \textbf{Gap} \\
        \hline
        AL    & 7    & 67     & Step-Exp. & 1.764 & 1.624 & 7.9\% & 1.574 & 11\% \\
              &       &       & Step-Max  & 1.783 & 1.628 & 8.7\% & 1.586 & 11\% \\
        \hline
        MS    & 4    &  82    & Step-Exp. & 1.862 & 1.690 & 9.3\% & 1.646 & 12\% \\
              &       &       & Step-Max & 1.728 & 1.674 & 3.2\% & 1.639 & 5\% \\
        \hline
        LA    & 6    &   64   & Step-Exp. & 1.716 & 1.628 & 5.1\% & 1.583 & 8\% \\
              &       &       & Step-Max & 1.722 & 1.638 & 4.9\% & 1.589 & 8\% \\
        \hline
        SC    &   7  &   46   & Step-Exp. & 1.380 & 1.380 & 3.5 hrs & 1.334 & 3\% \\
              &       &       & Step-Max & 1.370 & 1.370 & 2.8 hrs & 1.326 & 3\% \\\hline
        \end{tabular}
      \label{tab:ResultsLongRun}%
    \end{table}%

    \subsection*{Compact BN and LogE Models}
    BN-based models are competitive with relatively few breakpoints (4–7), though their primal gaps remain higher in most cases. LogE models, while slower to converge, deliver consistently strong dual bounds, especially for 90+ breakpoints. Their strength is most evident in MS and LA, where the dual gaps fall below 6\% with moderate runtime.
    
    \begin{table}[ht!]
      \centering
      \caption{Results for the BN-PWL approximation (2 hours each).}
      
        \begin{tabular}{|cccc|ccc|cc|}\hline
        \textbf{State} & $\numDistricts$ & $\numParcels$ & $\numBreakpoints$ & \textbf{Dual Bnd} & \textbf{MIP Obj.} & \textbf{Time/MIP Gap} & \textbf{Primal Bnd} & \textbf{Gap} \\
        \hline
        AL    & 7     & 67    & 4     & 2.064 & 1.938 & 6.1\% & 1.486 & 28\% \\
              &       &       & 6     & 1.812 & 1.672 & 7.7\% & 1.576 & 13\% \\
              &       &       & 7     & 1.818 & 1.618 & 11.0\% & 1.582 & 13\% \\
        \hline
        MS    & 4     & 82    & 4     & 1.962 & 1.851 & 5.7\% & 1.609 & 18\% \\
              &       &       & 6     & 1.752 & 1.674 & 4.5\% & 1.603 & 8\% \\
              &       &       & 7     & 1.753 & 1.622 & 7.5\% & 1.584 & 10\% \\
        \hline
        LA    & 6     & 64    & 4     & 2.080 & 2.051 & 1.4\% & 1.568 & 25\% \\
              &       &       & 6     & 1.785 & 1.691 & 5.2\% & 1.588 & 11\% \\
              &       &       & 7     & 1.834 & 1.626 & 11.3\% & 1.570 & 14\% \\
        \hline
        SC    & 7     & 46    & 4     & 1.692 & 1.692 & 45 sec & 1.310 & 23\% \\
              &       &       & 5     & 1.498 & 1.498 & 2.2 min & 1.302 & 13\% \\
              &       &       & 6     & 1.405 & 1.405 & 17 min & 1.329 & 5\% \\
              &       &       & 7     & 1.364 & 1.364 & 55 min & 1.334 & 2\% \\\hline
        \end{tabular}%
      \label{tab:BNFull}%
    \end{table}%

    \begin{table}[ht!]
      \centering
      \caption{Results for the LogE approximation of black representation (Most are 2 hours, an * indicates a 15-hour instance).}
        \begin{tabular}{|cccc|ccc|cc|}\hline
        \textbf{State} & $\numDistricts$ & $\numParcels$ & $\numBreakpoints$ & \textbf{Dual Bnd} & \textbf{MIP Obj.} & \textbf{Time/MIP Gap} & \textbf{Primal Bnd} & \textbf{Gap} \\
        \hline
        AL    & 7     & 67    & 25    & 1.872 & 1.744 & 6.8\% & 1.576 & 16\% \\
              &       &       & 90    & 1.797 & 1.615 & 10.1\% & 1.580 & 12\% \\
              &       &       & 90*   & 1.764 & 1.602 & 9.2\% & 1.582 & 10\% \\
              &       &       & 200   & 1.830 & 1.531 & 16.4\% & 1.510 & 17\% \\
        \hline
        MS    & 4     & 82    & 25    & 1.835 & 1.760 & 4.1\% & 1.585 & 14\% \\
              &       &       & 90    & 1.726 & 1.675 & 2.9\% & 1.642 & 5\% \\
              &       &       & 90*   & 1.728 & 1.629 & 5.7\% & 1.605 & 7\% \\
              &       &       & 200   & 1.708 & 1.617 & 5.3\% & 1.599 & 6\% \\
        \hline
        LA    & 6     & 64    & 25    & 1.800 & 1.754 & 2.6\% & 1.560 & 13\% \\
              &       &       & 90    & 1.736 & 1.610 & 7.3\% & 1.556 & 10\% \\
              &       &       & 90*   & 1.707 & 1.639 & 4.0\% & 1.605 & 6\% \\
              &       &       & 200   & 1.761 & 1.566 & 11.1\% & 1.545 & 12\% \\
        \hline
        SC    & 7     & 46    & 25    & 1.454 & 1.454 & 45 sec & 1.290 & 11\% \\
              &       &       & 90    & 1.366 & 1.366 & 2.2 min & 1.328 & 3\% \\
              &       &       & 200   & 1.479 & 1.233 & 17 min & 1.216 & 18\% \\\hline
        \end{tabular}%
      \label{tab:LogE}%
    \end{table}%

\subsection{HVAP + BVAP Objective}
When optimizing for combined minority representation, primal bounds are noticeably weaker, especially at low breakpoint counts. This is expected given the increased variance in the representation function and the larger number of small-population nodes that confound optimization. Still, Step-Exp models with 60 breakpoints reduce the primal-dual gap significantly compared to 25-point runs.

    \begin{table}[ht!]
    \centering
    \caption{Results for the Step-Exp. approximation of the HVAP+BVAP objective (2 hours each).}
    \begin{tabular}{|cccc|rrc|rc|}\hline
    \textbf{State} & $\numDistricts$ & $\numParcels$ & $\numBreakpoints$ & \multicolumn{1}{c}{\textbf{Dual Bnd}} & \multicolumn{1}{c}{\textbf{MIP Obj.}} & \textbf{Time/MIP Gap} & \multicolumn{1}{c}{\textbf{Primal Bnd}} & \textbf{Gap} \\
    \hline
    AL    & 7     & 67    & 25    & 1.565 & 1.276 & 18.5\% & 0.891 & 43\% \\
          &       &       & 60    & 1.360 & 1.242 & 8.7\% & 1.067 & 22\% \\
    \hline
    MS    & 4     & 82    & 25    & 1.722 & 1.487 & 13.6\% & 1.128 & 35\% \\
          &       &       & 60    & 1.635 & 1.401 & 14.3\% & 1.250 & 24\% \\
    \hline
    LA    & 6     & 64    & 25    & 1.550 & 1.456 & 6.0\% & 0.957 & 38\% \\
          &       &       & 60    & 1.308 & 1.186 & 9.3\% & 0.984 & 25\% \\
    \hline
    SC    & 7     & 46    & 25    & 1.114 & 1.114 & 1.1 min & 0.763 & 32\% \\
          &       &       & 60    & 1.004 & 1.004 & 1.4 min & 0.880 & 12\% \\\hline
    \end{tabular}%
    \label{tab:ResultsHVAP}%
\end{table}%

\section{Variable ranges}
\label{sec:variable_ranges}
Variable ranges were computed for our four main instances.  These helped improve the LogE formulation.  We present those bounds in Figure~\ref{fig:bvap_ranges}.
\newlength{\CropAmountLeft}
\setlength{\CropAmountLeft}{55pt}
\newlength{\CropAmount}
\setlength{\CropAmount}{65pt}
\newcommand{\ImageScale}{0.39}

\begin{figure}[ht!]
    \centering{
    \includegraphics[scale=\ImageScale,trim={\CropAmountLeft} 0 {\CropAmount} 0,clip]{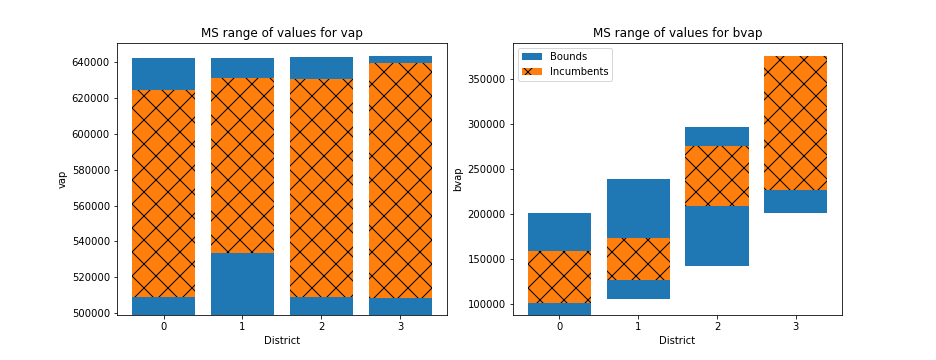}
    \includegraphics[scale=\ImageScale,trim={\CropAmountLeft} 0 {\CropAmount} 0,clip]{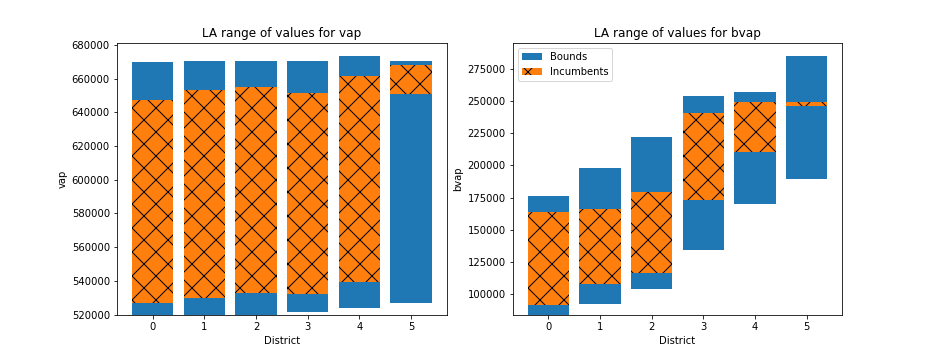}\\
    \includegraphics[scale=\ImageScale,trim={\CropAmountLeft} 0 {\CropAmount} 0,clip]{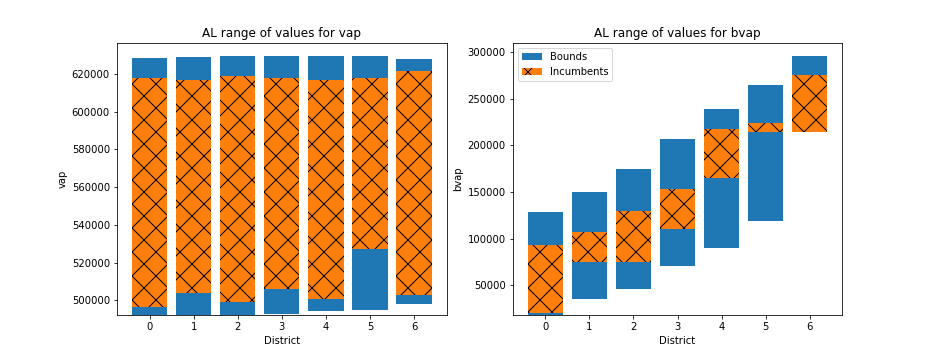}
    \includegraphics[scale=\ImageScale,trim={\CropAmountLeft} 0 {\CropAmount} 0,clip]{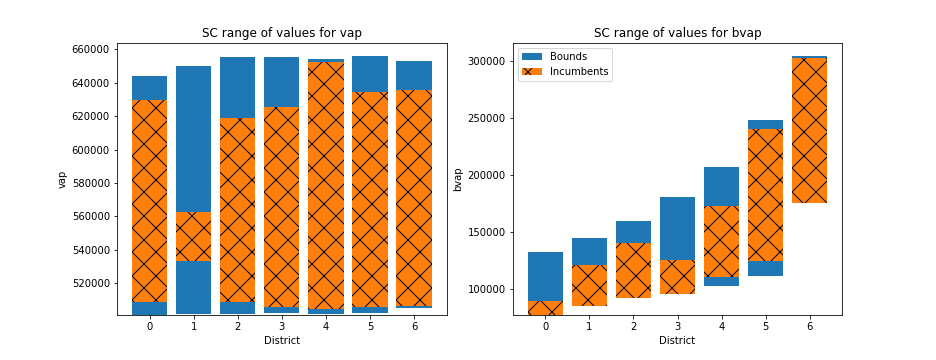}
    }
    \caption{Ranges computed for BVAP and vap given an ordering on the districts with respect to BVAP. These computations were done with a limit of 10 minutes; hence, for many of the computations, optimality was not reached.}
    \label{fig:bvap_ranges}
\end{figure}

\pagebreak
\section{Maps of best known primal solutions}\label{app:Maps}
Here we display several of our best solutions as maps.
    \begin{figure}[ht!]
    \centering
        \begin{subfigure}{.48\textwidth}
        \includegraphics[width=\textwidth,trim={5cm 5cm 5cm 5cm},clip]{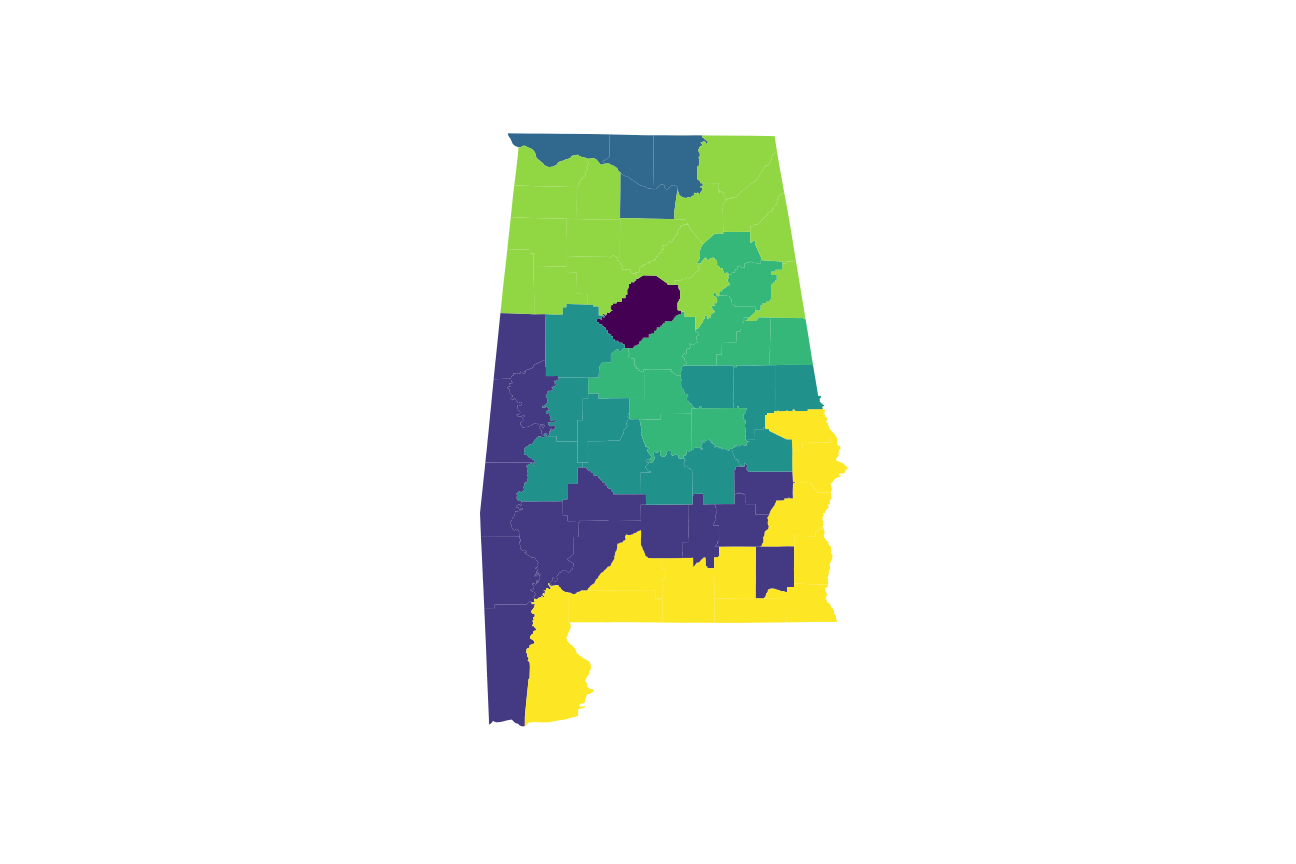}
        \caption{Alabama under Step-Max with $\ell=90$ for ten hours achieved a primal bound of 1.586.}
        \end{subfigure}~
        \begin{subfigure}{.48\textwidth}
        \includegraphics[width=\textwidth,trim={5cm 5cm 5cm 5cm},clip]{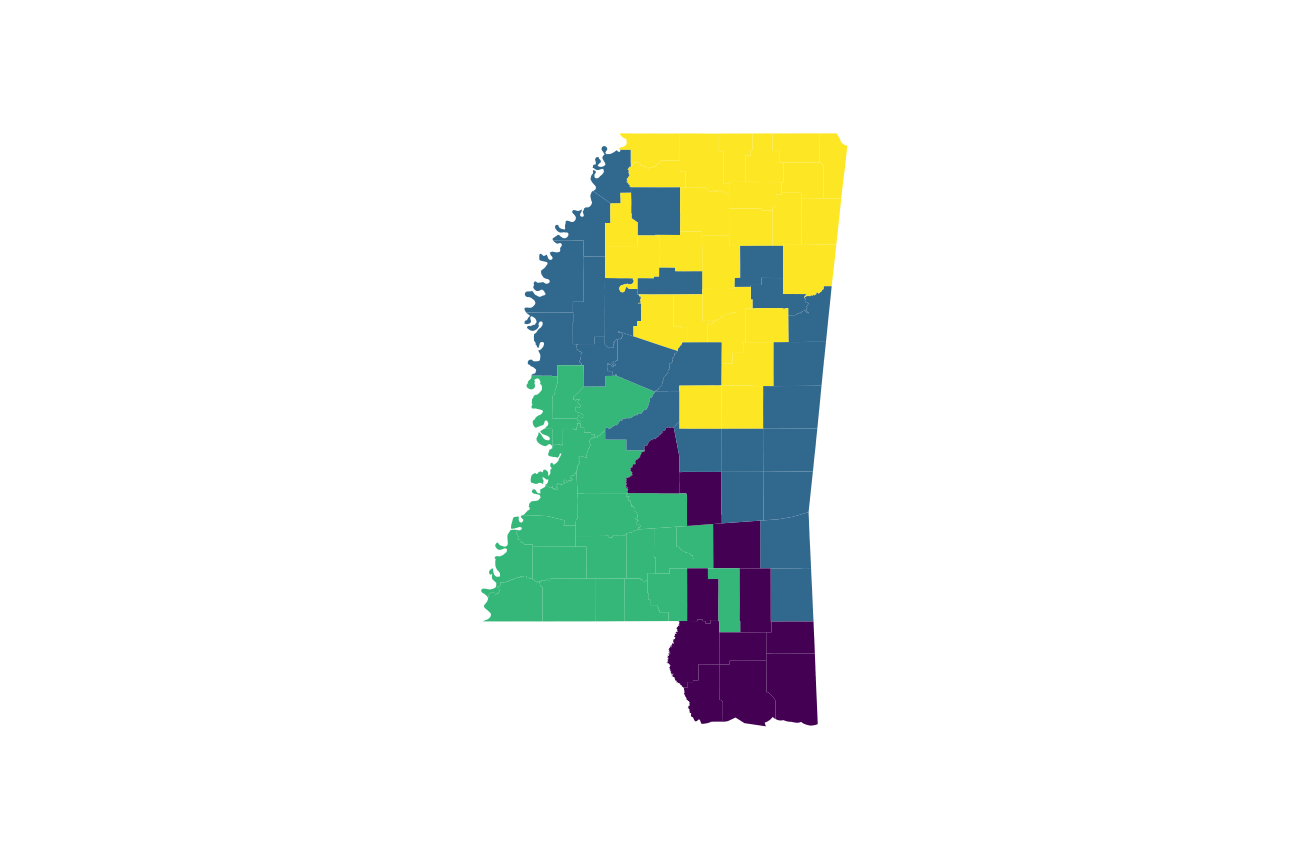}
        \caption{Mississippi under Step-Exp with $\ell=40$ for 2 hours achieved a primal bound of 1.664.}
        \end{subfigure}\\
        \begin{subfigure}{.48\textwidth}
        \includegraphics[width=\textwidth,trim={4cm 4cm 4cm 4cm},clip]{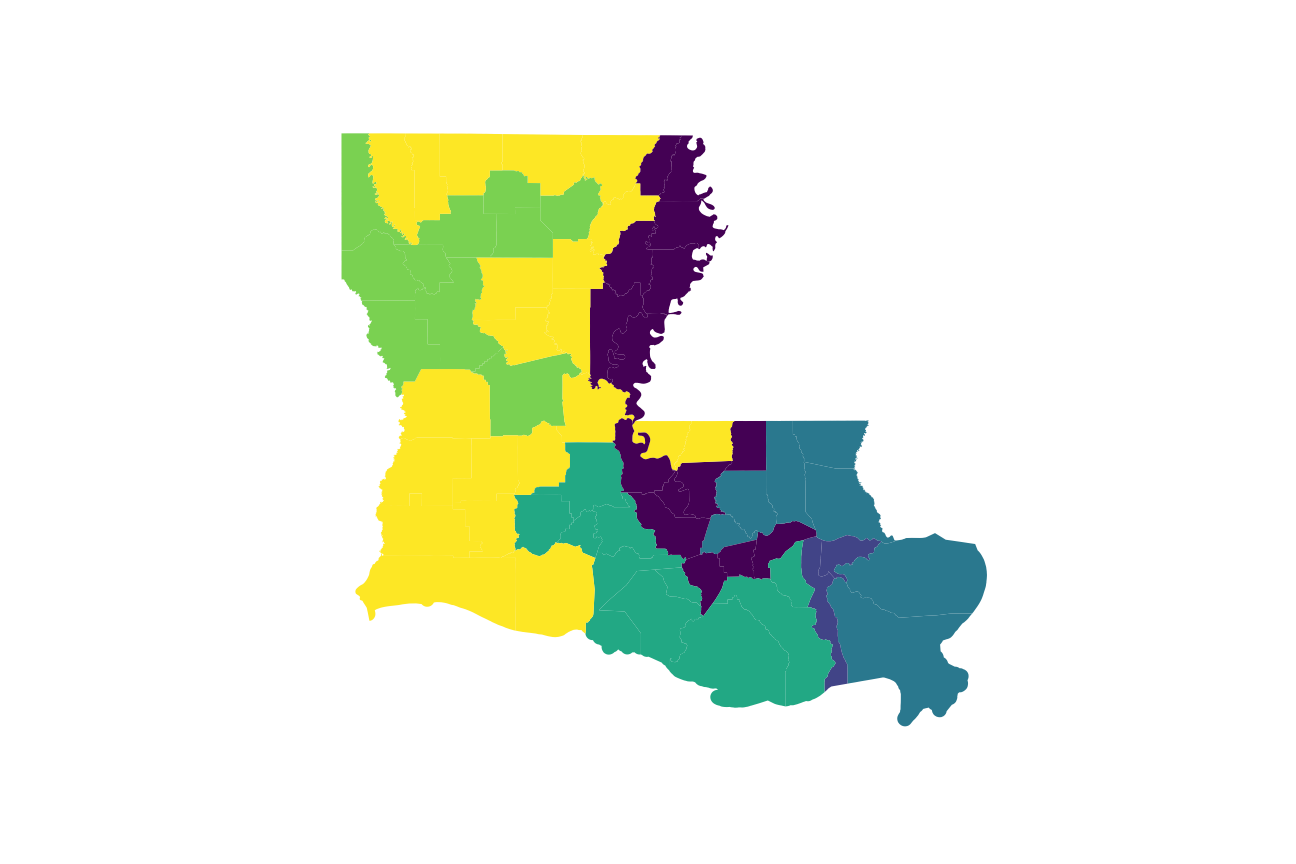}
        \caption{Louisiana under Step-Max with $\ell=90$ for ten hours achieved a primal bound of 1.589}
        \end{subfigure}~
        \begin{subfigure}{.48\textwidth}
        \includegraphics[width=\textwidth,trim={4cm 4cm 4cm 4cm},clip]{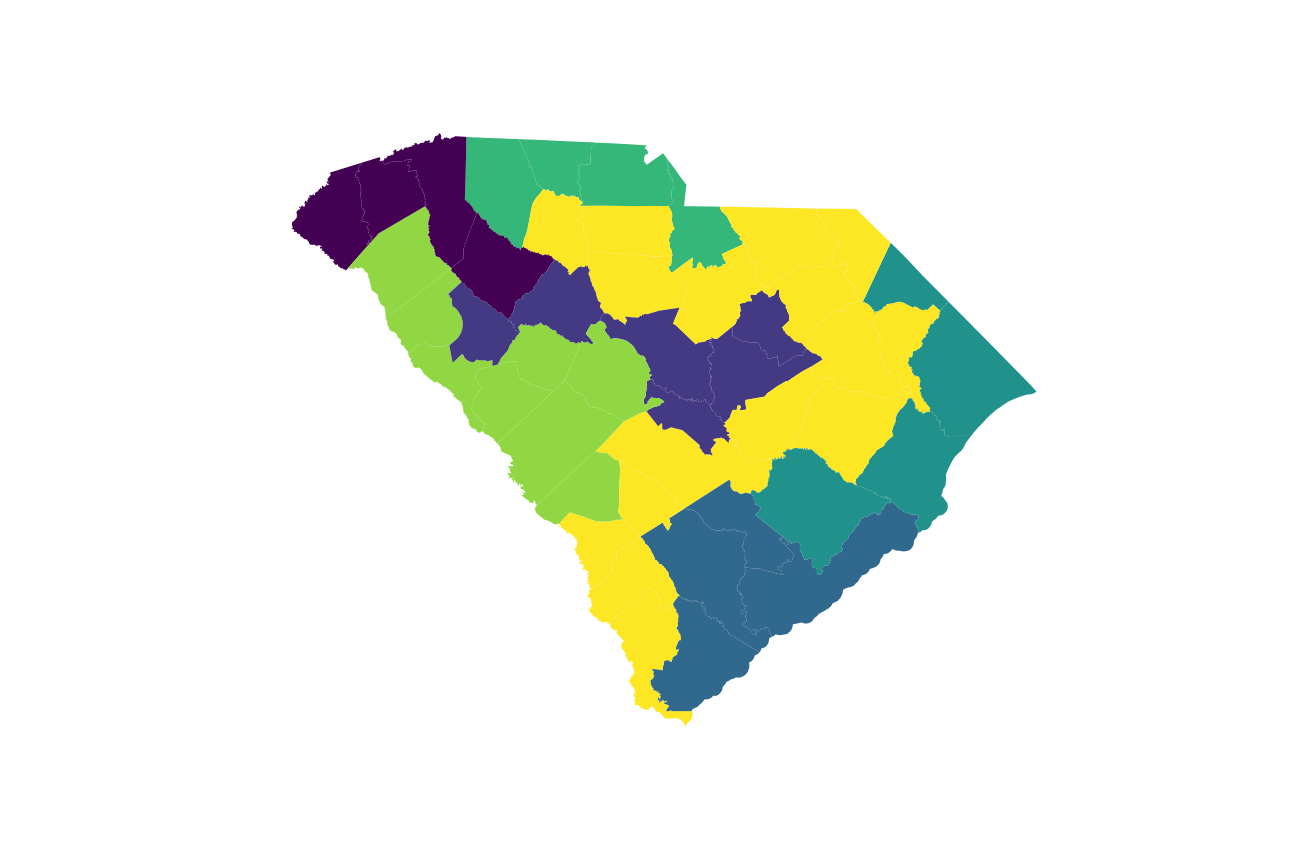}
        \caption{South Carolina under BN-Full with $\ell=7$ for two hours achieved a primal bound of 1.334.}
        \end{subfigure}~
    \caption{Maps of the best found primal BR-Redistricting solutions given in Table \ref{tab:BestInstances}.}
    \label{fig:BestInstances}
    \end{figure}

    \begin{figure}
    \centering
        \begin{subfigure}{.48\textwidth}
        \includegraphics[width=\textwidth,trim={6cm 6cm 6cm 6cm},clip]{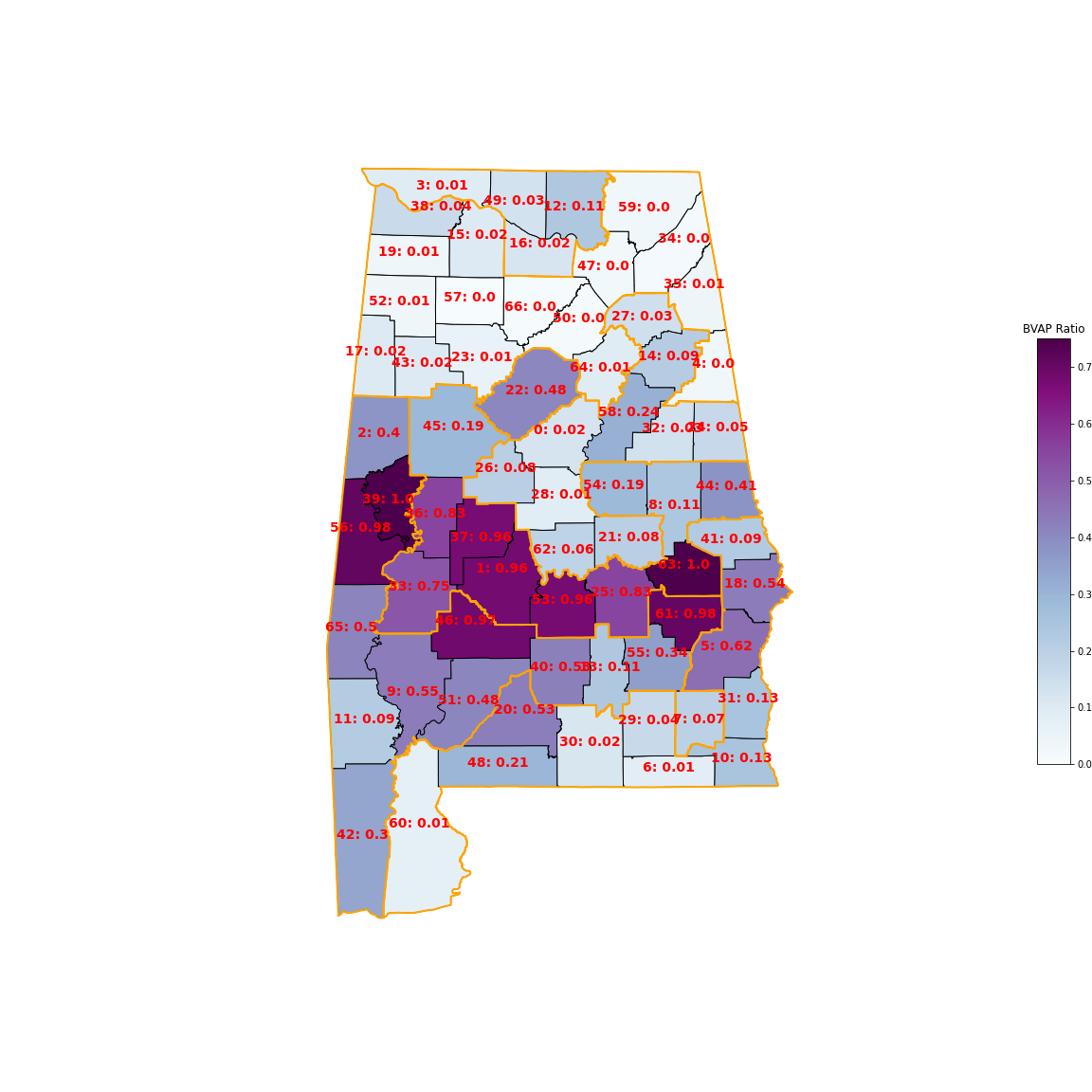}
        \caption{Alabama}
        \end{subfigure}~
        \begin{subfigure}{.48\textwidth}
        \includegraphics[width=\textwidth,trim={6cm 6cm 6cm 6cm},clip]{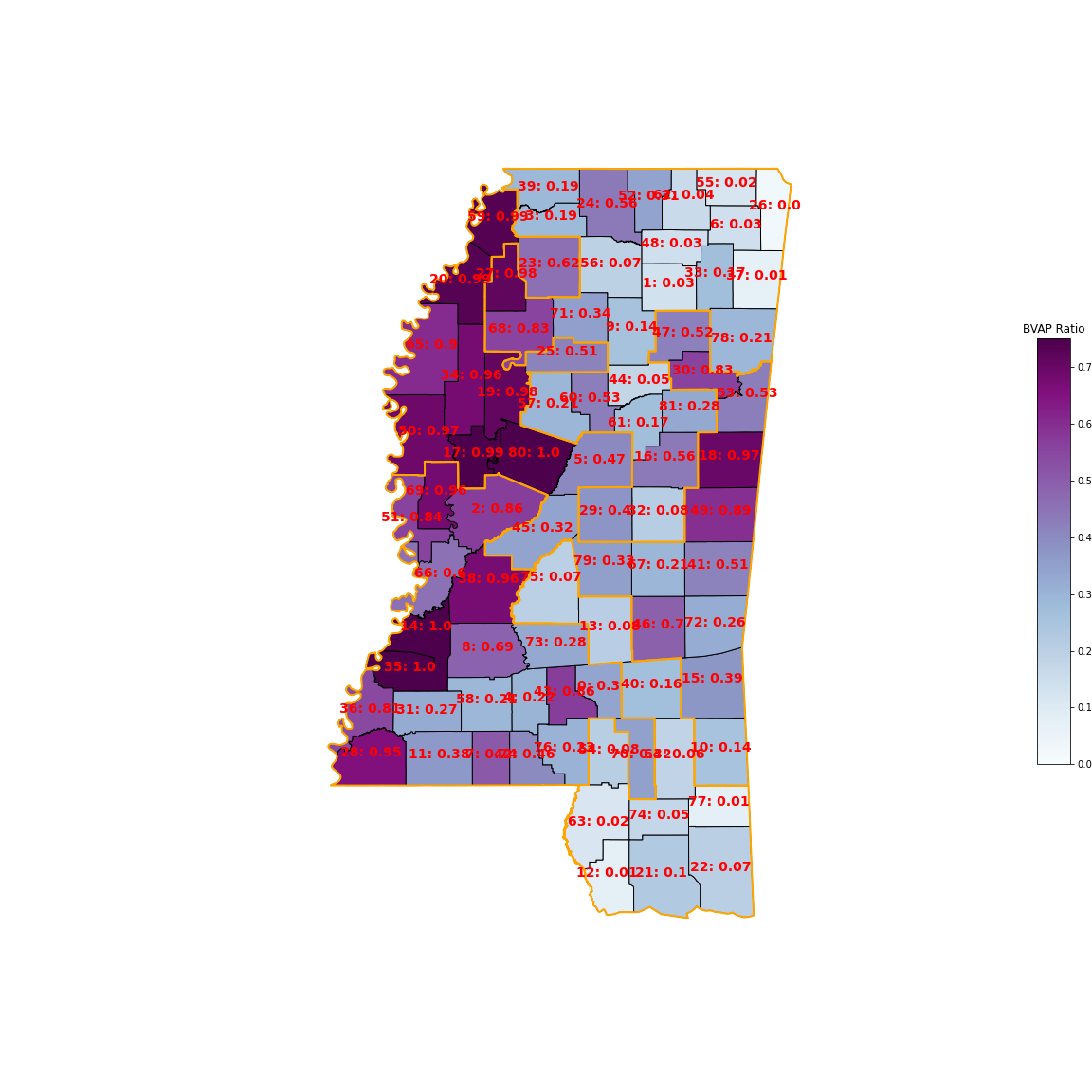}
        \caption{Mississippi}
        \end{subfigure}\\
        \begin{subfigure}{.48\textwidth}
        \includegraphics[width=\textwidth,trim={4cm 4cm 4cm 4cm},clip]{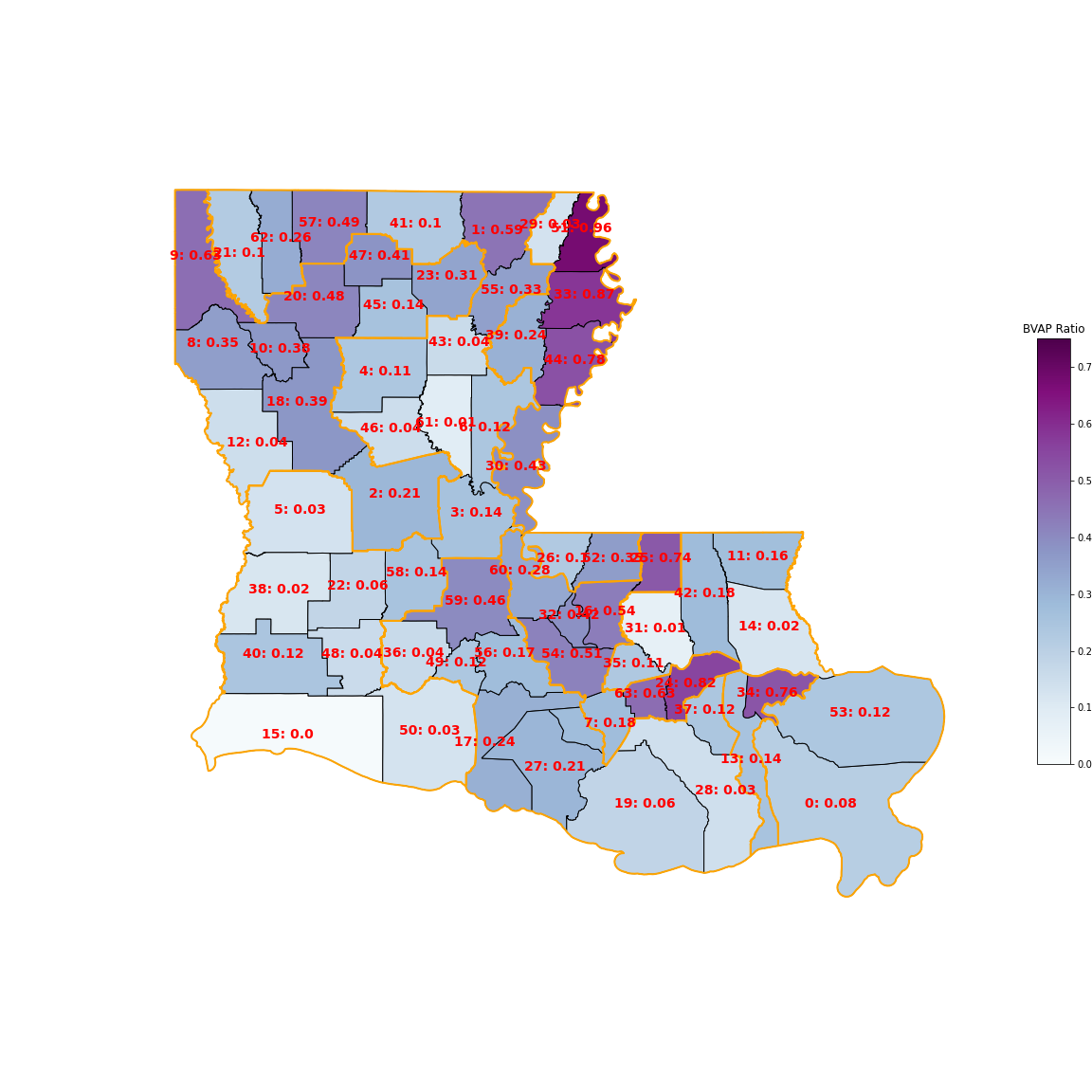}
        \caption{Louisiana}
        \end{subfigure}~
        \begin{subfigure}{.48\textwidth}
        \includegraphics[width=\textwidth,trim={1cm 1cm 4cm 1cm},clip]{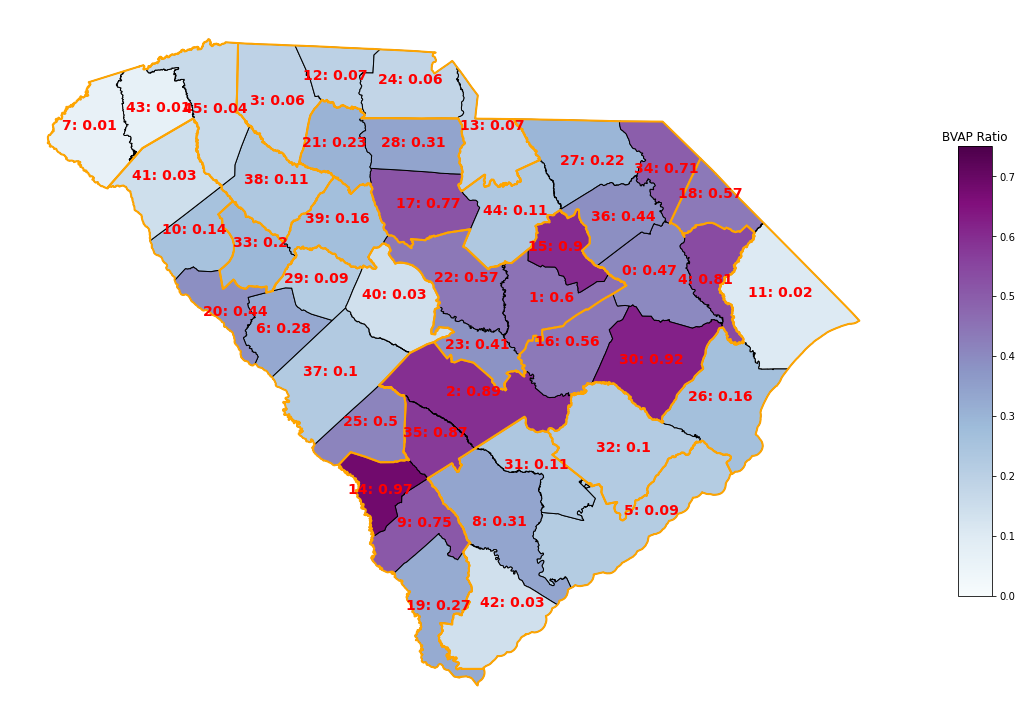}
        \caption{South Carolina}
        \end{subfigure}~
    \caption{County level BVAP/VAP heatmaps. Overlayed in orange are the districts from Figure \ref{fig:BestInstances}.}
    \label{fig:SCCountyVis}
    \end{figure}

    \begin{figure}
    \centering
        \begin{subfigure}{.48\textwidth}
        \includegraphics[width=\textwidth,trim={4cm 4cm 4cm 4cm},clip]{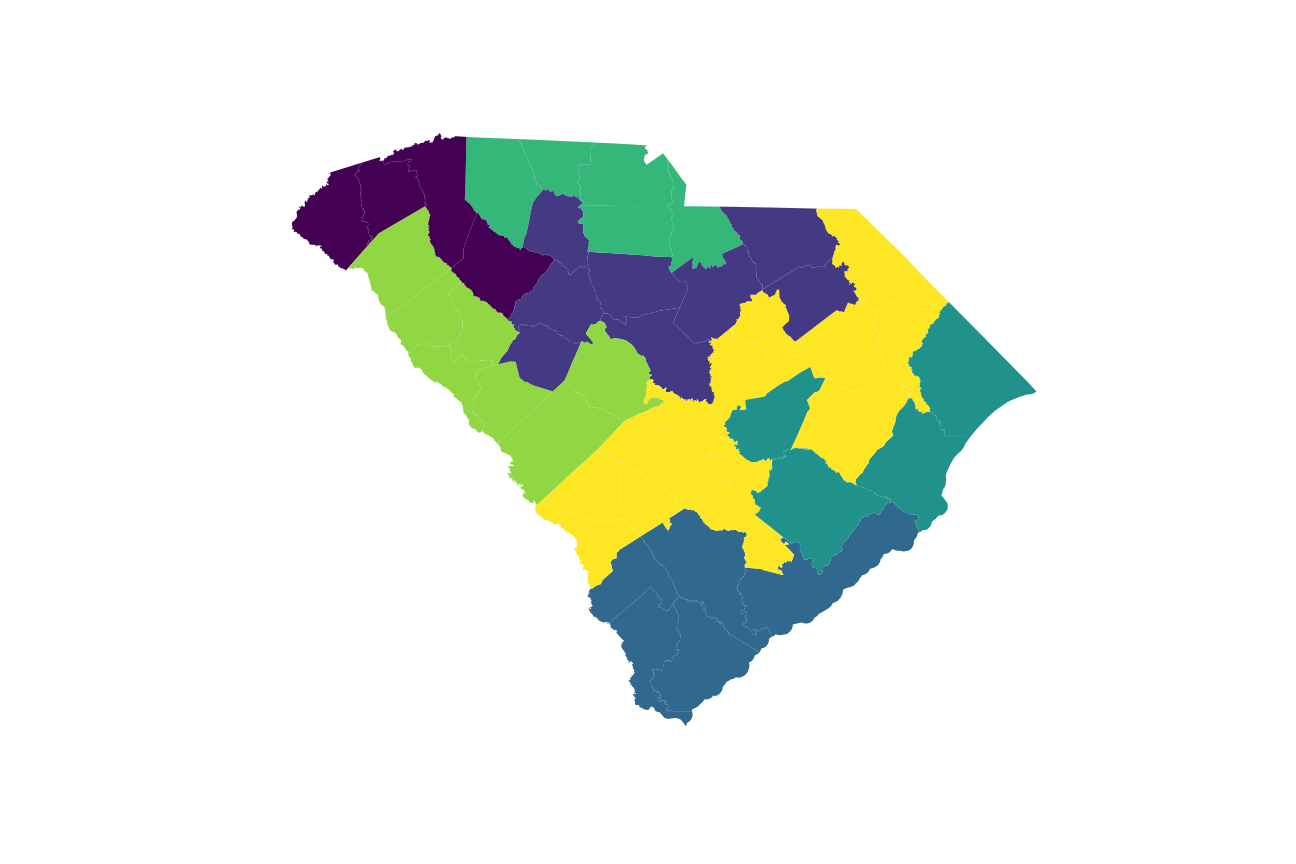}
        \caption{With a primal bound of 2.055, the best found CPVI solution for South Carolina looks very similar to its best Black Representatives solution.\footnotemark}
        \end{subfigure}~
        \begin{subfigure}{.48\textwidth}
        \includegraphics[width=\textwidth,trim={2cm 2cm 2cm 2cm},clip]{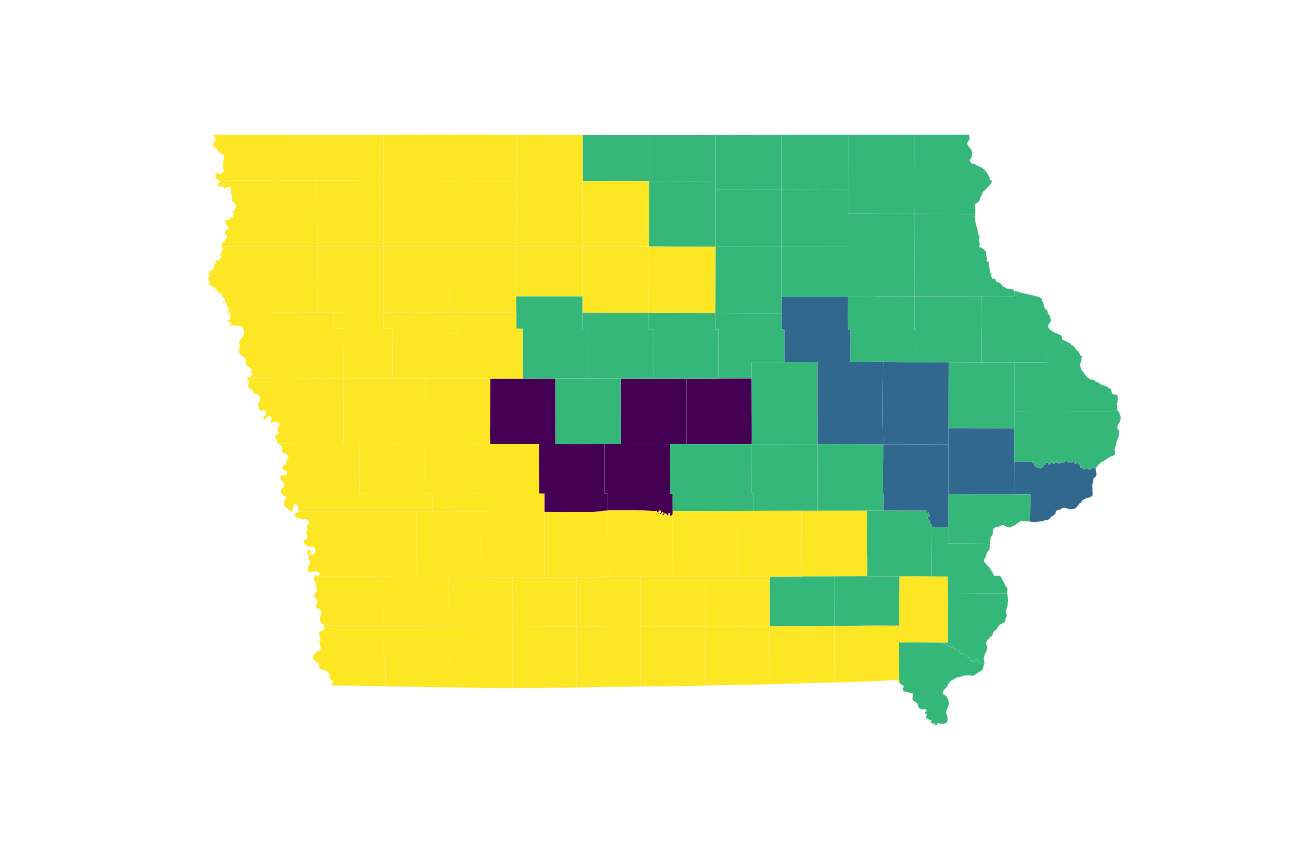}
        \caption{For Iowa, which is districted at the county level, this map is potentially an optimal Democrat gerrymander. It has a primal bound of 1.624 and a dual bound of 1.678 expected Democrat representatives}
        \end{subfigure}
    \caption{The best found CPVI-Redistricting solutions for South Carolina and Iowa as described in Table \ref{tab:CPVIResultsLongRuns}.}
    \label{fig:BestCPVI}
    \end{figure}
    \footnotetext{This map bears visual resemblance to the configuration at issue in \textit{Alexander v. South Carolina State Conference of the NAACP} (2024), a recent U.S. Supreme Court case concerning racial gerrymandering. We emphasize that our map is not intended to reflect legal standards or compliance, but rather to illustrate the outcome of an optimization model under a specific objective.}

\pagebreak
\section{Error derivations}
\label{app:error}
We provide a simple calculus result to discuss the error in linearization. 
    \begin{theorem}[{\cite[Theorem 3.3]{Burden2015-rt}}]
    Suppose \( x_0, x_1, \dots, x_n \) are distinct numbers in the interval \([a, b]\) and \( f \in C^{n+1}[a, b] \). Then, for each \(x\) in \([a, b]\), a number \( \xi(x) \) (generally unknown) between \( x_0, x_1, \dots, x_n \), and hence in \( (a, b) \), exists such that:
    \[
    f(x) = P(x) + \frac{f^{(n+1)}(\xi(x))}{(n+1)!} (x - x_0)(x - x_1) \cdots (x - x_n),
    \]
    where \( P(x) \) is the interpolating polynomial given satisfying $P(x_i) = f(x_i)$ for $i=0,1, \dots, n$.
    \end{theorem}
We adapt this to a linear approximation.

\begin{corollary}[Error bound for linear interpolation]
\label{prop:taylor}
Let \( f(x) \) be a function defined on \([a, b]\) with \( f \in C^2[a, b] \). Then the difference between \( f(x) \) and the line connecting the points \( (a, f(a)) \) and \( (b, f(b)) \) is bounded by:
\[
|f(x) - \ell(x)| \leq \frac{\|f''\|_\infty}{8} (b - a)^2
\]
where \( \ell(x) = \frac{f(b) - f(a)}{b - a} (x - a) + f(a) \) is the equation of the line connecting \( (a, f(a)) \) and \( (b, f(b)) \), and \( \|f''\|_\infty = \max_{\xi \in [a, b]} |f''(\xi)| \).
\end{corollary}

We next apply this computation to get bounds on the error of any linearization with respect to our main objective function.

\begin{remark}
Given the function \( \brObj(r) = \frac{1}{\sqrt{2\pi}} \int_{-\infty}^{\beta \cdot r - \beta_0} e^{-\frac{1}{2}t^2} \, dt \), representing the CDF of a normal distribution with a shifted upper limit, the first and second derivative with respect to \( r \) is:

\[
\brObj'(r) = \frac{\beta}{\sqrt{2\pi}} e^{-\frac{1}{2}(\beta \cdot r - \beta_0)^2},
\ \ \ 
\brObj''(r) = -\frac{\beta^2 (\beta \cdot r - \beta_0)}{\sqrt{2\pi}} e^{-\frac{1}{2}(\beta \cdot r - \beta_0)^2}.
\]

With $\beta = 6.826, \beta_0 = 2.827$, we compute the error, according to Proposition~\ref{prop:taylor},  of a piecewise linear approximation with   $n$ uniformly distributed breakpoints.
$$
\begin{array}{cllll}
 \text{n} & \text{Single Error} & \times 4 & \times 6 & \times 7 \\
 \hline
 10 & 0.0140931 & 0.0563725 & 0.0845588 & 0.0986519 \\
 25 & 0.0022549 & 0.0090196 & 0.0135294 & 0.0157843 \\
 60 & 0.000391476 & 0.0015659 & 0.00234885 & 0.00274033 \\
 90 & 0.000173989 & 0.000695957 & 0.00104394 & 0.00121792 \\
 120 & 0.0000978689 & 0.000391476 & 0.000587214 & 0.000685082 \\
 200 & 0.0000352328 & 0.000140931 & 0.000211397 & 0.00024663 \\
\end{array}
$$
\end{remark}

Given that each state we worked with has at most 7 districts, using 60 breakpoints ensures our bounds remain accurate within about 0.003, while using 90 or more breakpoints maintains accuracy to 0.0013. However, these error bounds are quite conservative and could be refined with more thorough analysis. We believe that the actual bounds provided by our LogE and BN-Full methods are likely much tighter than this simple analysis indicates.

\section{LogE Analysis and Comparison to PWL}
\label{appendix:proof-LogE}

We will prove Theorem~\ref{tech:LogE:Conjecture}.  But along the way, we will study the properties of the logE formulation.

We begin with a standard theorem about barycentric coordinates. Here, for any vectors \(\bm{a}, \bm{b} \in \mathbb{R}^2\), the symbol \(\bm{a} \times \bm{b}\) denotes the scalar 2D cross product:
\(
\bm{a} \times \bm{b} := a_1 b_2 - a_2 b_1.
\)

\begin{theorem}[Barycentric Coordinates via Cross Products]
\label{thm:barycentric_cross}
Let \(\bm{v}^1, \bm{v}^2, \bm{v}^3 \in \mathbb{R}^2\) be three non-collinear points forming a triangle, and let \(\bm{x} \in \mathbb{R}^2\) be any point in the plane. Then there exist unique scalars \(\lambda_1, \lambda_2, \lambda_3 \in \mathbb{R}\) such that
\[
\bm{x} = \lambda_1 \bm{v}^1 + \lambda_2 \bm{v}^2 + \lambda_3 \bm{v}^3
\quad \text{and} \quad
\lambda_1 + \lambda_2 + \lambda_3 = 1.
\]
These scalars are the barycentric coordinates of \(\bm{x}\) with respect to the triangle \((\bm{v}^1, \bm{v}^2, \bm{v}^3)\), and are given by:
\[
\lambda_1 = \frac{(\bm{v}^2 - \bm{v}^3) \times (\bm{x} - \bm{v}^3)}{(\bm{v}^2 - \bm{v}^3) \times (\bm{v}^1 - \bm{v}^3)},
\quad
\lambda_2 = \frac{(\bm{v}^3 - \bm{v}^1) \times (\bm{x} - \bm{v}^3)}{(\bm{v}^2 - \bm{v}^3) \times (\bm{v}^1 - \bm{v}^3)},
\quad
\lambda_3 = 1 - \lambda_1 - \lambda_2.
\]

\end{theorem}

We are particularly interested in the case where $v^2$ is a scaled version of $v^1$.
\begin{figure}[h!]
\begin{center}
\begin{tikzpicture}[scale=2]
    \draw[->] (-0.2,0) -- (2.2,0) node[right] {$x_1$};
    \draw[->] (0,-0.2) -- (0,2.2) node[above] {$x_2$};
    
    \coordinate (v1) at (1,0.5);
    \coordinate (v2) at (2,1);    %
    \coordinate (v3) at (0.5,2);
    
    \draw[blue,thick] (0,0) -- (v1);
    \draw[blue,thick] (0,0) -- (v2);
    \draw[blue,thick] (0,0) -- (v3);
    
    \fill[blue] (v1) circle (1.5pt);
    \node[below left] at (v1) {$\bm{v}^1$};
    \fill[blue] (v2) circle (1.5pt);
    \node[below left] at (v2) {$\bm{v}^2$};
    \fill[blue] (v3) circle (1.5pt);
    \node[below left] at (v3) {$\bm{u}$};
    
    \draw[dashed,red,thick] (v1) -- (v2) -- (v3) -- cycle;
    
    \coordinate (x) at (1.1667,1.1667);
    
    \fill[green] (x) circle (1.5pt);
    \node[below left] at (x) {$x$};
\end{tikzpicture}
\end{center}
\caption{The point $\bm{x}$  can be expressed in Barycentric Coordinates via $\bm{v}^1, \bm{v}^2, \bm{u}$.}
\end{figure}

\begin{corollary}[Barycentric Coordinates When \(\bm{v}^2 = \alpha \bm{v}^1\)]
\label{cor:degenerate_v2}
Let \(\bm{v}^1, \bm{u} \in \mathbb{R}^2\) be fixed and linearly independent, and suppose \(\bm{v}^2 = \alpha \bm{v}^1\) for some \(\alpha > 0\). Then for any point \(\bm{x} \in \mathbb{R}^2\), the barycentric coordinates \((\lambda_1, \lambda_2, \mu)\) with respect to the triangle \((\bm{v}^1, \bm{v}^2, \bm{u})\) are:
\begin{equation}
\label{eq:lambda3}
\lambda_1 = \frac{1 - \mu}{1 + \alpha}, \quad
\lambda_2 = \frac{\alpha (1 - \mu)}{1 + \alpha}, \quad \text{ and }\quad 
\mu = \frac{\bm{v}^1 \times \bm{x}}{\bm{v}^1 \times \bm{u}}.
\end{equation}

These satisfy
\[
\bm{x} = \lambda_1 \bm{v}^1 + \lambda_2 \bm{v}^2 + \mu \bm{u}, \quad \lambda_1 + \lambda_2 + \mu  = 1.
\]

\noindent In particular, \(\mu\) is independent of \(\alpha\), depending only on the cross products involving \(\bm{v}^1, \bm{v}^3\), and \(\bm{x}\).
\end{corollary}

\begin{remark}
\label{remark:scaling}
The coordinate \(\mu\) is independent of both the scalar \(\alpha > 0\) and the norm \(\|\bm{v}^1\|_2\). It depends only on the directions of \(\bm{v}^1\) and \(\bm{u}\), and the position of the point \(\bm{x}\) relative to them. This makes \(\mu\) invariant under positive rescaling of \(\bm{v}^1\).  However, $\mu$ is inversely proportional to $\|\bm{u}\|_2$.   
\end{remark}

\subsection{Linear Program with Four Points}

We now present a linear program whose feasible set is defined by four points 
\[
\bm{v}^1, \quad \bm{v}^2 = \alpha\,\bm{v}^1, \quad \bm{v}^3, \quad \bm{v}^4 = \beta\,\bm{v}^3 \quad \text{for } \alpha,\beta > 1.
\]
Observe that \(\bm{v}^2\) lies on the ray from the origin through \(\bm{v}^1\), and \(\bm{v}^4\) lies on the ray from the origin through \(\bm{v}^3\). Hence the convex hull 
\(\mathrm{conv}(\bm{v}^1, \bm{v}^2, \bm{v}^3, \bm{v}^4)\) remains a quadrilateral in general (unless some further degeneracy holds).

\begin{proposition}
\label{prop:LP-4points-basic-solution}
Consider the linear program
\[
z(\bm{x}) 
\;:=\;
\max_{\{\lambda_i\}}\Bigl\{
  f^- (\lambda_1 + \lambda_2) 
  + f^+(\lambda_3 + \lambda_4)
  \;\Bigm|\;
  \lambda_i \ge 0, \;
  \sum_{i=1}^4 \lambda_i = 1,\;
  \bm{x} \;=\; \sum_{i=1}^4 \lambda_i \bm{v}^i
\Bigr\},
\]
where \(\bm{v}^2 = \alpha \bm{v}^1\) and \(\bm{v}^4 = \beta \bm{v}^3\) for some \(\alpha,\beta > 1\) and $\bm{x} \in \conv(\bm{v}^1, \bm{v}^2,\bm{v}^3, \bm{v}^4)$. Then there exists an optimal solution such that at least one of the \(\lambda_i\) is zero.
\end{proposition}

\begin{proof}[Sketch of Proof]
This follows from the standard theory of linear programs: basic feasible solutions in a polytope occur at extreme points or edges formed by the constraints \(\lambda_i \ge 0\) and \(\sum \lambda_i = 1\).  Since \(\bm{v}^2\) is a scalar multiple of \(\bm{v}^1\) and \(\bm{v}^4\) is a scalar multiple of \(\bm{v}^3\), the feasible set's dimension is effectively smaller than four; one can show that an optimal solution cannot require all four directions simultaneously.  Thus, at least one \(\lambda_i\) must be zero.
\end{proof}

We next specialize to the case where \(\bm{v}^1\) and \(\bm{v}^3\) may also lie in a particular orientation with respect to \(\bm{x}\).  Let
\(
\bm{x} \;\in\; \mathrm{conv}(\bm{v}^1, \bm{v}^2, \bm{v}^3, \bm{v}^4),
\ \text{and assume}\  
\bm{v}^1 \times \bm{v}^3 \;\le\; 0,
\)
which typically means that \(\bm{v}^1\) and \(\bm{v}^3\) span a cone in which \(\bm{x}\) resides, but with a certain orientation (e.g.\ one vector might lie ``above” the other with respect to the origin).  If we suppose that $f^+ > f^-$, then we can decide which three vertices will be used in an optimal solution.

\begin{figure}
\begin{center}
\begin{tikzpicture}[scale=2]
    \draw[->] (-0.2,0) -- (2.2,0) node[right] {$x_1$};
    \draw[->] (0,-0.2) -- (0,2.2) node[above] {$x_2$};
    
    \coordinate (v1) at (1,0.5);
    \coordinate (v2) at (2,1);    %
    \coordinate (v3) at (0.25,1);
    \coordinate (v4) at (0.5,2);
    
    \draw[blue,thick] (0,0) -- (v1);
    \draw[blue,thick] (0,0) -- (v2);
    \draw[blue,thick] (0,0) -- (v3);
    \draw[blue,thick] (0,0) -- (v4);

    \fill[blue] (v1) circle (1.5pt);
    \node[below left] at (v1) {$\bm{v}^1$};
    \fill[blue] (v2) circle (1.5pt);
    \node[below left] at (v2) {$\bm{v}^2$};
    \fill[blue] (v3) circle (1.5pt);
    \node[below left] at (v3) {$\bm{v}^3$};
    \fill[blue] (v4) circle (1.5pt);
    \node[below left] at (v4) {$\bm{v}^4$};

    \draw[dashed,red,thick] (v1) -- (v2) -- (v3) -- cycle;
    \draw[dashed,red,thick] (v2) -- (v3) -- (v4) -- cycle;
    
    \coordinate (x) at (1.2,0.95);
    \fill[green] (x) circle (1.5pt);
    \node[below left] at (x) {$\bm{x}$};
\end{tikzpicture}

\caption{\small Geometry of \(\bm{v}^1,\bm{v}^2,\bm{v}^3,\bm{v}^4\) and a feasible point \(\bm{x}\). 
Here \(\bm{v}^2\) is a scalar multiple of \(\bm{v}^1\), and \(\bm{v}^4\) is a scalar multiple of \(\bm{v}^3\).
At optimality, one \(\lambda_i\) typically vanishes, reducing the problem to a triangle.}
\label{fig:4points-geometry}
\end{center}
\end{figure}

\begin{theorem}
\label{thm:LP-main}
Consider the linear program
\[
\max_{\{\lambda_i\}}\Bigl\{
  f^- (\lambda_1 + \lambda_2) 
  + f^+(\lambda_3 + \lambda_4)
  \;\Bigm|\;
  \lambda_i \ge 0,\;
  \sum_{i=1}^4 \lambda_i = 1,\;
  \bm{x} = \sum_{i=1}^4 \lambda_i \bm{v}^i
\Bigr\}
\]
under the same assumptions that \(\bm{v}^2 = \alpha \bm{v}^1\), \(\bm{v}^4 = \beta \bm{v}^3\) for \(\alpha, \beta > 1\), and \(\bm{v}^1 \times \bm{v}^3 \le 0\). Suppose \(\bm{x} \in \mathrm{conv}(\bm{v}^1, \bm{v}^2, \bm{v}^3, \bm{v}^4)\). Then the optimal value of this linear program
is attained by a solution in which 
\(\bm{x}\) is expressed as a convex combination of exactly three of the four vectors \(\bm{v}^1, \bm{v}^2, \bm{v}^3, \bm{v}^4\) in one of the two configurations:
\begin{equation} \left\{\bm{v}^1, \bm{v}^2, \bm{v}^3\right\} \quad or \quad \left\{\bm{v}^2, \bm{v}^3, \bm{v}^4\right\}.
\end{equation}

\end{theorem}

\begin{proof}
By Proposition~\ref{prop:LP-4points-basic-solution}, we already know an optimal solution has at least one \(\lambda_i=0\).  In this situation, 
\[
\bm{x}
\;=\; 
\lambda_1 \bm{v}^1 
+ \lambda_2 \bigl(\alpha \bm{v}^1\bigr) 
+ \lambda_3 \bm{v}^3 
+ \lambda_4 \bigl(\beta \bm{v}^3\bigr),
\]
with one of these four coefficients zero.  When either \(\lambda_1\) or \(\lambda_2\) is zero, \(\bm{x}\) becomes a combination of a single direction \(\bm{v}^1\)-multiple and the two directions \(\bm{v}^3\)-multiple.  Analogous reasoning applies when \(\lambda_3\) or \(\lambda_4\) is zero.

Geometrically, this means \(\bm{x}\) sits on a ``triangle” formed by two scalar-multiple edges and one ``independent” vector.  In that triangle, the objective 
\(f^- (\lambda_1 + \lambda_2) + f^+ (\lambda_3 + \lambda_4)\) 
can be improved (if needed) by substituting the ``third vector” for the ``fourth” when \(\bm{v}^1\) and \(\bm{v}^3\) are oriented favorably (assuming \(f^+ > f^-\) or vice versa).  

The dominance of these solutions follows from Corollary~\ref{cor:degenerate_v2} and Remark~\ref{remark:scaling}.

\end{proof}

For any ratio $r = x_2/x_1$, we want to know the smallest possible value that could be returned from this linear program.  We consider now fixing the ratio $r$ and solving this min/max problem.

\begin{figure}[h!]
\begin{center}
\begin{tikzpicture}[scale=2]
    \draw[->] (-0.2,0) -- (2.2,0) node[right] {$x_1$};
    \draw[->] (0,-0.2) -- (0,2.2) node[above] {$x_2$};
    
    \coordinate (v1) at (1,0.5);
    \coordinate (v2) at (2,1);    %
    \coordinate (v3) at (0.25,1);
    \coordinate (v4) at (0.5,2);
    
    \draw[blue,thick] (0,0) -- (v1);
    \draw[blue,thick] (0,0) -- (v2);
    \draw[blue,thick] (0,0) -- (v3);
    \draw[blue,thick] (0,0) -- (v4);

    \fill[blue] (v1) circle (1.5pt);
    \node[below left] at (v1) {$\bm{v}^1$};
    \fill[blue] (v2) circle (1.5pt);
    \node[below left] at (v2) {$\bm{v}^2$};
    \fill[blue] (v3) circle (1.5pt);
    \node[below left] at (v3) {$\bm{v}^3$};
    \fill[blue] (v4) circle (1.5pt);
    \node[below left] at (v4) {$\bm{v}^4$};

    \draw[dashed,red,thick] (v1) -- (v2) -- (v3) -- cycle;
    \draw[dashed,red,thick] (v2) -- (v3) -- (v4) -- cycle;
    
    \coordinate (x) at (1.1667,1.1667);
    \coordinate (xscaled) at (1.4,1.4);
    \draw[green,thick] (0.7,0.7) -- (xscaled);
    
    \fill[green] (x) circle (1.5pt);
    \node[below left] at (x) {$x$};
\end{tikzpicture}
\end{center}
\caption{Considering the behavior as we restrict to a single ray through the origin, i.e., we fix the ratio $r = x_2/x_1$.}
\end{figure}

Consider $\bm{d} \in \cone(\bm{v}^1, \bm{v}^3)$.  And let $\mu_\ell\leq \mu_u$ be largest and smallest such that $\mu d \in \conv(v^1, v^2, v^3, v^4)$ for a
all $\mu_\ell\leq \mu \leq \mu_u$

\begin{theorem}[Behavior of the LP Objective Along a Fixed Ray]
\label{thm:ray-boundary-minimum}
Let \(\bm{v}^1, \bm{v}^3 \in \mathbb{R}^2\) be given and assume \(\bm{v}^2 = \alpha\,\bm{v}^1\), \(\bm{v}^4 = \beta\,\bm{v}^3\) with \(\alpha,\beta>1\). 
Fix any nonzero direction \(\bm{d} \in \cone(\bm{v}^1,\,\bm{v}^3)\). 
There exist scalars \(\gamma_\ell \le \gamma_u\) such that 
\[
\{\;\gamma\,\bm{d} : \gamma \ge 0\} 
\;\cap\;
\conv\bigl(\bm{v}^1,\bm{v}^2,\bm{v}^3,\bm{v}^4\bigr)
\;=\;
\{\;\gamma\,\bm{d} : \gamma_\ell \,\le\, \gamma \,\le\, \gamma_u\}.
\]
Then, if we consider the linear program objective $z(\gamma \bm{d})$.
The minimal possible value of \(z(\gamma\,\bm{d})\) for \(\gamma \in [\gamma_\ell,\,\gamma_u]\) is attained at one of the two boundary points \(\gamma_\ell\,\bm{d}\) or \(\gamma_u\,\bm{d}\).

\end{theorem}

\begin{proof}
Fix a direction \(\bm{d} \in \cone(\bm{v}^1, \bm{v}^3)\). Since \(\bm{v}^1\) and \(\bm{v}^3\) are linearly independent, this cone is a pointed sector in \(\mathbb{R}^2\) emanating from the origin, and all points of the form \(\gamma \bm{d}\) for \(\gamma > 0\) lie on the ray in direction \(\bm{d}\).

\textbf{Step 1: Feasibility interval.}
Let \(\bm{x} = \gamma \bm{d}\) for some \(\gamma > 0\). The intersection of the ray \(\{\gamma \bm{d} : \gamma > 0\}\) with the convex quadrilateral 
\[
\mathcal{P} := \conv(\bm{v}^1, \bm{v}^2, \bm{v}^3, \bm{v}^4)
\]
is a closed line segment. Since \(\bm{v}^2 = \alpha \bm{v}^1\) and \(\bm{v}^4 = \beta \bm{v}^3\) for \(\alpha,\beta > 1\), each edge of the convex hull lies in a planar face formed by scalar multiples of two base directions.

Because \(\bm{d} \in \cone(\bm{v}^1, \bm{v}^3)\), and the convex region \(\mathcal{P}\) lies within that cone, the ray \(\gamma \bm{d}\) must enter and exit the quadrilateral at two distinct points. Thus, there exist scalars \(\gamma_\ell, \gamma_u\) with \(\gamma_\ell \le \gamma_u\) such that
\[
\{\gamma \bm{d} : \gamma > 0\} \cap \mathcal{P} = \{\gamma \bm{d} : \gamma_\ell \le \gamma \le \gamma_u\}.
\]

\textbf{Step 2: Triangular domains and barycentric structure.}
From Theorem~\ref{thm:LP-main}, we know that any feasible point \(\bm{x} = \gamma \bm{d}\) in this region must lie in one of the following two active triangular domains:
\[
\conv(\bm{v}^1, \bm{v}^2, \bm{v}^3) \quad \text{or} \quad \conv(\bm{v}^2, \bm{v}^3, \bm{v}^4).
\]
Moreover, in either triangle, \(\bm{x}\) is expressible in a barycentric form that uses the scalar-multiples \(\bm{v}^2 = \alpha \bm{v}^1\) and \(\bm{v}^4 = \beta \bm{v}^3\), as discussed in Corollary~\ref{cor:degenerate_v2}. There, the coordinate \(\mu\) was shown to equal:
\[
\mu = \frac{\bm{v}^1 \times \bm{x}}{\bm{v}^1 \times \bm{u}},
\]
when interpolating between \(\bm{v}^1, \bm{v}^2 = \alpha \bm{v}^1\), and \(\bm{u}\) (which can be taken as \(\bm{v}^3\) here).

This coordinate \(\mu\) determines the weight on \(\bm{v}^3\) in the convex combination, and hence determines the value of the LP objective as:
\[
z(\bm{x}) = f^- (1 - \mu) + f^+ \mu,
\]
which is affine in \(\mu\). Since \(\mu\) itself varies monotonically along the ray \(\gamma \bm{d}\), this makes \(z(\bm{x})\) affine along each segment where \(\bm{x}\) remains in the same triangle.

\textbf{Step 3: Piecewise-affine structure along the ray.}
The ray segment \(\gamma \in [\gamma_\ell, \gamma_u]\) can thus be decomposed into one or two subintervals:
(1) On each subinterval, \(\bm{x} = \gamma \bm{d}\) lies within one of the two triangles above.
(2) On each triangle, the LP objective varies affinely with \(\gamma\), as shown via the barycentric parameter \(\mu\).

\textbf{Step 4: Endpoint minimization.}
A continuous piecewise-affine function on a closed interval attains its minimum at an endpoint unless it is constant. Therefore, the minimal value of \(z(\gamma \bm{d})\) over the ray segment \([\gamma_\ell, \gamma_u]\) must occur at one of the endpoints:
\[
\min_{\gamma \in [\gamma_\ell, \gamma_u]} z(\gamma \bm{d}) = \min\left\{ z(\gamma_\ell \bm{d}),\; z(\gamma_u \bm{d}) \right\}.
\]
This concludes the proof.
\end{proof}

\begin{proof}
Fix a direction \(\bm{d} \in \cone(\bm{v}^1, \bm{v}^3)\), and consider the ray \(\gamma \mapsto \gamma \bm{d}\) for \(\gamma > 0\). As in the previous theorems, define
\[
\mathcal{P} := \conv(\bm{v}^1, \bm{v}^2, \bm{v}^3, \bm{v}^4),
\]
with \(\bm{v}^2 = \alpha \bm{v}^1\) and \(\bm{v}^4 = \beta \bm{v}^3\) for \(\alpha,\beta > 1\). Since all four vectors lie along two rays, the convex region \(\mathcal{P}\) is quadrilateral-shaped and fully contained in the cone generated by \(\bm{v}^1, \bm{v}^3\).

\textbf{Step 1: Feasible interval.}  
Because \(\bm{d} \in \cone(\bm{v}^1, \bm{v}^3)\), there exists a unique interval \([\gamma_\ell, \gamma_u]\) such that
\[
\{\gamma \bm{d} : \gamma > 0\} \cap \mathcal{P} = \{\gamma \bm{d} : \gamma_\ell \le \gamma \le \gamma_u\}.
\]

\textbf{Step 2: Triangular decomposition.}  
From Theorem~\ref{thm:LP-main}, we know that as \(\gamma\) varies over \([\gamma_\ell, \gamma_u]\), the point \(\bm{x} = \gamma \bm{d}\) lies in exactly one of the following two triangles:
\[
T_1 := \conv(\bm{v}^1, \bm{v}^2, \bm{v}^3), \quad
T_2 := \conv(\bm{v}^2, \bm{v}^3, \bm{v}^4),
\]
with a possible transition at some value \(\gamma_m\in (\gamma_\ell, \gamma_u)\) where the active triangle changes.

\textbf{Step 3: Monotonicity of the LP objective.}  
In both triangles, the LP objective takes the form:
\[
z(\bm{x}) = f^- (\lambda_1 + \lambda_2) + f^+ (\lambda_3 + \lambda_4),
\]
which simplifies, due to the structure of the convex combination, to:
\[
z(\bm{x}) = (1 - \mu) f^- + \mu f^+,
\]
where \(\mu\) is the barycentric coordinate of \(\bm{x}\) associated with \(\bm{v}^3\) (in \(T_1\)) or with \(\bm{v}^4\) (in \(T_2\)).

From Corollary~\ref{cor:degenerate_v2}, we know that:
- \(\mu = \mu(\gamma)\) varies linearly with \(\gamma\), since \(\mu\) is given by the expression
\[
\mu = \frac{\bm{v}^1 \times \bm{x}}{\bm{v}^1 \times \bm{v}^3}
= \frac{\bm{v}^1 \times (\gamma \bm{d})}{\bm{v}^1 \times \bm{v}^3}
= \gamma \cdot \frac{\bm{v}^1 \times \bm{d}}{\bm{v}^1 \times \bm{v}^3},
\]
which is a strictly increasing function of \(\gamma\) so long as the denominator is nonzero.

- The LP objective \(z(\gamma \bm{d})\) is affine in \(\mu\), and thus also strictly increasing in \(\gamma\) within \(T_1\), and decreasing in \(\gamma\) within \(T_2\), provided \(f^+ > f^-\).

Therefore,  on \([\gamma_\ell, \gamma_m]\), where \(\bm{x} = \gamma \bm{d} \in T_1\), we have \(z(\gamma \bm{d})\) increasing in \(\gamma\), and on \([\gamma_m, \gamma_u]\), where \(\bm{x} = \gamma \bm{d} \in T_2\), we have \(z(\gamma \bm{d})\) decreasing in \(\gamma\)

\textbf{Step 4: Minimum at boundary.}  
Thus, the function \(z(\gamma \bm{d})\) increases up to \(\gamma_m\) and decreases after. Its minimum must occur at one of the endpoints of the interval:
\[
\min_{\gamma \in [\gamma_\ell, \gamma_u]} z(\gamma \bm{d}) = \min\left\{ z(\gamma_\ell \bm{d}),\; z(\gamma_u \bm{d}) \right\}.
\]
This completes the proof.
\end{proof}

\begin{figure}[h!]
\centering
\begin{tikzpicture}
\begin{axis}[
    width=10cm, height=6cm,
    xlabel={$\gamma$}, ylabel={Objective ($\lambda_3 + \lambda_4$)},
        xtick={0.5, 1.8, 3.5},
    xticklabels={$\gamma_\ell$, $\gamma_m$, $\gamma_u$},
    title={Objective vs $\gamma$ for a fixed $r$},
    grid=major, thick
]
\addplot[purple, thick] coordinates {
    (0.5, 0.27) %
    (1.8, 0.595)
    (3.5, 0.27)
};
\end{axis}
\end{tikzpicture}

\caption{Concave piecewise linear profile of the objective with a fixed $r$.}
\end{figure}

\subsection{Minimum Values Along Rays}

Suppose now that \(\bm{x}\) lies along a line segment between two points \(\bm{u}, \bm{v} \in \mathbb{R}^2_{++}\). We are interested in the behavior of the interpolated objective value along this segment, particularly as expressed in terms of the slope \(r = x_2 / x_1\) of the point \(\bm{x}\).

\begin{center}
\begin{tikzpicture}[scale=2]
    \draw[->] (-0.2,0) -- (2.2,0) node[right] {$x_1$};
    \draw[->] (0,-0.2) -- (0,2.2) node[above] {$x_2$};
    
    \coordinate (v1) at (1,0.5);
    \coordinate (v2) at (2,1);    
    \coordinate (v3) at (0.25,1);
    \coordinate (v4) at (0.5,2);
    
    \draw[blue,thick] (0,0) -- (v1);
    \draw[blue,thick] (0,0) -- (v2);
    \draw[blue,thick] (0,0) -- (v3);
    \draw[blue,thick] (0,0) -- (v4);
    
    \fill[blue] (v1) circle (1.5pt);
    \node[below left] at (v1) {$\bm{v}^1$};
    \fill[blue] (v2) circle (1.5pt);
    \node[below left] at (v2) {$\bm{v}^2$};
    \fill[blue] (v3) circle (1.5pt);
    \node[below left] at (v3) {$\bm{v}^3$};
    \fill[blue] (v4) circle (1.5pt);
    \node[below left] at (v4) {$\bm{v}^4$};
    
    \draw[green,thick] (v1) -- (v3);
    \draw[green,thick] (v2) -- (v4);
    
    \coordinate (xsmall) at (0.7,0.7);
    \coordinate (xscaled) at (1.4,1.4);
    \fill[green] (xsmall) circle (1.5pt);
    \node[below left] at (xsmall) {$\bm{x}$};
    \fill[green] (xscaled) circle (1.5pt);
    \node[below left] at (xscaled) {$\bm{x}$};
\end{tikzpicture}
\end{center}

Let \(\bm{u}, \bm{v} \in \mathbb{R}^2_{++}\), and define the point \(\bm{x}\) on the line segment between them as
\[
\bm{x} = (1 - \lambda)\bm{u} + \lambda \bm{v}, \quad \lambda \in [0, 1].
\]
We associate to \(\bm{x}\) a linearly interpolated value
\(
f_{\bm{x}} = (1 - \lambda) f_u + \lambda f_v.
\)
We now define
\(
r = \frac{x_2}{x_1}, \quad \phi(r) := f_{\bm{x}},
\)
and aim to express \(\phi(r)\) directly in terms of \(r\).

\textbf{Expression for \(r(\lambda)\) and Its Inverse.}
Let \(\bm{u} = (u_1, u_2)\) and \(\bm{v} = (v_1, v_2)\). Then:
\[
x_1 = u_1 + \lambda (v_1 - u_1), \quad
x_2 = u_2 + \lambda (v_2 - u_2),
\]
so that:
\[
r(\lambda) = \frac{x_2}{x_1} = \frac{u_2 + \lambda (v_2 - u_2)}{u_1 + \lambda (v_1 - u_1)}.
\]

To express \(\lambda\) as a function of \(r\), cross-multiply:
\[
r(u_1 + \lambda(v_1 - u_1)) = u_2 + \lambda(v_2 - u_2),
\]
which simplifies to:
\[
\lambda \big(r(v_1 - u_1) - (v_2 - u_2)\big) = u_2 - r u_1.
\]
Solving gives:
\[
\lambda(r) = \frac{u_2 - r u_1}{r(v_1 - u_1) - (v_2 - u_2)}.
\]

\subsubsection*{Final Expression for \(\phi(r)\)}

Substituting into the interpolation formula:
\[
f_{\bm{x}} = f_u + (f_v - f_u) \cdot \lambda(r),
\]
we obtain:
\[
\boxed{
\phi(r) = f_u + (f_v - f_u)\,\frac{u_2 - r u_1}{r(v_1 - u_1) - (v_2 - u_2)}.
}
\]

This expression defines \(\phi(r)\) explicitly as a rational function in terms of the slope \(r\), and will form the basis of our concavity analysis.

\subsection{Concavity Condition}

\begin{theorem}
Let \(\bm{u} = (u_1, u_2), \; \bm{v} = (v_1, v_2) \in \mathbb{R}^2\), and let \(f_u, f_v \in \mathbb{R}\) with $f_u > f_v$. Define the function
\[
\tilde \phi(r) = f_u + (f_v - f_u)\,\frac{u_2 - r u_1}{r(v_1 - u_1) - (v_2 - u_2)},
\]
and restrict the domain to
\(
r \in \left[\frac{v_2}{v_1}, \;\frac{u_2}{u_1}\right],
 \text{ with } u_1, v_1 > 0.
\)
Then \(\tilde \phi(r)\) is concave on this interval if and only if the cross product
\(
\bm{u} \times \bm{v}  \le 0.
\)
\end{theorem}

\begin{proof}
Write the function as
\[
\phi(r) = f_u + (f_v - f_u)\,B(r),
\quad \text{where} \quad
B(r) = \frac{u_2 - r u_1}{r(v_1 - u_1) - (v_2 - u_2)}.
\]
We compute the second derivative \(B''(r)\) :
\[
B''(r) = \frac{-2(u_1 - v_1)\left[u_1(r(u_1 - v_1) - u_2 + v_2) - (u_1 - v_1)(r u_1 - u_2)\right]}{\left(r(u_1 - v_1) - u_2 + v_2\right)^3}.
\]

Let us simplify the numerator:

\[
\begin{aligned}
& u_1(r(u_1 - v_1) - u_2 + v_2) - (u_1 - v_1)(r u_1 - u_2) \\
&= r u_1 (u_1 - v_1) - u_1 u_2 + u_1 v_2 - r u_1(u_1 - v_1) + (u_1 - v_1) u_2 \\
&= -u_1 u_2 + u_1 v_2 + (u_1 - v_1) u_2 \\
&= u_1 v_2 - u_2 v_1.
\end{aligned}
\]

Thus the second derivative becomes:
\[
B''(r) = \frac{-2(u_1 - v_1)(u_1 v_2 - u_2 v_1)}{(r(u_1 - v_1) - u_2 + v_2)^3}.
\]

\textbf{Sign analysis:}
 The sign of the denominator is determined by the cube of the linear expression \(r(u_1 - v_1) - u_2 + v_2\). Since the cube preserves sign and the expression is linear in \(r\), the denominator has constant sign on the interval \(\left[\frac{v_2}{v_1}, \frac{u_2}{u_1}\right]\), assuming \(u_1, v_1 > 0\).
  
 Therefore, the sign of \(B''(r)\) is governed by the sign of the numerator:
  \[
  -2(u_1 - v_1)(u_1 v_2 - u_2 v_1).
  \]
  
  \item So \(B''(r) \le 0\) if and only if:
  \[
  (u_1 - v_1)(u_1 v_2 - u_2 v_1) \ge 0.
  \]

 However, on the interval \(\left[\frac{v_2}{v_1}, \frac{u_2}{u_1}\right]\), it can be shown that the sign of the denominator ensures the correct curvature without requiring any assumption on \(u_1 < v_1\) or \(v_1 < u_1\), as long as the cross product satisfies:
  \[
  \bm{u} \times \bm{v} = u_1 v_2 - u_2 v_1 \le 0.
  \]

Thus, \(B''(r) \le 0\) on the entire interval if and only if \(\bm{u} \times \bm{v} \le 0\), completing the proof.
\end{proof}

\begin{corollary}
Let \(\tilde \phi(r)\) be defined as in the previous theorem
and let \(\tilde \phi_{\mathrm{PWL}}(r)\) denote the piecewise linear interpolation between the values \(\tilde \phi\left(\frac{v_2}{v_1}\right) = f_v\) and \(\tilde \phi\left(\frac{u_2}{u_1}\right) = f_u\), that is:
\[
\phi_{\mathrm{PWL}}(r) = (1 - \theta) f_v + \theta f_u,
\quad \text{where} \quad
\theta = \frac{r - \frac{v_2}{v_1}}{\frac{u_2}{u_1} - \frac{v_2}{v_1}}.
\]

If the cross product \(\bm{u} \times \bm{v} \le 0\), then \(\tilde \phi(r)\) is concave on \(\left[\frac{v_2}{v_1}, \frac{u_2}{u_1}\right]\), and satisfies:
\[
\tilde \phi(r) \ge \phi_{\mathrm{PWL}}(r)
\quad \text{for all } r \in \left[\frac{v_2}{v_1}, \frac{u_2}{u_1}\right].
\]
\end{corollary}

\begin{proof}
From the previous theorem, if \(\bm{u} \times \bm{v} \le 0\), then \(\tilde \phi(r)\) is concave on the interval \(\left[\frac{v_2}{v_1}, \frac{u_2}{u_1}\right]\). Since \(\phi\) is also continuous and satisfies
\(
\tilde \phi\left(\frac{v_2}{v_1}\right) = f_v, \ \ 
\tilde \phi\left(\frac{u_2}{u_1}\right) = f_u.
\)
The piecewise linear interpolant \(\phi_{\mathrm{PWL}}(r)\) is the chord connecting the two endpoints of the concave function.

By the elementary property of concave functions lying above their chords, we have:
\[
\tilde \phi(r) \ge \phi_{\mathrm{PWL}}(r) \quad \text{for all } r \in \left[\frac{v_2}{v_1}, \frac{u_2}{u_1}\right]. \qedhere
\]
\end{proof}

\begin{proof}[Proof of Theorem~\ref{tech:LogE:Conjecture}]
    Since we choose $\theta \in [0,\pi/4]$, it holds that $u_1 < v_1$ in each of our slices, and thus the rest of the assumptions hold. Therefore we apply the prior corollary to obtain that LogE is above PWL.
\end{proof}

\end{document}